\documentclass[a4paper, reqno, 12pt]{amsart}

\usepackage[utf8]{inputenc}
\usepackage{amsthm,amsfonts,amssymb,amsmath,amsxtra,amsrefs}
\usepackage{mathtools}
 \usepackage{bbold}
\usepackage{latexsym}
 \usepackage{stmaryrd}
\usepackage{graphicx}
\usepackage{tikz}
\usepackage{tikz-cd}
\usepackage{todonotes,cancel}
\usepackage[margin=1.25in]{geometry}
\usepackage{mathrsfs}
\usepackage{bm}

\usepackage[all]{xy}
\SelectTips{cm}{}
\usepackage{xr-hyper}
\usepackage[colorlinks=false,
   citecolor=Black,
   linkcolor=Red,
   urlcolor=Blue]{hyperref}
\usepackage{verbatim}
\RequirePackage{xspace}
\RequirePackage{etoolbox}
\RequirePackage{varwidth}
\RequirePackage{enumitem}
\RequirePackage{tensor}
\RequirePackage{mathtools}
\RequirePackage{longtable}
\RequirePackage{multirow}

\usepackage{rotating}
\newcommand*{\isoarrow}[1]{\arrow[#1,"\rotatebox{90}{\(\sim\)}"
]}

\tikzset{
    labl/.style={anchor=south, rotate=90, inner sep=.5mm}
}

\newtheorem{thm}{Theorem}
\newtheorem{theorem}[thm]{Theorem}
\newtheorem{thmintro}{Theorem}
\newtheorem{propintro}[thmintro]{Proposition}
\newtheorem{prop}[thm]{Proposition}
\newtheorem{proposition}[thm]{Proposition}
\newtheorem{lem}[thm]{Lemma}
\newtheorem{lemma}[thm]{Lemma}

\newtheorem{cor}[thm]{Corollary}
\newtheorem{corollary}[thm]{Corollary}

\theoremstyle{definition}
\newtheorem{definition}[thm]{Definition}
\newtheorem{defi}[thm]{Definition}
\newtheorem{example}[thm]{Example}

\newtheorem{remark}[thm]{Remark}
\newtheorem{remarks}[thm]{Remarks}

\numberwithin{equation}{section}
\numberwithin{thm}{section}

\newcommand{\qbinom}[2]{\genfrac{[}{]}{0pt}{0}{#1}{#2}}

\newcommand{\I}{I}
\newcommand{\coroot}{\alpha^\vee}
\newcommand{\A}{\mathcal{A}}
\newcommand{\U}{\mathrm{U}}
\newcommand{\fkU}{\mathfrak{U}}
\newcommand{\AU}{{_\mathcal{A}\dot{\U}}}
\newcommand{\Ui}{\U^\imath}
\newcommand{\ev}{\bar{0}}
\newcommand{\odd}{\bar{1}}
\newcommand{\one}{\mathbf{1}}
\newcommand{\fr}{{\textbf{Fr}'}}
\newcommand{\ofr}{\textbf{Fr}}
\newcommand{\pdotU}{{_{\A'}\dot{\U}}}
\newcommand{\cdotU}{{_{\A'}\dot{\fkU}}}
\newcommand{\ifr}{{\textbf{Fr}'_\imath}}
\newcommand{\iofr}{{\ofr_\imath}}
\newcommand{\fpdotU}{{_{\mathbf{F}}\dot{\U}}}

\newcommand{\udoti}{{}_k\dot{\U}^\imath}
\newcommand{\bq}{\mathbf{q}}
\newcommand{\B}{\mathrm{B}}
\newcommand{\LR}[2]{\left\llbracket \begin{matrix} #1\\#2 \end{matrix} \right\rrbracket}
\def\lr#1#2{\ensuremath{\left(\kern-.3em\left(\genfrac{}{}{0pt}{}{#1}{#2}\right)\kern-.3em\right)}}
\newcommand{\baV}{{V_{\BF_p}}}
\newcommand{\mO}{\mathcal{O}}
\newcommand{\dist}{\text{Dist}}
\newcommand{\lie}{\text{Lie}}

\newcommand{\Qq}{\mathbb{Q}(q)}

\newcommand{\WD}{\widehat{\Delta}}

\newcommand{\G}{\mathbf{G}}
\newcommand{\mU}{\mathfrak{U}}

\newcommand{\cAUi}{{_{\mathcal{A}'}\dot{\mathfrak{U}}^\imath}}
\newcommand{\pAUi}{{_{\mathcal{A}'}\dot{\mathrm{U}}^\imath}}
\newcommand{\cAtUi}{{_{\mathcal{A}'_2}\dot{\mathfrak{U}}^\imath}}
\newcommand{\pAtUi}{{_{\mathcal{A}'_2}\dot{\mathrm{U}}^\imath}}

\newcommand{\BC}{\ensuremath{\mathbb {C}}\xspace}

\newcommand{\BF}{\ensuremath{\mathbb {F}}\xspace}
\newcommand{{\BG}}{\ensuremath{\mathbb {G}}\xspace}

\newcommand{{\BK}}{\ensuremath{\mathbb {K}}\xspace}

\newcommand{\BN}{\ensuremath{\mathbb {N}}\xspace}

\newcommand{\BP}{\ensuremath{\mathbb {P}}\xspace}
\newcommand{\BQ}{\ensuremath{\mathbb {Q}}\xspace}

\newcommand{\BZ}{\ensuremath{\mathbb {Z}}\xspace}

\newcommand{\CA}{\ensuremath{\mathcal {A}}\xspace}
\newcommand{\CB}{\ensuremath{\mathcal {B}}\xspace}

\newcommand{\CO}{\ensuremath{\mathcal {O}}\xspace}

\newcommand{\CV}{\ensuremath{\mathcal {V}}\xspace}

\newcommand{\CX}{\ensuremath{\mathcal {X}}\xspace}
\newcommand{\CY}{\ensuremath{\mathcal {Y}}\xspace}
\newcommand{\CZ}{\ensuremath{\mathcal {Z}}\xspace}

\newcommand{\ff}{\ensuremath{\mathfrak {f}}\xspace}
\newcommand{\fe}{\ensuremath{\mathfrak {e}}\xspace}
\newcommand{\fb}{\ensuremath{\mathfrak {b}}\xspace}

\newcommand{\RF}{\ensuremath{\mathrm {F}}\xspace}
\newcommand{\RE}{\ensuremath{\mathrm {E}}\xspace}
\newcommand{\RB}{\ensuremath{\mathrm {B}}\xspace}
\newcommand{\RO}{\ensuremath{\mathrm {O}}\xspace}

\begin{document}

\title[]{Symmetric subgroup schemes, Frobenius splittings, and quantum symmetric pairs}

\author[Huanchen Bao]{Huanchen Bao}
\address{Department of Mathematics, National University of Singapore, Singapore.}
\email{huanchen@nus.edu.sg}

\author[Jinfeng Song]{Jinfeng Song}
\address{Department of Mathematics, National University of Singapore, Singapore.}
\email{j\_song@u.nus.edu}


\begin{abstract}
Let $G_k$ be a connected reductive algebraic group over an algebraically closed field $k$ of characteristic $\neq 2$. Let $K_k \subset G_k$ be a quasi-split symmetric subgroup of $G_k$ with respect to an involution $\theta_k$ of $G_k$. The classification of such involutions is independent of the characteristic of $k$ (provided not $2$).

We first construct a closed subgroup scheme $\G^\imath$ of the Chevalley group scheme $\G$ over $\BZ$. The pair $(\G, \G^\imath)$ parameterizes   symmetric pairs of the given type over any algebraically closed field of characteristic $\neq 2$, that is, the geometric fibre of $\G^\imath$ becomes the reductive group $K_k \subset G_k$ over any algebraically closed field $k$ of characteristic $\neq 2$. 
As a consequence, we show the coordinate ring of the group $K_k$ is spanned by the dual $\imath$canonical basis of the corresponding $\imath$quantum group.

We then construct a quantum Frobenius splitting for the quasi-split $\imath$quantum group at roots of $1$. This generalizes Lusztig's quantum Frobenius splitting for quantum groups at roots of $1$. Over a field of positive characteristic, our quantum Frobenius splitting induces a Frobenius splitting of the algebraic group $K_k$.

Finally, we construct Frobenius splittings of the flag variety $G_k / B_k$ that compatibly split certain $K_k$-orbit closures over positive characteristics.  We deduce cohomological vanishings of line bundles as well as normalities. Results apply to characteristic $0$ as well, thanks to the existence of the scheme $\G^\imath$. Our construction of splittings is based on the quantum Frobenius splitting of the corresponding $\imath$quantum group. %
\end{abstract}

	\maketitle
	
	\tableofcontents

	\section{Introduction}
	The main part of this introduction is divided into three subsections discussing three main topics considered in this paper: the symmetric subgroup $K_k$ of a reductive group $G_k$; the quantum Frobenius splitting of an $\imath$quantum group; the Frobenius splitting of $K_k$-orbit closures on the flag variety $\CB_k$ of $G_k$. 
	 \subsection{The symmetric subgroup $K_k$}
	 
\subsubsection{Quantum groups and coordinate rings}
		Let $G_{\mathbb{C}}$ be a connected reductive algebraic group  over $\BC$. We denote by $\RO_\BC$ the coordinate ring of $G_{\mathbb{C}}$.  In his famous papers  \cites{Ch55, Ch95}, Chevalley constructed an integral form $\RO_\BZ$ of $\RO_\BC$ such that $\RO_\BC = \RO_\BZ \otimes_{\BZ} \BC$.   The integral form defines the {\it Chevalley group scheme}  $\G=\G_\BZ$  over $\BZ$ such that the geometric fiber $G_k\footnote{We identify an algebraic variety over $k$ with its set of $k$-rational points.} = \G_{\BZ} \times_{{Sp\, \BZ}} {Sp\, } k$ is the linear algebraic group associated with the given root datum for any algebraically closed field $k$.

Chevalley's approach depends on the choice of representations of $G_{\mathbb{C}}$. Kostant \cite{Ko66} identified (without proof) $\RO_\BZ$ intrinsically as the restricted dual of the Kostant's $\BZ$-form of the enveloping algebra of the Lie algebra of $G_{\mathbb{C}}$. Lusztig \cite{Lu07} reformulated (and proved) Kostant's construction using his theory of canonical bases on the quantum group $\U$. The coordinate ring $\RO_\BZ$ is identified as a dual subspace of the modified quantum group (at $q=1$) spanned by the dual canonical basis.

	\subsubsection{The symmetric subgroups} Let $G_k$ be a connected reductive algebraic group over an algebraically closed field $k$ of characteristic $\neq 2$. Let $\theta_k$ be an involution of $G_k$.  We denote the fixed point subgroup by $K_k$. It was shown by Steinberg \cite{Ste68} that $K_k$ is reductive. We remark that $K_k$ may not be connected.
	
		It was shown by Springer \cite{Spr87} that involutions of $G_k$ are classified in terms of Satake diagrams (when $G_k$ is simple) independent of the characteristic of $k$, provided $\neq 2$. We reformulate Springer's classification using {\em $\imath$root datum} in \S\ref{sec:class}. 
		
	We assume $K_k$ or $(G_k, \theta_k)$ is of quasi-split type, that is, there exists a Borel subgroup $B_k\subset G_k$ such that $B_k \cap \theta_k(B_k) =T_k$ is a maximal torus of $G_k$. Such a Borel subgroup $B_k$ is called $\theta_k$-anisotropic. In terms of Satake diagrams, this means no black dots in the Satake diagram. We also call the relevant $\imath$root datum quasi-split in this case.

	We only consider quasi-split cases in this paper. We fix a quasi-split $\imath$root datum for the rest of this introduction. Hence, for any algebraically closed field $k$ of characteristic $\neq 2$, we obtain $G_k$, $\theta_k$, $K_k$, $\theta_k$-anisotropic $B_k$. 

	\subsubsection{Quantum symmetric pairs}
		Associated to the $\imath$root datum, we can construct a quantum symmetric pair $(\U,\U^\imath)$\footnote{The superscript $\imath$ stands for invariant or involution. Since $\U^\imath$ is not the fixed point of any natural involution $\theta$ of $\U$ for generic $q$, one can not denote it by $\U^\theta$.}, where $\U^\imath \subset \U$ is a coideal subalgebra of the quantum group associated to the root datum of $G_k$. The quantum symmetric pair $(\U,\U^\imath)$ is a quantization of the pair of enveloping algebras of the symmetric pair $(\mathfrak{g}_k, \mathfrak{g}_k^{\theta_k})$. Here $\mathfrak{g}_k$ denotes the Lie algebra of $G_k$, and $\mathfrak{g}_k^{\theta_k}$ denotes the Lie algebra of $K_k$ (cf. \cite{Bo91}*{\S9.4}).  We often call $\U^\imath$ the {\em $\imath$quantum group}.
		
		Quantum symmetric pairs were originally introduced by Letzter \cite{Le99}, generalized by Kolb \cite{Ko14} to Kac-Moody cases.  The first named author and Wang \cite{BW18} initiated the theory of canonical bases arising from quantum symmetric pairs. We refer to the survey \cite{WangICM} for recent developments in quantum symmetric pairs. 
  
  Let ${}_\A\dot{\U}^\imath$ be the $\A = \BZ[q,q^{-1}]$-form of the modified $\imath$quantum group. This is a free $\A$-subalgebra of the modified $\imath$quantum group $\dot{\U}^\imath$ with basis $\dot{\B}^\imath$. We call $\dot{\B}^\imath$ the $\imath$canonical basis of ${}_\A\dot{\U}^\imath$. 
		
		For any commutative ring $R$ and ring homomorphism $\A \rightarrow R$, we write ${}_R\dot{\U}^\imath=  R \otimes_{\A} {}_\A\dot{\U}^\imath$ if there is no ambiguity. We abuse notations and denote by  $\dot{\B}^\imath$ the basis of ${}_R\dot{\U}^\imath$ after base change.

	\subsubsection{Symmetric subgroup schemes} 
	
	We consider the ring homomorphism $\A \rightarrow \BZ$ mapping $q$ to $1$ and the algebra ${}_\BZ\dot{\U}^\imath=  \BZ \otimes_{\A} {}_\A\dot{\U}^\imath$.  Let $\RO^\imath_{\BZ}$ be the subspace of $\text{Hom}_{ {\BZ}}( {}_\BZ\dot{{\U}}^\imath,  \BZ)$ spanned by the dual basis of $\dot{\B}^\imath$.  We state the first main theorem of this paper. 
	
	\begin{thmintro}[Theorem~\ref{thm:Hopfi} \& Proposition~\ref{prop:GAi} \& \S\ref{sec:pro} (a) \& Theorem~\ref{thm:Oik}]\label{thm:1}  \phantom{x}
	
	\begin{enumerate}
	\item  The subspace $\RO^\imath_{{\BZ}}$ is naturally a commutative and cocommutative Hopf algebra. 
    \item The algebra $\RO^\imath_{{\BZ}}$ defines a closed subgroup scheme of the Chevalley group scheme $\G$, denoted by $\G^\imath \subset \G$.  
    \item  Assume $R$ is an integral domain with characteristic not 2. Then $\RO^\imath_{R} = \RO^\imath_{{\BZ}} \otimes_\BZ R$ is reduced.
    \item  Let $k$ be any algebraically closed field $k$ of characteristic $\neq 2$. We identify $\G^\imath_k$ with its set of $k$-rational points, denoted by $G^\imath_k$. 
	We have $G^\imath_k = K_k \subset G_k$. In particular, $\RO^\imath_{k}$ is the coordinate ring of $K_k$.
	\end{enumerate}
	
	\end{thmintro}
  We call $\G^\imath$ the {\em symmetric subgroup scheme} over $\BZ$. 
   The theorem generalizes Lusztig's construction \cite{Lu07} of the Chevalley group scheme $\G$ using (dual) canonical bases arising from quantum groups. While $\G$ parameterizes reductive groups associated to a given root datum, the pair $(\G, \G^\imath)$ parameterizes symmetric pairs associated to a given $\imath$root datum.

\subsubsection{Future works}
We expect to generalize Theorem~\ref{thm:1} to arbitrary (quantum) symmetric pairs in our future works. Our current results rely on the (strong) compatibility of $\imath$canonical bases with respect to the projective system introduced by the first named author and Wang in  \cite{BW18a}*{\S6.2} (see Theorem~\ref{thm:stab}). Such compatibility was conjectured by the first named author and Wang in \cite{BW18a}*{Remark~6.18} and established by Watanabe \cite{Wa21}  for quasi-split cases (as well as all real rank one cases in \cite{Wa22}) using the crystal theory for $\imath$quantum groups. 
While the crystal theory for $\imath$quantum groups is an interesting research direction on its own, only a weaker version is required in our project. 

We expect that $\G^\imath$ is the fixed point subscheme of $\G$ (cf. Fogarty \cite{Fo73}). It would be interesting to see the connection with the recent preprint \cite{ALRR22} by Achar-Louren{\c c}o-Richarz-Riche.

	\subsection{Quantum Frobenius splittings}
 Let $k$ be an algebraically closed field of characteristic $p > 2$.
		\subsubsection{Definitions} 
		The concept of Frobenius splitting was introduced by Mehta-Ramanathan \cite{MR85}, followed by Ramanan-Ramanathan \cite{RR85}, in their study of Schubert varieties.  Mehta-Ramanathan \cite{MR85} and Ramanan-Ramanathan \cite{RR85} used Frobenius splitting method to prove the vanishing of higher cohomologies of ample line bundles on Schubert variety, as well as the normality of Schubert varieties (which implies the Demazure character formula). Such results apply to characteristic $0$ as well by the semi-continuity theorem \cite{Ha77}*{Chapter III, Theorem 12.8}.  
		
		A variety $\mathcal{X}$ over $k$ is called Frobenius split if the natural morphism $\mathcal{O}_\mathcal{X} \rightarrow F_* \mathcal{O}_\mathcal{X}$ admits a splitting. Here $F: \mathcal{X} \rightarrow \mathcal{X}$ denotes the absolute Frobenius morphism, which fixes all the points and raises the functions to their $p$-th powers.
		
		Let $\mathcal{Y} \subset \mathcal{X}$ be a closed subvariety of $\mathcal{X}$. A splitting of $F: \mathcal{X} \rightarrow \mathcal{X}$ is said to compatibly split $\mathcal{Y}$ if it induces a splitting of $F: \mathcal{Y} \rightarrow \mathcal{Y}$. This can be equivalently defined as a splitting preserving the ideal sheaf of $\mathcal{Y}$ as $\mathcal{O}_\mathcal{X}$-modules. 
	
	We refer to the book \cite{BK05} by Kumar-Brion for details on Frobenius splittings. 
		 
	\subsubsection{Quantum Frobenius splittings of $\U$} We assume $p > 2$ (and $p > 3$ if $G_k$ has a $G_2$ component).  Let $\A'$ be the quotient of $\A$ by the $p$-th cyclotomic polynomial, that is, $q$ is a $p$-th primitive root of $1$ in $\A'$.  We are interested in two ring homomorphisms from $\A$ to $\A'$, mapping $q$ to $1$ and $p$-th root of $1$, respectively.   We then obtain the two $\A'$-algebras ${}_{\A'} \dot{\mathfrak{U}}$ (mapping $q$ to $1$) and ${}_{\A'} \dot{\U}$ (mapping $q$ to $p$-th root of $1$) via base changes from the $\A$-form ${}_{\A} \dot{\U}$ of the modified quantum group. 		 


	Lusztig \cite{Lu90} defined a quantum Frobenius morphism $\ofr :{}_{\A'} \dot{\U} \rightarrow  {}_{\A'} \dot{\mathfrak{U}}$ mapping divided powers  $E^{(n)}_i 1_\zeta $  of Chevalley generators to $E^{(n/p)}_i 1_{\zeta/p}$ (and $0$ if not divisible). Lusztig \cite{Lu93} also defined a (more mysterious) splitting $\fr: {}_{\A'} \dot{\mathfrak{U}} \rightarrow {}_{\A'} \dot{\U}$ of the quantum Frobenius morphism, such that $\ofr \circ \fr = {\rm id}$.
	
	\subsubsection{A Frobenius splitting of $G_k$} Note that ${}_{k}\dot{\U} = k \otimes_{\A'}{}_{\A'} \dot{\mathfrak{U}} = k \otimes_{\A'}{}_{\A'}\dot{\U}$, when $k$ is a field of characteristic $p$. 
	
	
	By taking duals, and combining with the Frobenius automorphis of $k$ (see \S\ref{sec:kFr} for more details), both ${}_k\ofr^*$ and ${}_k\fr^{, *}$ become additive endomorphisms of the coordinate ring $\RO_k$ of $G_k$.  The following theorem seems never appear explicitly in literature, which we would like to highlight. 
	\begin{thmintro}[Theorem~\ref{thm:splitO}] \label{introthm:2}
		The induced morphims ${}_k\ofr^{*}: \RO_k \rightarrow \RO_k$ is  the $p$-th power map. The induced map ${}_k\fr^{, *}:\RO_k \rightarrow \RO_k$ gives a Frobenius splitting of the algebraic group $G_k$.
	\end{thmintro}
	
	Since $G_k$ is a smooth affine variety, it is always Frobenius split.  The quantum Frobenius splitting ${}_k\fr^{, *}$ induces a concrete splitting on $G_k$. By construction, the splitting compatibly splits the chosen Borel subgroup $B_k$ and the maximal torus $T_k \subset B_k$. 
	
	\subsubsection{Quantum Frobenius splittings of $\U^\imath$}We now consider the two $\A'$-algebras ${}_{\A'} \dot{\mathfrak{U}}^\imath$ and ${}_{\A'} \dot{\U}^\imath$ associated to the $\imath$quantum group. They are obtained via base changes from ${}_{\A} \dot{\U}^\imath$ and ring homomorphisms $\A \rightarrow \A'$, where $q$ maps $1$ for ${}_{\A'} \dot{\mathfrak{U}}^\imath$ and $p$-th root of $1$ for ${}_{\A'} \dot{\U}^\imath$. Here we use different notations to distinguish the two $\A'$-algebras when considered together. 
 
 It was shown by the first named author and Sale \cite{BS21} that Lusztig's quantum Frobenius morphism restricts to a quantum Frobenius morphism $\iofr: {}_{\A'} \dot{\U}^\imath \rightarrow {}_{\A'} \dot{\mathfrak{U}}^\imath$. On the other hand, one can not expect Lusztig's quantum Frobenius splitting restricts to a splitting of the $\imath$quantum group. This is more transparent geometrically, that is, the splitting ${}_k\fr^{ ,*}$ of $G_k$ induced by $\fr$ does not compatibly split the symmetric subgroup $K_k$ in general.  
	
	The following theorem generalizes Lusztig's quantum Frobenius splitting.
	
	\begin{thmintro}[Theorem~\ref{thm:iFr}]
 We construct an $\A'$-algebra homomorphism $\ifr : {}_{\A'} \dot{\mathfrak{U}}^\imath \rightarrow {}_{\A'} \dot{\U}^\imath $ that splits the quantum Frobenius morphism $\iofr: {}_{\A'} \dot{\U}^\imath  \rightarrow {}_{\A'} \dot{\mathfrak{U}}^\imath $, that is $\iofr \circ \ifr =  {\rm id}$.
	\end{thmintro}
			The theorem is established by direct computation. Our computation uses crucially the explicit formula of $\imath$divided powers by Berman-Wang \cite{BerW18}, as well as higher $\imath$Serre relations by Chen-Lu-Wang \cite{CLW18}.

	\subsubsection{A Frobenius splitting of $K_k$}	
		 Recall ${}_k \dot{\mathfrak{U}}^\imath = {}_k \dot{\U}^\imath$ as $k$-algebras. So ${}_k\iofr^*$ and ${}_k\ifr^{ ,*}$ (again combined with the Frobenius automorphism of $k$) become  additive endomorphisms of the coordinate ring $\RO_k^\imath$ of $K_k$.
	\begin{thmintro}[Theorem~\ref{thm:splitO}]\label{introthm:4}	
		The map $_k\iofr^*: \RO_k^\imath \rightarrow  \RO_k^\imath $ is the $p$-th power map. The map $_k\ifr^{, *}:  \RO_k^\imath \rightarrow  \RO_k^\imath $ gives a Frobenius splitting of the symmetric subgroup $K_k$.
		\end{thmintro}
	
	    \subsubsection{Future works} 	
	    It is shown by Benito-Muller-Rajchgot-Smith \cite{BMRS15} that any upper cluster algebra admits a canonical Frobenius splitting. We believe the splittings constructed here are compatible with the cluster structure on the reductive groups.

      We expect the splitting of $G_k$ in Theorem~\ref{introthm:2} compatibly splits Bruhat cells. It is also interesting to determine the compatibly split subvarieties of $K_k$ induced by the splitting in Theorem~\ref{introthm:4}.

	    Recall that the splitting in Theorem~\ref{introthm:2} does not compatibly split  $K_k$ in general. Since $K_k \subset G_k$ is a smooth subvariety of $G_k$, we know abstractly there exists a Frobenius splitting of $G_k$ that compatibly splits $K_k$ \cite{BK05}*{Proposition~1.1.6}.  It remains an open question to construct concretely an algebraic Frobenius splitting of $G_k$ that compatibly splits $K_k$.

	    We would like to extend our results to general symmetric pairs (including Kac-Moody types). The main difficulty is that the $\imath$divided powers beyond quasi-split types are much more involved and not well-understood. We hope results here can motivate further research in this direction as well.

	\subsection{$K_k$-orbits on flag varieties}Let $k$ be an algebraically closed field of characteristic $\neq  2$.
 
	\subsubsection{Algebraic Frobenius splittings} 		
		Mehta-Ramanathan's splitting on the flag variety $\CB_k$ has been translated by Kummar-Littelmann \cite{KL2} to the algebraic setting using Lusztig's quantum Frobenius splittings. 
		
		The algebraic splitting method has also been applied to the setting of cotangent bundles of flag varieties by Hague \cite{Ha13}. This simplifies the (geometric) Frobenius splitting on  cotangent bundles of flag varieties by Kumar-Lauritzen-Thomsen \cite{KLT99}. 
		
	
	\subsubsection{$K_k$-orbits on flag varieties} Let $\CB_k = G_k/B_k$ be the flag variety of $G_k$. It follows by Springer \cite{Spr85} that $K_k$ has only finitely many orbits on $\CB_k$. 
 Since the Borel subgroup $B_k$ is $\theta_k$-anisotropic, the $K_k$-orbit of $B_k/B_k$ is the unique open dense orbit. The geometry of $K_k$-orbits is a classical subject of study, with intrinsic connections with real groups, number theory, etc.

    We show the poset of $K_k$-orbits on $\CB_k$ can be parameterized independently of the characteristic of $k$ (provided not $2$) in Proposition~\ref{prop:Korbits}. 
     Using the symmetric subgroup scheme $\G^\imath_{\BZ[2^{-1}]}$ over ${\BZ[2^{-1}]}$, we further show the closure of orbits can be defined over a open subset of $ \BZ[2^{-1}]$.
 
	\begin{propintro}[Proposition~\ref{prop:KorbitsZ}]\label{prop:5}Let $\CO_k$ be a $K_k$-orbit on $\CB_k$. 
 There exists a closed subscheme $\CZ$ of the flag scheme $\CB_{\BZ[2^{-1}]}$, such that

	(1) $\CZ\rightarrow Sp\,\BZ[2^{-1}]$ is flat;
	
	(2) there is a nonempty open subset $U$ of $Sp\,\BZ[2^{-1}]$ such that for any algebraically closed field $k'$ and a morphism $Sp\,k'\rightarrow U$, the base change $\CZ_{k'} = \CZ \times_{Sp\, \BZ[2^{-1}]} Sp\, k'$ gives the closure $\overline{\mathcal{O}_{k'}}$ of the $K_{k'}$-orbit on $\mathcal{B}_{k'}$ of the same type. In particular, $\CZ_{k'}$ is reduced.
	\end{propintro} 
        The existence of such closed sub-schemes of $\CB_{\BZ[2^{-1}]}$ enables us to deduce characteristic $0$ results from the 	characteristic $p$ results.

	  \subsubsection{Codimensional one orbits}   The involution $\theta_k$ of $G_k$ induces an involution $\tau$ on the Dynkin diagram of $G_k$. Let $\I$ be the set of simple roots of $G_k$. The codimensional one $K_k$-orbits on $\CB_k$ can be described as $\{\mathcal{O}_i, i \in \I \vert \tau(i) \neq i\} \cup \{\mathcal{O}^{\pm}_i, i \in \I \vert \tau(i) = i\}$. We might have $\mathcal{O}^{+}_i = \mathcal{O}^{-}_i$, e.g., when $G_k = PGL_{2,k}$.
	    
	     The closures of the orbits $\{\mathcal{O}_i, i \in \I \vert \tau(i) \neq i\}$ are multiplicity-free divisors on $\CB_k$ in the sense of Brion \cites{Br01a, Br01b}, while the closures of the orbits $\{\mathcal{O}^{\pm}_i, i \in \I \vert \tau(i) = i\}$ are often of multiplicity-two.
	     Various nice geometric properties (e.g. normality) of multiplicity-free subvarieties have been established by Brion \cite{Br01a}. 
	    
	\subsubsection{Frobenius splitting of $K_k$-orbits} Assume the characteristic of $k$ is positive now. It is desirable to find a Frobenius splitting of the flag variety $\CB_k$ that compatibly splits $K_k$-orbit closures.  
	However, one can not expect splitting of all $K_k$-orbit closures, cf. \cite{Br01b}*{Introduction}. So the splitting of $K_k$-orbit closures is much more complicated than the case of Schubert varieties. 
	
	
	\begin{thmintro}[Theorem~\ref{thm:spl1} \& Theorem~\ref{thm:spl2}]\label{thmintro:split}
	We parameterize the codimensional one $K_k$-orbits on $\CB_k$ as $\{\mathcal{O}_i, i \in \I \vert \tau(i) \neq i\} \cup \{\mathcal{O}^{\pm}_i, i \in \I \vert \tau(i) = i\}$.
	
	    \begin{enumerate}
	        \item The quantum Frobenius splitting $\ifr$ induces a Frobenius splitting of $\CB_k$ that compatibly splits all the closures of  $\{\mathcal{O}_i, i \in \I \vert \tau(i) \neq i\} $.
	        \item Let $J \subset \I$ satisfying conditions in \ref{sec:propJ}. The quantum Frobenius splitting $\ifr$ induces a Frobenius splitting of $\CB_k$ that compatibly splits all the closures of  $\{\mathcal{O}^{\pm}_j \vert j \in J\} $. 
	       
        \item In particular, when $G_k$ is simply-laced, there exists a Frobenius splitting of $\CB_k$ that compatibly splits any given   codimensional one $K_k$-orbit.
	    \end{enumerate}
	\end{thmintro}

	The splittings constructed in (1) and (2) of the theorem are different. The splitting of multiplicity-free divisors is known by Knutson \cite{Kn09}. See also the work by He-Thomsen \cite{HT12} for some splittings relevant to symmetric subgroups. Our algebraic construction using $\imath$quantum groups is new. The splitting of non-multiplicity-free divisors is new.

	\subsubsection{Geometric consequences}
 As standard consequences of the splittings in Theorem~\ref{thmintro:split}, we obtain cohomological vanishing of ample line bundles, reducedness of scheme-theoretical intersections, etc. 
 The splittings apply to partial flag varieties as well as some orbits of smaller dimensions. We further show vanishing of higher cohomology of certain semi-ample line bundles in \S\ref{sec:geo}. These results hold even for  characteristic $0$ by semicontinuity, thanks to Proposition~\ref{prop:5}.  
 
 Normality of orbit closures can sometime be deduced from Frobenius splittings. However, this would require detailed information on the Bruhat order of $K_k$-orbits on $\CB_k$. We demonstrate for the symmetric pair $(GL_{n,k}, GL_{\lceil n/2\rceil,k}\times GL_{\lfloor n/2 \rfloor,k} )$ in \S\ref{sec:AIII}. Here the Bruhat order has been obtained by Wyser \cite{Wy16}. 
	
\subsubsection{Future works} Frobenius splittings of the flag varieties can be classified by elements in the tensor product of Steinberg modules \cite{BK05}*{\S2.3}. We expect that our splittings correspond to (dual) $\imath$canonical basis elements in such a tensor product (with respect to various choice of parameters). 

We expect $K_k$-orbit closures can be defined over $\BZ[2^{-1}]$, that is, $U = Sp\, \BZ[2^{-1}]$ in Proposition~\ref{prop:5}. Similar results for Schubert varieties were first established by Seshadri \cite{Se83}.

	\subsection{Organization of this paper}
	Let us discuss the organization of this paper in more details.
	
	We give a rather comprehensive review on the preliminaries in \S\ref{sec:pre}. Results are essentially known (except Lemma~\ref{le:nop}). Various formulations are new, including Proposition~\ref{prop:cft} and Proposition~\ref{prop:Korbits}.

	We construct the symmetric subgroup scheme $\G^\imath$ in \S\ref{sec:SymGpSch}. We show $\G^\imath$ is a closed subgroup scheme of the Chevalley group scheme $\G$. We show it is reduced over any integral domain with characteristic not 2. As consequences, we construct the coordinate ring of $K_k$ using dual $\imath$canonical basis in Theorem~\ref{thm:Oik} and relate $K_k$-modules with ${}_k\dot{\U}^\imath$-modules in Proposition~\ref{prop:KvsUi}.

	We construct the quantum Frobenius splitting of the $\imath$quantum group in \S\ref{sec:qFri}. The first part of this section is devoted to study the $\imath$quantum group over a localization of the ring $\CA= \BZ[q,q^{-1}]$ (essentially inverting the quantum $2$). The second part of this section is devoted to the proof of Theorem~\ref{thm:iFr}.  We consider the induced splittings on $\U^\imath$-modules in \S\ref{sec:splitmod}.
	
	We apply our quantum Frobenius splitting to construct a Frobenius splitting of the coordinate ring of $G^\imath_k= K_k$ in \S\ref{sec:iFrO}. We construct a Frobenius splitting of the coordinate ring of $G_k$ using Lusztig's Frobenius splitting in Theorem~\ref{thm:splitO}. 
	
	We apply our quantum Frobenius splitting to study the geometry of $K_k$-orbits on the flag variety in \S\ref{sec:Frflag}. We construct two families of splittings $\Psi_k$ and $\Psi^J_k$ of $\CB_k$. The splitting $\Psi_k$ (\S\ref{sec:Phi}) compatibly splits multiplicity-free codimensional one $K$-orbit closures simultaneously. For certain subset $J \subset \I$, we construct the splitting $\Psi^J_k$ (\S\ref{sec:PhiJ}) compatibly splits multiplicity-two codimensional one $K_k$-orbit closures relevant to $J$. We obtain geometric consequences of our splittings in \S\ref{sec:geo}. Finally, we study in details in \S\ref{sec:AIII} the symmetric pair $( GL_{n,k}, GL_{\lceil n/2\rceil,k}\times GL_{\lfloor n/2 \rfloor,k})$ to indicate how to deduce normality from our results. 
	
\vspace{.2cm}
\noindent {\bf Acknowledgement: } HB is supported by MOE grants A-0004586-00-00 and A-0004586-01-00.


\section{Preliminaries} \label{sec:pre}
\subsection{Quantum groups}
	\subsubsection{Root data}
	
	A \emph{root datum} $(\I,Y,X,A,(\alpha_i)_{i\in\I},(\alpha_i^\vee)_{i\in\I})$ consists of 
	\begin{itemize}
	    \item a finite set $\I$;
	    \item two finitely generated free abelian groups $X$, $Y$ and a perfect pairing $\langle\,,\,\rangle:Y\times X\rightarrow \mathbb{Z}$;
	    \item a (symmetrizable) generalized Cartan matrix $A=(a_{ij})_{i,j\in\I}$;
	    \item an embedding $\I\subset X$ ($i\mapsto \alpha_i$) and an embedding $\I\subset Y$ ($i\mapsto \coroot_i$) such that $\langle \coroot_i,\alpha_j\rangle=a_{ij}$.
	\end{itemize}
	
	Let $D=diag(\epsilon_i)_{i\in\I}$ be the (unique) diagonal matrix such that $DA$ is a symmetric matrix, with $\epsilon_i\in\mathbb{\BZ}_{>0}$   and $\epsilon_i=1$ for some $i\in\I$. The number $\epsilon_i$ is called the \emph{root length} of $i$. A root datum is called of \emph{finite type} if the matrix $A$ is a Cartan matrix, namely, $DA$ is positive definite.
	
	In this paper, we will assume that the root datum defined above is \emph{X-regular} and \emph{Y-regular}, namely, elements in $\{\alpha_i\mid i\in\I\}$ are linearly independent in $X$, and elements in $\{\coroot_i\mid i\in\I\}$ are linearly independent in $Y$. The isomorphism between two root data is defined in an evident manner.
	
		Let $X^+=\{\mu\in X\mid\langle\alpha_i^\vee,\mu\rangle\geqslant 0,\forall i\in\I\}$ be the set of \emph{dominant weights}, and let $X^{++}=\{\mu\in X\mid\langle\alpha_i^\vee,\mu\rangle>0,\forall i\in\I\}$ be the set of \emph{regular dominant weights}.
	
	\begin{example}
	
	Let $k$ be an algebraically closed field. For any triple $(G,T,B)$, where $G$ is a connected reductive group, $T$ is a maximal torus of $G$, and $B$ is a Borel subgroup of $G$ containing $T$, one can associate a root datum (of finite type) in the following way.
	
	 Let $X=\text{Hom}(T,k^\times)$, and $Y=\text{Hom}(k^\times,T)$ be the character lattice and cocharacter lattice. Then there is a natural pairing between $Y$ and $X$. Let $(\alpha_i)_{i\in\I}\subset X$ (resp. $(\coroot_i)_{i\in\I}\subset Y$) be the simple roots (resp. simple coroots) associated to the Borel subgroup $B$, where $\I$ is a finite index set. Write $a_{ij}=\langle \alpha_i^\vee,\alpha_j\rangle$, for any $i,j$ in $\I$. Then $(\I,Y,X,A,(\alpha_i)_{i\in\I},(\coroot_i)_{i\in\I}))$ is a root datum of finite type. 
	 
	 For fixed $G$, by choosing different pair $(T,B)$, one will get isomorphic root data. We call any root datum in this isomorphism class a \emph{root data associated to $G$}. There is a bijection between isomorphism classes of connected reductive groups over algebraically closed field, and isomorphism classes of root data of finite type.
	\end{example}
	
	\subsubsection{Definitions} 
	Fix a root datum $(\I,Y,X,A,(\alpha_i)_{i\in\I},(\coroot_i)_{i\in\I})$. Let $q$ be an indeterminate. We write $\A=\mathbb{Z}[q,q^{-1}]$ be the subring of $\mathbb{Q}(q)$. Write $q_i=q^{\epsilon_i}$ for all $i$ in $\I$. For $m,n,d\in \mathbb{Z}$ with $m\geqslant 0$, define 
	\begin{equation*}
	[n]=\frac{q^n-q^{-n}}{q-q^{-1}} \quad \text{ and } \quad [m]!=[1][2]\cdots[m].
	\end{equation*}
	These are called \emph{q-integers} and \emph{q-factorials}. Also define \emph{q-binomial coefficients}:
	\begin{equation*}
	\qbinom{n}{d}=\left\{\begin{array}{cc}
	\frac{[n][n-1]\cdots[n-d+1]}{[d]!}, & \text{ if }d\geqslant 0 \\
	0, & \text{ if }d<0
	\end{array}
	\right.
	\end{equation*}
	Similarly define $[n]_i$, $[m]_i!$ and $\qbinom{n}{d}_i$ with $q$ replaced by $q_i$. Note that these are all elements in $\A$.
	
	The quantum group $\U$ associated to the given root datum is the associated $\mathbb{Q}(q)$-algebra (with 1) with generators 
	\begin{equation*}
	E_i \quad (i\in\I),\qquad F_i \quad (i\in\I), \qquad K_\mu \quad ( \mu\in Y),
	\end{equation*}
	subject to the relations
	\begin{align*}
	&K_0=1,\quad K_\mu K_{\mu'}=K_{\mu+\mu'}, \text{ for all }\mu,\,u'\in Y;\\
	&K_\mu E_i=q^{\langle \mu,\alpha_i\rangle}E_i K_\mu, \text{ for all }i\in\I,\, \mu\in Y;\\
	&K_\mu F_i=q^{-\langle \mu,\alpha_i\rangle}F_i K_\mu, \text{ for all }i\in\I,\, \mu\in Y;\\
	&E_iF_j-F_jE_i=\delta_{ij}\frac{\Tilde{K}_i-\Tilde{K}_{i}^{-1}}{q_i-q_i^{-1}}, \text{ for all }i,j\in\I, \text{ where }\Tilde{K}_i=K_{\coroot_i}^{\epsilon_i} ;\\
	&\sum_{n=0}^{1-a_{ij}}(-1)^nE_i^{(n)}E_jE_i^{(1-a_{ij}-n)}=0,\qquad
	\sum_{n=0}^{1-a_{ij}}(-1)^nF_i^{(n)}F_jF_i^{(1-a_{ij}-n)}=0.
	\end{align*}
Here $E_i^{(n)}=E_i^n/[n]_i!$ and $F_i^{(n)}= F_i^n/[n]_i!$ are the divided powers.
	
	 Let $\dot{\U}$ be the modified quantum group \cite{Lu93}*{23.1} and $\dot{\RB}$ be its canonical basis (\cite{Lu93}*{25.2.1}). Let $\AU$ be the $\A$-form of $\dot{\U}$, which is an $\A$-subalgebra of  $\dot{\U}$ generated by $E_i^{(n)} 1_\lambda$ and $F_i^{(n)} 1_\lambda$ for various $i \in I$, $n \ge 0$ and $\lambda \in X$. Then $\AU$ is a free $\A$-module with basis $\dot{\RB}$. We denote by $\AU^{>0}$ (resp. $\AU^{<0}$) the $\A$-subalgebra generated by $E_i^{(n)} 1_\lambda$ (resp. $F_i^{(n)} 1_\lambda$) for various $i \in I$, $n \ge 0$ and $\lambda \in X$.
	
\subsubsection{Modules} \label{sec:qanniR}

For any $\lambda\in X^+$, let $L(\lambda)$ be the highest weight $\U$-module with highest weight $\lambda$, and $^\omega L(\lambda)$ be the $\U$-module where the action is twisted by the involution $\omega$ (\cite{Lu93}*{3.1.3}). We denote by $\RB(\lambda)$ and ${}^\omega\RB(\lambda)$ the canonical bases of $L(\lambda)$ and $^\omega L(\lambda)$, respectively.  Write $v_\lambda^+$ and $v_{-\lambda}^-$ to denote the highest weight vector and lowest weight vector in $\RB(\lambda)$ and ${}^\omega\RB(\lambda)$, respectively. $L(\lambda)$ also admits a natural $\dot{\U}$-action. We write $_\A L(\lambda)={_\A\dot{\U}\cdot v_\lambda^+}$ to be the $\A$-form of $L(\lambda)$.

For any $\A$-algebra $R$ and $\lambda \in X^+$, we write $_R\dot{\U}=R\otimes_{\A} {_\A\dot{\U}}$ and $_R L(\lambda)=R\otimes_{\A} {_\A L(\lambda)}$.

For any $\lambda_1, \lambda_2$ in $X^+$, and any $\A$-algebra $R$, the action map
$
{}_R\dot{\U}\rightarrow {_R^\omega L(\lambda_1)}\otimes{_R L(\lambda_2)}$, $ x\mapsto x\cdot (v_{\lambda_1}^-\otimes v_{\lambda_2}^+)$ has kernel 
\[
_R P(\lambda_1,\lambda_2) = \!\!\!\sum_{\lambda\neq\lambda_2-\lambda_1} \!\!\! {_R\dot{\U}\one_\lambda}
+
\!\!\!\sum_{i\in\I,\,a_i>\langle \alpha_i^\vee,\lambda_1\rangle}\!\!\! {_R\dot{\U}E_i^{(a_i)}}\one_{\lambda_2-\lambda_1}
+
\!\!\!\sum_{i\in\I,\, b_i>\langle \coroot_i,\lambda_2\rangle}\!\!\! {_R\dot{\U}}F_i^{(b_i)}\one_{\lambda_2-\lambda_1}.
\]

The claim follows from \cite{Lu93}*{Proposition 23.3.6} for the case $R=\A$. The general case follows from the fact that $_\A P(\lambda_1,\lambda_2)$ is a free $\A-$submodule of $_\A\dot{\U}$ spanned by canonical basis elements \cite{Lu93}*{Theorem 25.2.1}. In particular, $_R P(\lambda_1,\lambda_2)$ is a free $R$-module.

	\subsubsection{Quantum Frobenius homomorphism}\label{sec:qFr}
	
	Fix an odd number $l>1$. We require that $l$ is relatively prime to all the root length $\epsilon_i$. Let $f_l\in \mathcal{A}$ be the $l$-th cyclotomic polynomial. Set $\mathcal{A}'=\mathcal{A}/(f_l)$, and $\mathbf{F}$ be the quotient ring of $\mathcal{A}'$. Let $\phi:\mathcal{A}\rightarrow\mathcal{A}'$ be the natural quotient map. Then $\phi$ extends to a ring homomorphism from the local ring $\mathcal{A}_{(f_l)}$ to $\mathbf{F}$.
	
	Let $c:\mathcal{A}\rightarrow\mathcal{A'}$ be the ring homomorphism sending $q$ to 1. We will write $\A'_\phi$ and $\A'_c$ to distinguish two different $\A$-algebra structures. Define
	\begin{align*}  
	_{\mathcal{A}'}\dot{\mathfrak{U}}={_{\A'_c}\dot{\U}}={\mathcal{A}'_c}\otimes_\A ({_{\mathcal{A}}\dot{\mathrm{U}}}),\qquad 
	_{\mathcal{A}'}\dot{\mathrm{U}}={_{\A'_\phi}\dot{\U}}=\mathcal{A}'_\phi\otimes_\A (_\mathcal{A}\dot{\mathrm{U}}).
	\end{align*}
	
	For any $i\in\I$, $\zeta\in X$, we use $\mathfrak{e}_i^{(n)}\one_\zeta$ (resp. $\mathrm{E}^{(n)}_i\one_\zeta$) to denote the image of $E_i^{(n)}\one_\zeta$ in $\cdotU$ (resp. $\pdotU$). Similar notations are used for $\mathfrak{f}_i^{(n)}\one_\zeta$ and $\mathrm{F}_i^{(n)}\one_\zeta$.
	
	\begin{theorem}(\cite{Lu93}*{35.1.9},\cite{Mc07}*{Proposition~3.4})\label{thm:ClaFr}

		(a) There is a unique $\A'$-algebra homomorphism \ofr: $\pdotU \rightarrow\cdotU$, such that for all $i\in\I$, $n\in\mathbb{N}$ and $\zeta\in X$, we have:
		
		$\ofr(\mathrm{E}_i^{(n)}\one_\zeta)$ equals $\mathfrak{e}_i^{(n/l)}\one_{\zeta/l}$, if $l$ divides $n$ and $\zeta\in lX$, and equals 0, if otherwise;
		
		$\ofr(\mathrm{F}_i^{(n)}\one_\zeta)$ equals $\mathfrak{f}_i^{(n/l)}\one_{\zeta/l}$, if $l$ divides $n$ and $\zeta\in lX$, and equals 0, if otherwise.
		
		(b) There is a unique $\A'$-algebra homomorphism $\fr$: $\cdotU\rightarrow\pdotU$, such that for all $i\in\I$, $n\in\mathbb{N}$ and $\zeta\in X$, we have $\fr(\mathfrak{e}_i^{(n)}\one_\zeta)=\mathrm{E}_i^{(nl)}\one_{l\zeta}$, and $\fr(\mathfrak{f}_i^{(n)}\one_\zeta)=\mathrm{F}_i^{(nl)}\one_{l\zeta}$. In particular, we have $\ofr\circ\fr=$id.
	\end{theorem}

For any $\lambda\in X^+$, set ${_{\A'}L}(\lambda)=\A'_{\phi} \otimes_\A ({_\A L(\lambda)})$, and ${_{\A'}\mathfrak{L}}(\lambda)=\A'_c \otimes_\A ({_\A L(\lambda)})$. We will abuse the notations, and still use $v_\lambda^+$ to denote the highest weight vector in these modules.
We also have the following proposition by Theorem \ref{thm:ClaFr}{(a)} and \S\ref{sec:qanniR}.

\begin{proposition}\label{prop:ga}
    For any $\mu\in X^+$, there exists a unique $_{\A'}\dot{\U}$-module homomorphism
    \[
    \gamma_\mu: {_{\A'}L}(l\mu)\longrightarrow \big({_{\A'}\mathfrak{L}}(\mu)\big)^\ofr, \quad v_{l\mu}^+ \mapsto v_\mu^+.
    \]
  Here $\big({_{\A'}\mathfrak{L}}(\mu)\big)^\ofr$ stands for the $_{\A'}\dot{\U}$-module which is isomorphic to ${_{\A'}\mathfrak{L}}(\mu)$ as $\A'$-modules, and the action is twisted by $\ofr$.
\end{proposition}


\subsection{Chevalley group schemes}
In this subsection, we shall assume that our root datum $(\I,Y,X,A,(\alpha_i)_{i\in\I},(\alpha_i^\vee)_{i\in\I})$ is of finite type. We recall Lusztig's construction of the group scheme $\G$ in \cite{Lu07}.
\subsubsection{Lusztig's construction}\label{sec:nocp}

Let $A$ be a commutative ring with 1, which is viewed as an $\A$-algebra, where $q$ acts by 1. For any element $a\in\A$, we will abuse the notation and still write $a$ to denote its image in $A$.

Let ${}_A\dot{\U} = A \otimes_{\A} {}{_\A}\dot{\U}$ and ${}_A\dot{\RB}=\{1\otimes b\mid b\in\dot{\RB}\}$. The set ${}_A\dot{\RB}$ is called the canonical basis of $_A\dot{\U}$. If there is no confusion, we will write ${}_A\dot{\RB}=\dot{\RB}$, and use $b$ to denote the image $1\otimes b$.

Let $_A\widehat{\mathrm{U}}$ be the $A$-module consisting of all the formal linear combinations
\[
    \sum_{a\in\dot{\RB}} n_a a, \quad n_a\in A.
\]
By \cite{Lu07}*{1.11}, $_A\widehat{\U}$ has a structure of $A$-algebra compatible with the embedding $_A\dot{\U} \subset {}_A\widehat{\U}$.  
    
 Let $_A\widehat{\mathrm{U}}^{(2)}$ be the $A$-module consisting of all the formal linear combinations
\[
    \sum_{(a,a')\in \dot{\RB}\times \dot{\RB}} n_{a,a'}a\otimes a', \quad n_{a,a'}\in A.
\]
Then ${}_A\widehat{\U}^{(2)}$ also has an $A$-algebra structure by \cite{Lu07}*{1.11}. We have an $A$-algebra homomorphism $\WD:{_A\widehat{\mathrm{U}}}\longrightarrow {_A\widehat{\mathrm{U}}}^{(2)}$, compatible with the coproduct on $_A\dot{\U} $ by \cite{Lu07}*{1.17}.
Let $S:{_A\dot{\U}}\rightarrow {_A\dot{\U}}$ be the antipode map. It extends to $\widehat{S}:{_A\widehat{\U}}\rightarrow {_A\widehat{\U}}$.


Let $\RO_A$ be the $A$-submodule of $_A\dot{\U}^*=\text{Hom}_A({_A\dot{\U}},A)$, spanned by $\{b^*\mid b\in \dot{\RB}\}$, where $b^* \in {}_A\dot{\U}^*$ is the \emph{dual canonical basis} element which sends $b$ to 1, and sends other canonical basis element to 0. Then for any element $f$ in $\RO_A$, it extends to an $A$-linear form on the completion $_A\widehat{\U}$. We write $\hat{f}$ to denote this extension.


Lusztig defined an $A$-Hopf algebra structure on $\RO_A$, using structure constants for the algebra $_A\dot{\U}$ (\cite{Lu07}*{3.1}). Let $\delta:\RO_A\rightarrow\RO_A\otimes\RO_A$ to be the coproduct, $\sigma:\RO_A\rightarrow\RO_A$ to be the antipode, and $\epsilon:\RO_A\rightarrow A$ to be the counit.

For any $f$, $g \in \RO_A$, it follows from \cite{Lu07}*{3.1} that the following diagrams commute
\begin{equation*}
    \begin{tikzcd}
    & _A\widehat{\U} \arrow[r,"\WD"] \arrow[rd,"\widehat{fg}"'] & {_A\widehat{\U}}^{(2)} \arrow[d,"f\otimes g"] \\ &  & A
    \end{tikzcd}
    \begin{tikzcd}
    &  {_A\widehat{\U}}\otimes{_A\widehat{\U}} \arrow[r,"m"] \arrow[rd,"\delta(f)"'] & {_A\widehat{\U}} \arrow[d,"\hat{f}"] \\ & & A
    \end{tikzcd}
  \quad
    \begin{tikzcd}
    _A\widehat{\U} \arrow[r,"\widehat{S}"] \arrow[rd,"\widehat{\sigma(f)}"'] & {_A\widehat{\U}} \arrow[d,"\hat{f}"] \\ & A
    \end{tikzcd}.
\end{equation*}
Here $f\otimes g: {_A\widehat{\U}^{(2)}}\rightarrow A$ sends formal linear combination $\sum_{(a,a')\in \dot{\RB}\times\dot{\RB}}n_{a,a'}a\otimes a'$ to $\sum_{(a,a')\in \dot{\RB}\times\dot{\RB}}n_{a,a'}f(a) g(a')$, where the second sum is finite, thanks to the definition of $\RO_A$. Similarly $\delta(f)\in \RO_A\otimes\RO_A$ induces a well-defined $A$-linear form on $_A\widehat{\U}\otimes{_A\widehat{\U}}$. We still use $\delta(f)$ to denote this map.

Since $\RO_A$ is a commutative Hopf algebra over $A$, the set $\text{Hom}_{A-\text{alg}}(\RO_A, A)$, consisting of $A$-algebra homomorphisms $\RO_A\rightarrow A$ preserving 1, has a group structure.

By definition, $\text{Hom}_{A-\text{alg}}(\RO_A,A)$ is an $A$-submodule of $\RO_A^{*}$, the $A$-linear duals of $\RO_A$. There is an $A$-linear bijective map  from $\RO_A^{*}$ to $_A\widehat{\U}$, given by $\phi\mapsto \sum_{b\in\dot{\RB}}\phi(b^*)b$. Then under this bijection, the subset $\text{Hom}_{A-\text{alg}}(\RO_A,A)$ is sent to
\[
G_A=\{\xi=\sum_{b\in\dot{\RB}}n_b b\in {_A\widehat{\mathrm{U}}}\mid \WD(\xi)=\xi\otimes\xi,\,n_{1_0}=1\}.
\]

Here for any $\xi=\sum_{b\in \dot{\RB}}n_bb$ and $\xi'=\sum_{b'\in\dot{\RB}}n'_{b'}b'$ in $_A\widehat{\U}$, the element $\xi\otimes\xi'$ in $_A\widehat{\U}^{(2)}$ is defined to be the formal linear combination
\[
\xi\otimes\xi'=\sum_{(b,b')\in \dot{\RB}\times \dot{\RB}} n_{b}n'_{b'}b\otimes b'.
\]

Then the subset $G_A$ is closed taking products. It moreover admits a group structure, where the unit is $\sum_{\lambda\in X}\one_\lambda$, and the inverse is given by the restriction of the antipode $\widehat{S}:{_A\widehat{\U}}\rightarrow{_A\widehat{\U}}$. The bijection $\text{Hom}_{A-\text{alg}}(\RO_A,A)\xrightarrow{\sim}G_A$ is a group isomorphism.

Define the $\BZ$-group scheme $\G$ by setting $\G(A)=G_A$ \footnote{Here we identify a $\BZ$-group scheme with the associated $\BZ$-group functor following \cite{Jan03}} for any (unital) commutative ring $A$. Then $\G\cong Sp\,\RO_\BZ$ is the Chevalley group scheme associated to the given root datum.

\subsubsection{The reductive group $G_k$}\label{sec:regpq}
Let $k$ be an algebraically closed field. Set $G_k$ be the subset of $_k\widehat{\U}$ defined as above by letting $A=k$.

Then $G_k$ is a connected reductive group over $k$ with the coordinate ring $\RO_k$ by \cite{Lu07}*{Theorem~4.11}. 
Let $T_k$ be the subset of $G_k$ consisting of elements of the form $\sum_{\lambda\in X}n_\lambda\one_\lambda$, such that $n_\lambda n_{\lambda'}=n_{\lambda+\lambda'}$. Let $_k\widehat{\U}^{>0}$ be the $k$-subspace of $_k\widehat{\U}$ consisting of all elements of the form $  \sum_{b\in \dot{\RB} \cap {}_k\dot{\U}^{>0},\lambda\in X}n_b(b\one_\lambda)$ with $n_b\in k$. Set $G_k^{>0}=G_k\cap {_k\widehat{\U}^{>0}}$ and $B_k=G_k^{>0}T_k$. It follows from \cite{Lu07}*{4.11} that $T_k$ is a maximal torus of $G_k$ and $B_k$  is a Borel subgroup of $G_k$.

We have the following isomorphisms of free abelian groups
\begin{align*}
X&\xrightarrow{\sim} \text{Hom}(T_k,k^\times),\qquad &Y\xrightarrow{\sim} &\text{Hom}(k^\times,T_k)\\
\lambda & \longmapsto \big(\sum_{\lambda\in X}n_\lambda\one_\lambda\mapsto n_\lambda\big),\qquad &\gamma\longmapsto &\big(a\mapsto \sum_{\lambda\in X}a^{\langle \gamma,\lambda\rangle}\one_\lambda\big).
\end{align*}
These two maps give an isomorphism between the root data defining the quantum group with the root data associated to $(G_k,T_k,B_k)$.

For any $i\in\I$ and $\xi\in k$, define 
\[
x_i(\xi)=\sum_{c\in\BN,\;\lambda\in X}\xi^cE_i^{(c)}\one_\lambda,\qquad y_i(\xi)=\sum_{c\in\BN,\;\lambda\in X}\xi^cF_i^{(c)}\one_\lambda.
\]
Then $\{x_i,y_i\mid i\in\I\}$ is a {\em pinning} for the triple $(G_k,T_k,B_k)$, that is, the morphism 
\[
\begin{pmatrix}
1 & a \\ & 1
\end{pmatrix}\mapsto x_i(a)\quad 
\begin{pmatrix}
1 &  \\ b &1
\end{pmatrix}\mapsto y_i(b)\quad
\begin{pmatrix}
t & \\ & t^{-1}
\end{pmatrix}\mapsto \coroot_i(t)
\]
for any $a, b$ in $k$, and $t$ in $k^*$, give a group morphism from $\mathrm{SL}_{2,k}$ to $G$.

\subsubsection{Lie algebras and  algebras of distributions}\label{sec:liealg} 

 For any linear algebraic group $H_k$ over $k$, we use the notation $\lie(H_k)$ and $\dist(H_k)$ to denote the Lie algebra and the  algebra of distributions  on $H_k$. They are subalgebras of $k[H_k]^*$, the linear dual of the coordinate ring of $H_k$. 
 


Recall $\RO_k$ is the coordinated ring the connected reductive group $G_k$. We identify $\RO_k^*$ with $_k\widehat{\U}$, via $\mu\mapsto\sum_{b\in\dot{\RB}}\mu(b^*)b$. Then its Lie algebra and the distribution algebra are identified with subalgebras of $_k\widehat{\U}$.

For any $i\in\I$, $n\in\BN$, and $\mu\in Y$, define the following elements in $_k\widehat{\U}$:
\[
e_i^{(n)}=\sum_{\lambda\in X}E_i^{(n)}\one_\lambda,\qquad f_i^{(n)}=\sum_{\lambda\in X}F_i^{(n)}\one_\lambda,\qquad \binom{h_\mu}{n}=\sum_{\lambda\in X}\binom{\langle \mu,\lambda\rangle}{n}\one_\lambda.
\]
Set $e_i=e_i^{(1)}$, $f_i=f_i^{(1)}$, and $h_\mu=\binom{h_\mu}{1}$. Here we use the same notations to denote the elements in the algebra after base change.
 
The Lie algebra $\lie(G_k)$ is identified as the Lie subalgebra of $_k\widehat{\U}$ generated by $e_i$ ($i\in\I$), $f_i$ ($i\in\I$), and $h_\mu$ ($\mu\in Y$). The distribution algebra of $G_k$ is identified as the subalgebra of $_k\widehat{\U}$ generated by $e_i^{(n)}$ ($i\in\I$, $n\in\BN$), $f_i^{(n)}$ ($i\in\I$, $n\in\BN$), and $\binom{h_\mu}{n}$ ($\mu\in Y$, $n\in \BN$).

\subsection{Quantum symmetric pairs} 
\subsubsection{$\imath$root datum}
\begin{definition}
A quasi-split \emph{$\imath$root datum} consists of the following 
\begin{itemize}
    \item a root datum $(\I,Y,X,A,(\alpha_i)_{i\in\I},(\coroot_i)_{i\in\I}))$;
    \item an involution $\tau$ of the index set $\I$, such that $a_{ij}=a_{\tau i,\tau j}$, for all $i, j\in\I$;
    \item an involution $\theta_X$ on $X$, and an involution $\theta_Y$ on $Y$, such that $\langle \theta_Y(\lambda),\theta_X(\mu)\rangle=\langle \lambda,\mu\rangle$, for $\lambda\in Y$, $\mu\in X$, and  $\theta_X(\alpha_i)=-\alpha_{\tau i}$, $\theta_Y(\coroot_i)=-\coroot_{\tau i}$, for any $i\in\I$.
\end{itemize}
\end{definition}

If there is no ambiguity, we will use $\theta$ to denote both the involution on $X$ and on $Y$. We write $(\I,Y,X,A,(\alpha_i)_{i\in\I},(\coroot_i)_{i\in\I}),\tau,\theta)$ to denote an $\imath$root datum. We call an $\imath$root datum \emph{of finite type} if the associated Cartan datum is of finite type. Morphisms and isomorphisms between $\imath$root data are defined in an evident manner.

For an $\imath$root datum $(\I,Y,X,A,(\alpha_i)_{i\in\I},(\coroot_i)_{i\in\I}),\tau,\theta)$, set 
	\begin{equation*}
	X_\imath=X/\langle \lambda-\theta\lambda\mid\lambda\in X\rangle\qquad \text{and}\qquad Y^\imath=\{\mu\in Y\mid \theta\mu=\mu\}.
	\end{equation*}
	We call $X_\imath$ the \emph{$\imath$weight lattice}, and call $Y^\imath$ the \emph{$\imath$coweight lattice}. There is a natural pairing $Y^\imath\times X_\imath\rightarrow\mathbb{Z}$ inherited from the pairing between $Y$ and $X$. Note that this pairing is not necessarily perfect. For any $\lambda\in X$, we write $\overline{\lambda}$ its image in $X_\imath$.

\begin{remark}
One can define $\imath$root data of more general (not necessarily quasi-split) type. In this paper we shall only consider quasi-split $\imath$root data, and we will just call them $\imath$root datum for brevity.
\end{remark}

\begin{lemma}\label{le:nop}
The abelian group $X_\imath$ has no odd torsion. Namely, for any $\overline{\lambda} \in X_\imath$ and any odd $n\in\BN$ with $n\overline{\lambda}=0$, we have  $\overline{\lambda}=0$.
\end{lemma}
	
\begin{proof}
Let  $\BZ[2^{-1}] \subset \BQ$ be the localization of $\BZ$ inverting $2$. Let $\Breve{X} = \langle \lambda-\theta\lambda\mid\lambda\in X\rangle$. By definition, we have the short exact sequence
\[
0\rightarrow \Breve{X}\rightarrow X\rightarrow X_\imath\rightarrow 0.
\]
Tensoring with the flat $\BZ$-module $\BZ[2^{-1}]$, we get
\[
0\rightarrow \Breve{X}\otimes \BZ[2^{-1}]\rightarrow X\otimes \BZ[2^{-1}]\rightarrow X_\imath\otimes \BZ[2^{-1}]\rightarrow 0.
\]

Note that the $\BZ[2^{-1}]$-module homomorphism 
\[
X\otimes\BZ[2^{-1}]\longrightarrow \Breve{X}\otimes \BZ[2^{-1}]
\]
sending $\lambda\otimes 1$ to $(\lambda-\theta\lambda)\otimes 1/2$ gives a splitting for the above short exact sequence. Hence $X_\imath\otimes \BZ[2^{-1}]$ can be viewed as a submodule of $X\otimes \BZ[2^{-1}]$, which is torsion-free. We deduce that $X_\imath\otimes \BZ[2^{-1}]$ is torsion-free. Then by the structure theory of finitely generated abelian groups, we see that $X_\imath$ has no odd torsion. 
\end{proof}
\subsubsection{$\imath$quantum groups}\label{sec:iqp}
	For each $i\in\I$, we fix a parameter $\varsigma_i\in q^{\mathbb{Z}}$, such that 
	\begin{equation*}
	\varsigma_{\tau i}=q_i^{-a_{i,\tau i}}{\varsigma}_i^{-1},\quad
	\varsigma_{\tau i}=\varsigma_i\text{ if }a_{i,\tau i}=0,\quad
\varsigma_i=q_i^{-1} \text{ if }\tau i=i.
	\end{equation*}
	Then the associated {\em $\imath$quantum group} $\U^\imath=\U^\imath_{\mathbf{\varsigma}}$ is defined to be the $\mathbb{Q}(q)$-subalgebra of $\U$ generated by elements
	\begin{equation*}
	B_i=F_i+\varsigma_iE_{\tau i}\Tilde{K}_i^{-1}\quad (i\in\I),\qquad K_\mu\quad (\mu\in Y^\imath).
	\end{equation*}
	
	The pair $(\U, \U^\imath)$ is called a {\em quantum symmetric pair}. We refer to \cite{BW21}*{\S3} for a more general definition for $\imath$quantum groups. 
	
	
	Let $\dot{\U}^\imath$ be the modified algebra of $\U^\imath$ and let $\dot{\RB}^\imath$ be the canonical basis of $\dot{\U}^\imath$ defined in \cite{BW18a}*{\S3.7, \S6.4}. The algebra $\dot{\U}$ is naturally a $(\dot{\U}^\imath, \dot{\U}^\imath)$-bimodule by \cite{BW18a}*{\S3.7}.
Let $_\A\dot{\U}^\imath$ be the $\A$-form of $\dot{\U}^\imath$ \cite{BW18a}*{Definition~3.19}. For any $\A$-algebra $A$, set ${}_A\dot{\U}^\imath=A\otimes_{\A} {_\A\dot{\U}^\imath}$.

(a) {\it For any $\A$-algebra $A$, we have an $A$-algebra embedding ${}_A\dot{\U}^\imath \rightarrow  {_A\widehat{\mathrm{U}}}$, $x \mapsto \sum_{\lambda \in X} x \one_\lambda$. Here $x \one_\lambda \in \dot{\U}$ is via the ${}_A\dot{\U}^\imath$-action on ${}_A\dot{\U}$.}

	Since  ${}_\A\dot{\U}^\imath$ is a free $\A$-module by  \cite{BW18a}*{Theorem~6.17}, it suffices to prove the statement for $A = \A$. The claim for $A = \A$ follows from \cite{BW18a}*{\S3.7} and \cite{Lu07}*{\S1.11}.

We denote by $\iota_\lambda: {_\A\dot{\U}^\imath\one_{\overline{\lambda}}} \rightarrow {_\A\dot{\U}\one_\lambda}$ the map $ x \mapsto x \one_\lambda$. We define the composition
\begin{equation}\label{eq:pim}
p_{\imath,\lambda}=\pi_\lambda\circ\iota_\lambda:{_\A\dot{\U}^\imath\one_{\overline{\lambda}}}\xrightarrow{\iota_\lambda}{_\A\dot{\U}\one_\lambda}\xrightarrow{\pi_\lambda}{_\A\U^-\one_\lambda}, 
\end{equation}
where $\pi_\lambda:{_\A\dot{\U}\one_\lambda}\rightarrow {_\A\dot{\U}\one_\lambda}/{_\A\dot{\U}\U^+\one_\lambda}\xrightarrow{\sim} {_A\U^-}\one_\lambda$.
It follows from \cite[Corollary 6.20]{BW18a} that $p_{\imath,\lambda}$ is an isomorphism of $\A$-modules.

Let $A$ be any $\A$-algebra. After tensoring with $A$, we deduce that the composition
\begin{equation}\label{eq:pimA}
p_{\imath,\lambda}=\pi_\lambda\circ\iota_\lambda:{_A\dot{\U}^\imath\one_{\overline{\lambda}}}\xrightarrow{\iota_\lambda}{_A\dot{\U}\one_\lambda}\xrightarrow{\pi_\lambda} {_A\U^-\one_\lambda}
\end{equation}
is an isomorphism between $A$-modules. 

We recall the construction for \emph{$\imath$divided powers}.
	
	For $i\in\I$ with $\tau i\neq i$, and $m\in\mathbb{N}$, set
	\begin{equation*}
	B_i^{(m)}=\frac{B_i^m}{[m]_i!}.
	\end{equation*}
	For $i\in \I$ with $\tau i=i$, and $m\in\mathbb{N}$, set
	\begin{eqnarray*}
	&&B_{i,\odd}^{(m)} =\frac{1}{[m]_i^!}\left\{ \begin{array}{ccccc} B_i\prod_{j=1}^k (B_i^2-[2j-1]_i^2 ), & \text{if }m=2k+1;\\
	\prod_{j=1}^k (B_i^2-[2j-1]_i^2), &\text{if }m=2k; \end{array}\right.
	\\
	&&B_{i,\ev}^{(m)} = \frac{1}{[m]_i^!}\left\{ \begin{array}{ccccc} B_i\prod_{j=1}^k (B_i^2-[2j]_i^2 ), & \text{if }m=2k+1;\\
	\prod_{j=1}^{k} (B_i^2-[2j-2]_i^2), &\text{if }m=2k. \end{array}\right.
	\end{eqnarray*}
 
	We also recall the {\em modified $\imath$divided powers}.
	For $\tau i\neq i$, $\zeta\in X_\imath$, and $m\in\mathbb{N}$, we define
	$B_{i,\zeta}^{(m)}=B_i^{(m)}\one_\zeta$.
	For $\tau i=i$, $\zeta\in X_\imath$, and $m\in\mathbb{N}$, we define $B_{i,\zeta}^{(m)}=B_{i,\ev}^{(m)}\one_\zeta$ if $\langle \alpha^\vee_i,\zeta\rangle$ is even;
	$B_{i,\zeta}^{(m)}=B_{i,\odd}^{(m)}\one_\zeta$  if $\langle \alpha^\vee_i,\zeta\rangle$ is odd.
	The algebra ${_\A}\dot{\U}^\imath$ is generated by $B_{i,\zeta}^{(n)}$, for various $i\in\I$, $\zeta\in X_\imath$, and $n\in\mathbb{N}$.

\subsubsection{Stability}

For the remaining part of this section, we assume root datum is of finite type. We also fix parameters as following: $\varsigma_i=q^{-1}_i$, if $\tau i=i$; $\varsigma_i=1$, if $a_{i,\tau i}=0$; and $\{\varsigma_i,\varsigma_{\tau i}\}=\{1,q_i^{-1}\}$, if $a_{i,\tau i}=-1$. 

For any $\mu\in X^+$, let $\RB[\lambda]^\imath$ be the $\imath$canonical basis for the irreducible module $L(\lambda)$ (\cite{BW18a}*{Theorem 5.7}). 
Let $\pi_\lambda^\imath:\dot{\U}^\imath\one_{\overline{\lambda}}\rightarrow L(\lambda)$ be the $\U^\imath$-module homomorphism sending $u$ to $u\cdot v_\lambda^+$. Note that $\dot{\U}^\imath\one_{\overline{\lambda}}$ is spanned by the $\imath$canonical basis. 

\begin{theorem}\label{thm:stab}(\cite{Wa21}*{Theorem 6.2.8})
For any $\lambda\in X^+$, the map $\pi_\lambda^\imath$ sends $\imath$canonical basis element to an $\imath$canonical basis element or zero. And the kernel of $\pi_\lambda^\imath$ is a subspace of $\dot{\U}^\imath\one_{\overline{\lambda}}$ spanned by $\imath$canonical basis elements. In particular, there are only finitely many $\imath$canonical basis element $b$ in $\dot{\RB}^\imath$, such that $b\cdot v_\lambda^+\neq 0$.
\end{theorem}

\subsection{Symmetric subgroups}

Let $k$ be an algebraically closed field with characteristic not 2. 

\subsubsection{Definitions} Let $G_k$ be a connected reductive algebraic group defined over $k$. A \emph{symmetric pair} $(G_k,\theta_k)$ consists of a connected reductive group $G_k$, and an involution $\theta_k$ of $G_k$. Two symmetric pairs $(G_k,\theta_k)$ and $(G'_k,\theta_k')$ are called isomorphic if there exists an isomorphism between algebraic groups $f:G_k\rightarrow G'_k$, such that $f\circ\theta_k=\theta_k'\circ f$.

\begin{definition}\cite{Spr87}*{3.4}
An involution $\theta_k$ of $G_k$ as algebraic groups is called \emph{quasi-split} if there exists some Borel subgroup $B_k$, such that $\theta_k(B_k)\cap B_k$ is a maximal torus of $G_k$. Such a Borel subgroup $B_k$ is called \emph{$\theta_k$-anisotropic}.

 We call a symmetric pair $(G_k,\theta_k)$ \emph{quasi-split} if the involution $\theta_k$ is quasi-split.
\end{definition}

In this paper, we shall only consider quasi-split symmetric pairs. In the sequel, symmetric pairs are always assumed to be quasi-split.

For any {\em anisotropic triple} $(G_k,\theta_k,B_k)$, where $(G_k,\theta_k)$ is a symmetric pair  and $B_k$ is a $\theta_k$-anisotropic Borel subgroup of $G_k$, one can associate an $\imath$root datum in the following way.

Let $T_k=B_k\cap\theta_k(B_k)$ be the maximal torus. Let $(\I,Y,X,(\alpha_i)_{i\in\I},(\coroot_i)_{i\in\I}))$ be the root datum associated to $(G_k,T_k,B_k)$. It is easy to see that $\theta_k$ induces involutions on $X$ and $Y$, which respects the pairing, and $\theta_k$ restricts to an involution on the set of simple roots, which induces a graph involution $\tau$ on the Cartan datum $(\I,\cdot)$. Therefore, $(\I,Y,X,(\alpha_i)_{i\in\I},(\coroot_i)_{i\in\I},\tau,\theta_k)$ forms a quasi-split $\imath$root datum of finite type. 

For a fixed symmetric pair $(G_k,\theta_k)$, it is direct to see that by choosing different $\theta_k$-anisotropic Borel subgroups $B_k$, the triple $(G_k,\theta_k,B_k)$ gives isomorphic $\imath$root data. We call any $\imath$root datum in this isomorphism class the \emph{$\imath$root datum associated to $(G_k,\theta_k)$.}

\begin{lemma}\label{lem:pinning}
 Let $(G_k,\theta_k,B_k)$ be a triple as above, and write $T_k=B_k\cap\theta_k(B_k)$. Then there is a pinning $\{x_i,y_i;i\in\I\}$ of $(G_k,T_k,B_k)$, such that $\theta_k(x_i(\xi))=y_{\tau i}(\xi)$, for any $i\in\I$ and $\xi\in k$.
\end{lemma}

\begin{proof}
By \cite{Spr87}*{\S1.5}, there exists a pinning $\{x_i,y_i;i\in\I\}$ such that $\theta_k(x_i(\xi))=y_{\tau i}(-\xi)$, for any $i\in\I$ and $\xi\in k$. This is because $n$ in \cite{Spr87}*{\S1.5} is a representative of the longest element of $W$ for quasi-split cases. Now the lemma is immediate by rescaling.
\end{proof}


We will call such pinning an \emph{anisotropic pinning} of $(G_k,\theta_k,B_k)$. It is easy to see that anisotropic pinning is not unique.
 
\subsubsection{Classifications}\label{sec:class}

Take an $\imath$root datum $(\I,Y,X,(\alpha_i)_{i\in\I},(\coroot_i)_{i\in\I},\tau,\theta)$ of finite type. We follow the same notations as before. There is a well-defined $\A$-algebra automorphism ${_\A\dot{\U}}\rightarrow{_\A\dot{\U}}$, sending $E_i^{(n)}\one_\lambda$ to $F_{\tau i}\one_{\theta\lambda}$, and $F_i^{(n)}\one_\lambda$ to $E_{\tau i}\one_{\theta\lambda}$, for any $i\in\I$, $n\in\BN$, and $\lambda\in Y$. Let $A$ be any commutative ring with unit. It is direct to see that after base change, we have the induced $A$-algebra involution 
\begin{equation}\label{eq:thetaA}
\theta_A:{_A\widehat{\U}}\rightarrow{_A\widehat{\U}}, \text{ which restricts to a group homomorphism } \theta_A:G_A\rightarrow G_A.
\end{equation}

In particular, $\theta_k:G_k\rightarrow G_k$ defines a quasi-split involution for the reductive group $G_k$, with $B_k$ an anisotropic Borel subgroup (\S \ref{sec:regpq}). Then the $\imath$root datum associated to the pair $(G_k,\theta_k)$ is isomorphic to  
the one we start with.

By Lemma~\ref{lem:pinning} and \cite{Spr87}*{\S1.6}, we have the following proposition. 
\begin{prop}\label{prop:cft}
We have a canonical bijection:
\[
\{\text{iso. classes of symmetric pairs $(G_k,\theta_k)$}\}\leftrightarrow\{\text{iso. classes of finite type $\imath$root data}\}.
\]
\end{prop}

In particular, the isomorphic classes of symmetric pairs $(G_k,\theta_k)$ is independent of the field $k$ (provided char $k$ $\neq 2$). 

\subsubsection{$K_k$-orbits on the flag variety}\label{sec:Korbits}

We fix an anisotropic triple $(G_k,\theta_k,B_k)$ and set $K_k=G_k^{\theta_k}$ be the closed subgroup of $G_k$ consisting of $\theta_k$-fixed elements. It is called the \emph{symmetric subgroup} of $G_k$ associated to $\theta_k$. Let $\{x_i,y_i;i\in\I\}$ be an anisotropic pinning of $(G_k,\theta_k,B_k)$. Let $T_k=B_k\cap \theta_k(B_k)$ be the maximal torus, and $W=N(T_k)/T_k$. For any $i\in\I$, we set $n_i=x_i(1)y_i(-1)x_i(1)$.

Let $\CB_k=G_k/B_k$ be the flag variety of $G_k$. Then $K_k$ acts on $\CB_k$ with finitely many orbits by \cite{Spr85}*{Corollary~4.3}. The $K_k$-orbits on $\CB_k$ can be parameterized equivalently using $B_k$-orbits on $G_k/K_k$. The results are essentially in \cite{Spr85}. Since our choice of $(\theta_k, B_k)$ is different than Springer, who considers $B_k$ to be $\theta_k$-stable,  we reformulate the results here.    

Let $\phi: G_k/K_k \rightarrow G_k$, $gK_k \mapsto g \theta(g)^{-1}$ be the $G_k$-equivariant embedding \cite{Spr85}*{Proposition~2.2}, where the $G_k$-action on $G_k$ is given by $h \ast g = hg\theta(h)^{-1}$. Let $S_k = \phi (G_k/K_k)$. Since $\Big(B_k \times \theta_k(B_k)\Big)\text{-double cosets of $G_k$}$ are parameterized by the Weyl group $W$, we can define the composition 
\[
    \Pi: \{B_k\text{-orbits on }  G_k/K_k\} \rightarrow \{\Big(B_k \times \theta_k(B_k)\Big)\text{-orbits on }  G_k\} \rightarrow W.
\]

One sees that $K_kB_k/B_k$ is the unique open $K_k$-orbit on $\CB_k$ by \cite{Spr85} or \cite{Spr87}*{\S1.3}. 

We next construct $K_k$-orbits inductive from the open orbit following  \cite{Spr85}*{Theorem~6.5}. We consider $B_k$-orbits on $S_k \cong G_k/K_k$.  Let $\Omega \subset S_k$ be a $B_k$-orbit via the $\ast$-action. We denote by $P_i$ the standard parabolic subgroup of $G_k$ associated to $i \in I$. We describe $P_i \ast \Omega$ following \cite{Spr85}*{\S6.7}. Let $P_i \ast \Omega = \Omega \cup \Omega'$, where 
\[
     \Omega' = \{u n_i x \theta(un_i)^{-1} \vert x \in \Omega, u \in U_i\}
\]

Let $w=\Pi(\Omega)\in W$. We divided into several cases.
    
    \begin{enumerate}
    \item[(a)$'$]  If $\ell(s_i w \theta(s_i)) = \ell(w) -2$, then $\Omega'$ is the unique open dense $B_k$-orbit in $\overline{P_i \ast \Omega}$. We have $\Omega' =B_kn_i \ast \Omega$. We have  $\dim \Omega' = \dim \Omega +1$.
    
    \item[(a)$''$]  If $\ell(s_i w \theta(s_i)) = \ell(w) + 2$, then $\Omega'$ is a single  $B_k$-orbit in $P_i \ast \Omega$. We have $\Omega' =B_kn_i \ast \Omega$. We have  $\dim \Omega' = \dim \Omega - 1$.
    \end{enumerate}
 
    Now we consider the cases where $s_i w \theta(s_i) = w$. We have an involution $\psi$ on the subgroup $G_i$ generated by $ x_i(k), y_i(k) $ defined by $\psi(g) = n \theta(g) n^{-1}$, where $n$ is any representative of $w$ in $\Omega$. We are reduced to rank one computations.

   \begin{enumerate}
   \item[(b)] If $\psi =\rm{id}$, then $\Omega' = \Omega$.
    
    \item[ (c)] We assume $\psi(x_i(a)) = x_i(-a)$, $\psi(y_i(a)) = y_i(-a)$, $\psi(n_i) = n_i^{-1}$. Note that $ n_i = y_i(-1) x_i(1/2) \psi( y_i(-1) x_i(1/2))^{-1}$ in this case. 
    
    We write  $\Omega_1 = B_k \ast n_i \Omega= B_ky_i(-1) x_i(1/2) \ast  \Omega$ and $\Omega_2 = B_k \ast n_i \Omega \theta(n_i)^{-1}= B_k n_i \ast  \Omega$. We then have $\Omega' = \Omega_1  \cup \Omega_2 $. Note that $\Omega = \Omega_2$ if  $G_i \cong PGL_{2,k}$.
    
     \item[(d)] We assume $\psi(x_i(a)) = y_i(a)$, $\psi(y_i(a)) = x_i(a)$, $\psi(n_i) = n_i^{-1}$. Note that $x_i(1/2) y_i(-1) \psi(x_i(1/2) y_i(-1))^{-1} = n_i$.
    
    Then we write $\Omega_1 = B_k \ast n_i \Omega= B_k x_i(1/2) y_i(-1) \ast  \Omega$ and $\Omega_2= B_k \ast \alpha^\vee_i(-1) n_i  \Omega_1= B_k x_i(-1/2) y_i(1) * \Omega $. Hence we have 
    $\Omega' = \Omega_1 \cup \Omega_2$.  Note that $\Omega_1 = \Omega_2$ if $G_i \cong PGL_{2,k}$.
    
     \end{enumerate}
	 
	\begin{prop}\label{prop:Korbits}
	 (1)  For any $K_k$-orbit on $\CB_k$, one can take a representative $vB_k$, which can be written as products of $n_i^{- 1}$, $  y_{i}(1)x_i(-1/2)$ and $ y_i(-1) x_{i}(1/2) $, for various $i \in \I$, and $v^{-1}\theta_k(v)\in N(T_k)$. 
	 
	 (2) (for quasi-split types) Codimension one $K_k$-orbits are of the form: 
	    \begin{itemize} 
	    \item $\mO_i=K_k n_iB_k/B_k$, for $i\in \I$, with $\tau i\neq i$; (Note that these orbits may not be distinct.)
	    \item $\mO_i^+=K_ky_i(1)B_k/B_k$ and $\mO_i^-=K_ky_i(-1)B_k/B_k$,  for $i\in \I$, with $\tau i=i$. (Note that $\mO_i^+$ and $\mO_i^-$ may not be distinct.) 
	    \end{itemize}
	  
	 (3) The poset of $K_k$-orbits on $\mathcal{B}_k$ only depends on the $\imath$root datum associated to $(G_k,\theta_k)$. (Hence it is independent of the underlying field $k$.) We write $\CO_k = \CO$ whenever it is necessary to emphasize the field.
	\end{prop}
	
	\begin{proof}
	Part (1) and part (2) is immediate from the above discussion. 
	
	We prove part (3). The claim on partial orders follows by \cite{RS90}*{Theorem~4.6 or Theorem~7.11}, since the case analysis above for $P_i \ast \Omega$ is independent of the characteristic of $k$. It suffices to show that the set of $K_k$-orbits on $\mathcal{B}_k$ is independent of the characteristic of $k$.
 
 The description for the representatives in (1) only depends on the Weyl group $W$ and the involution on $W$ induced by $\theta_k$, which can be read from the $\imath$root datum. Two different representatives in (1) may represents the same orbit. It will suffice to show this phenomenon can be verified independent of $k$. 
	
	Let $v$, $v'$ be two different representatives in (1). Write $n=v^{-1}\theta_k(v)$, and $n'=v'^{-1}\theta_k(v')$. Then $n$ and $n'$ belong to $N(T_k)$. Then by the above discussion and (uniqueness part of) Bruhat's lemma, we have: $K_kvB_k=K_kv'B_k$ $\Leftrightarrow$ $n'\in B_k\ast n$ $\Leftrightarrow$ $n'\in T_k \ast n$ $\Leftrightarrow$ $n$ and $n'$ represent the same element in $W$, say $w$, and 
	\begin{equation}\label{eq:tsol}
	   \text{$\exists\; t\in T_k$, such that }w^{-1}(t)\theta_k(t)=n^{-1}n' = t_0\in T_k.
	\end{equation} 
	Recall $Y=\text{Hom}(k^\times, T_k)$, and there is a canonical isomorphism $T_k\cong k\otimes_\BZ Y\cong k\otimes_{\BZ[2^{-1}]}Y_2$, where $Y_2=\BZ[2^{-1}]\otimes_\BZ Y$. Under our construction, $t_0=1\otimes t_0'$, for some $t_0'\in Y_2$, independent of $k$. And $\theta_k$ is obtained from $\theta:Y\rightarrow Y$. Since $\theta_k(n)=n^{-1}$, we have $(\theta w)^2=id:Y\rightarrow Y$. By the same proof of Lemma \ref{le:nop}, the cokernel of the map $id-\theta w:Y\rightarrow Y$ has no odd torsion. Hence condition \eqref{eq:tsol} is equivalent to $\theta(t_0')\in \text{im}(id-\theta w)\subseteq Y_2$. Here we extend $\theta$ and $w$ to $Y_2$. This is clearly independent of the field $k$. We complete the proof.
	\end{proof}

\subsection{Frobenius splittings} \label{sec:algFr}
	Let $k$ be an algebraically closed field of positive characteristic $p$ (possibly $2$ in this section). 
	
	\subsubsection{Definitions}
	By \emph{schemes}, we mean separated schemes of finite type over $k$. 
	\begin{defi}
	Let $ \CX$ be a scheme; then the \emph{absolute Frobenius morphism}
	\[
	F:\CX\longrightarrow \CX
	\]
	is the identity on the underlying space of $\CX$, and the $p$-th power map on the structure sheaf $\mathcal{O}_\CX$.
	\end{defi}
	
	Following \cite[\S1.1.3]{BK05}, a scheme $\CX$ is \emph{Frobenius split} (or simply \emph{split}) if the $\mathcal{O}_\CX$-linear map $F^\#:\mathcal{O}_\CX\longrightarrow F_*\mathcal{O}_\CX$ splits. Namely, there exists a $\mathcal{O}_\CX$-linear map $\varphi\in\text{Hom }(F_*\mathcal{O}_\CX,\mathcal{O}_\CX)$ such that $\varphi\circ F^\#={\rm id}$. Such a map $\varphi$ is called a \emph{splitting} of $X$. Then clearly any $\mathcal{O}_\CX$-linear map $\varphi:F_*\mathcal{O}_\CX\longrightarrow \mathcal{O}_\CX$ is a splitting if and only if $\varphi(1)=1$. 
	
	Let $\CY\subset \CX$ be a closed subscheme of $\CX$, then we say $\CX$ is \emph{split compatibly with $\CY$} (or \emph{$\CY$ is compatibly split}) if there is a splitting $\varphi$ of $\CX$
	such that $\varphi(F_*\mathcal{I}_\CY)\subset \mathcal{I}_\CY$, where $\mathcal{I}_\CY$ is the ideal sheaf of $\CY$. 

We collect the following consequences on the Frobenius splitting schemes by \cite{BK05}*{\S1.2}.

(a) {\em Frobenius split schemes are reduced, and weakly normal.}

(b) {\em Let $\CX$ be a scheme, $\phi$ be a Frobenius splitting of $X$. Then the set of closed subschemes of $\CX$ which are compatibly split under $\phi$ is closed under taking irreducible components, and taking scheme-theoretic intersection. In particular, the scheme-theoretic intersection of two compatibly split subschemes is reduced.}

(c) {\em Let $\CX$ be a proper scheme and $\CY$ be a closed subscheme of $\CX$. Suppose $\CX$ is Frobenius split compatibly with $\CY$. Let $\mathcal{L}$ be an ample line bundle on $\CX$. Then $H^i(\CX,\mathcal{L})=0$, for all $i>0$, the restriction map $H^0(\CX,\mathcal{L})\rightarrow H^0(\CY,\mathcal{L})$ is surjective, and $H^i(\CY,\mathcal{L})=0$, for all $i>0$.}

One can often apply the positive characteristic techniques of Frobenius splitting to certain schemes in characteristic zero as well, e.g., Schubert varieties in characteristic zero. We refer to \cite{BK05}*{\S1.6} for details.

\subsubsection{Flag varieties}

	Let $\CB_k=G_k/B_k$ be the flag variety of a connected reductive group $G_k$. For any $\lambda\in X$, let $k_\lambda$ be the one-dimensional representation of the group $B_k=U_kT_k$, where $U_k$ is the unipotent radical of $B_k$ and $T_k$ is the maximal torus, on which $T_k$ acts via the isomorphism $X\xrightarrow{\sim} \text{Hom}(T_k,k^\times)$ and $U_k$ acts trivially. Let
	\[
	\mathcal{L}_\lambda=G_k\times^{B_k} k_{\lambda}\longrightarrow \CB_k
	\]
	be the associated line bundle over the flag variety $\CB_k$. By definition, the total space $G_k\times ^{B_k} k_{\lambda}$ is the quotient space $(G_k\times k)/B_k$, where $B_k$ acts by $(g,t)\cdot b=(gb^{-1},\lambda(b)t)$, for $g\in G_k$, $b\in B_k$, and $t\in k$. The line  bundle $\mathcal{L}_\lambda$ is (very) ample if and only if $\lambda \in X^{++}$.
	
	The space of global sections $H^0(\lambda)=H^0(\CB_k,\mathcal{L}_\lambda)$ admits a natural (rational) $G_k$-action. It is well-known that $H^0(\lambda)$ is nonzero if and only if $\lambda\in X^+$, in which case  $H^0(\lambda) \cong V_k(\lambda)^*$. Here $V_k(\lambda)$ denotes the Weyl module of $G_k$ of highest weight $\lambda$.
	
Let $\mathcal{L}_\lambda$ be an ample line bundle on $\CB_k$. Define a graded $k$-algebra
	\[
	R_{\mathcal{L}_\lambda}=\bigoplus_{n\geqslant 0} H^0(\mathcal{B}_k,\mathcal{L}_\lambda^n) =\bigoplus_{n\geqslant 0} H^0(\mathcal{B}_k,\mathcal{L}_{n\lambda}).
	\]
	Let $\mathcal{Y}\subset \CB_k$ be a closed subvariety, and $I_{\mathcal{Y},{\mathcal{L}_\lambda^n}}$ be the kernel of the restriction map $H^0(\CB_k,{\mathcal{L}_\lambda^n})\rightarrow H^0(\mathcal{Y},{\mathcal{L}_\lambda^n})$, for any $n\geqslant 0$. Set
	\[
	I_{\mathcal{Y},{\mathcal{L}_\lambda}}=\bigoplus_{n\geqslant 0}I_{\mathcal{Y},{\mathcal{L}_\lambda^n}}.
	\]
	It is an ideal of $R_{\mathcal{L}_\lambda}$. The following lemma follows from \cite{KL2}*{Theorem~6.4 \& 6.7}.
	
	\begin{lem}\label{le:algFrfl}
	Let $\phi : R_{\mathcal{L}_\lambda} \rightarrow R_{\mathcal{L}_\lambda}$ be a graded additive endomorphism such that (a) $\phi(f^pg)=f\phi(g)$, for any $f,g$ in $R_{\mathcal{L}_\lambda}$;  (b) $\phi(1)=1$; (c) $\phi ( I_{\mathcal{Y},{\mathcal{L}_\lambda}}) \subset I_{\mathcal{Y},{\mathcal{L}_\lambda}}$. Then  
	the $\mathcal{B}_k$ is Frobenius split compatibly with $\mathcal{Y}$.
	\end{lem}


	
\section{Symmetric subgroup schemes}\label{sec:SymGpSch}


Let $(\I,Y,X,A,(\alpha_i)_{i\in\I},(\coroot_i)_{i\in\I},\tau,\theta)$ be a quasi-split $\imath$root datum of finite type. Let $(\U, \U^\imath)$ be the quantum symmetric pair associated to this datum. We also fix parameters of $\U^\imath$ as follows: $\varsigma_i=q^{-1}_i$, if $\tau i=i$; $\varsigma_i=1$, if $a_{i,\tau i}=0$; and $\{\varsigma_i,\varsigma_{\tau i}\}=\{1,q_i^{-1}\}$, if $a_{i,\tau i}=-1$.

\subsection{The finiteness conditions}\label{sec:fin}

Recall the isomorphism of $\A$-modules from \eqref{eq:pim}: 
\[
p_{\imath,\lambda}=\pi_\lambda\circ\iota_\lambda:{_\A\dot{\U}^\imath\one_{\overline{\lambda}}}\xrightarrow{\iota_\lambda}{_\A\dot{\U}\one_\lambda}\xrightarrow{\pi_\lambda}{_\A\dot{\U}^-\one_\lambda}.
\]
 Set $s_\lambda=p_{\imath,\lambda}^{-1}\circ\pi_\lambda:{_\A\dot{\U}\one_\lambda}\longrightarrow{_\A\dot{\U}^\imath}\one_{\overline{\lambda}}.$ Then $s_\lambda\circ \iota_\lambda=id$. 

Recall  the canonical basis $\dot{\RB}$ of $\dot{\U}$  and the $\imath$canonical basis $\dot{\RB}^\imath$  of $\dot{\U}^\imath$. 
For any $\lambda\in X^+$, let $\dot{\RB}^\imath_{\overline{\lambda}} =\dot{\RB}^\imath\cap{\dot{\U}^\imath\one_{\overline{\lambda}}}$
and $\dot{\RB}_\lambda=\dot{\RB}\cap{\dot{\U}\one_\lambda}$. 
Then ${\dot{\U}^\imath\one_{\overline{\lambda}}}$ has basis $\dot{\RB}^\imath_{\overline{\lambda}}$ and ${\dot{\U}\one_\lambda}$ has basis $\dot{\RB}_\lambda$. 

We write  $$\iota_\lambda(b)=\sum_{b'\in\dot{\RB}_\lambda}\iota_{\lambda,b,b'}b', \quad \text{for } b\in\dot{\RB}^\imath_{\overline{\lambda}},  \iota_{\lambda,b,b'} \in \A; $$
$$s_\lambda(b)=\sum_{b'\in\dot{\RB}_{\overline{\lambda}}}s_{\lambda,b,b'}b',\quad \text{for }  b\in\dot{\RB}_\lambda,  s_{\lambda,b,b'} \in \A.$$
We view $_\A\dot{\U}^\imath$ as an $\A$-subalgebra of $_\A\widehat{\U}$ via the embedding  $x \mapsto \sum_{\lambda \in X}\iota_\lambda(x) $ as in \S \ref{sec:iqp}. We write 
\[
    b=\sum_{b'\in\dot{\RB}}\iota_{b,b'}b' \in {}_\A\widehat{\U}, \quad \text{for } b\in\dot{\RB}^\imath,  \iota_{b,b'} \in \A.
\]

\begin{lemma}\label{le:finit}
Let $\mu\in X$. Then for any $b'\in \dot{\RB}_\mu$, there are only finitely many $b \in \dot{\RB}^\imath_{\overline{\mu}}$, such that $\iota_{\mu,b,b'}\neq 0$. For any $b'\in\dot{\RB}^\imath_{\overline{\mu}}$, there are only finitely many $b\in\dot{\RB}_\mu$, such that $s_{\mu,b,b'}\neq 0$.

In particular,  for any $b'\in \dot{\RB}$, there are only finitely many $b\in\dot{\RB}^\imath$, such that $\iota_{b,b'}\neq 0$.
\end{lemma}

\begin{proof}

Suppose there exists some $b'\in\dot{\RB}_\mu$, such that there are infinitely many $b\in\dot{\RB}^\imath_{\overline{\mu}}$, such that $\iota_{\mu,b,b'}\neq 0$. 

Take $\mu_1, \mu_2\in X^+$, such that $b'\cdot v_{-\mu_1}^-\otimes v_{\mu_2}^+\neq 0$ in $^\omega L(\mu_1)\otimes L(\mu_2)$. Then by the assumption, we get infinitely many $b\in\dot{\RB}^\imath_{\overline{\mu}}$, such that $b\cdot v_{-\mu_1}\otimes v_{\mu_2}\neq 0$.

Recall from \cite{BW18a}*{Theorem 4.18} that we have $\U^\imath$-module isomorphism $$\mathcal{T}:{ L(\tau\mu_1)}\rightarrow {^\omega L(\mu_1)}$$ sending $v_{\tau\mu_1}^+$ to $v_{-\mu_1}^-$. Then we have $\U^\imath$-module isomorphism $$\mathcal{T}^{-1}\otimes id:{^\omega L(\mu_1)}\otimes L(\mu_2)\rightarrow L(\tau\mu_1)\otimes L(\mu_2)$$ sending $v_{-\mu_1}^-\otimes v_{\mu_2}^+$ to $v_{\tau \mu_1}^+\otimes v_{\mu_2}^+$. Hence via this isomorphism, we get infinitely many $\imath$canonical basis elements $b$, with $b\cdot v_{\tau \mu_1}^+\otimes v_{\mu_2}^+\neq 0$. Via the $\U$-module embedding $L(\tau \mu_1+\mu_2)\rightarrow L(\tau \mu_1)\otimes L(\mu_2)$, sending $v_{\tau\mu_1+\mu_2}^+$ to $v_{\tau \mu_1}^+\otimes v_{\mu_2}^+$, we get infinitely many $\imath$canonical basis elements $b$, such that $b\cdot v_{\tau\mu_1+\mu_2}^+\neq 0$. This contradicts Theorem \ref{thm:stab}. We proved the first claim.

For the second claim, suppose there exists some $b'\in \dot{\RB}^\imath_{\overline{\mu}}$, such that there are infinitely many $b\in\dot{\RB}_\mu$, with $s_{\mu,b,b'}\neq 0$. We take $\mu\in X^+$, such that $b'\cdot v_{\mu}^+\neq 0$ in $L(\mu)$. Then for any $b\in \dot{\RB}_\mu$, since $b\cdot v_{\mu}^+=s_\mu(b)\cdot v_{\mu}^+$, it follows that $b\cdot v_{\mu}^+\neq 0$, whenever $s_{\mu,b,b'}\neq0$. Therefore we get infinitely many canonical basis element $b$, with $b\cdot v_{\mu}^+\neq 0$, which is a contradiction. This proves the second claim.
\end{proof}

\begin{corollary}\label{cor:fimu}
For any $b',b''\in\dot{\RB}^\imath$, write $b'\cdot b''=\sum_{b\in\dot{\RB}^\imath}m_{b',b''}^b b$. Then for any $b\in\dot{\RB}^\imath$, there are only finitely many pair $(b',b'')$ in $\dot{\RB}^\imath$, such that $m_{b',b''}^b\neq 0$.
\end{corollary}

\begin{proof}
For any $b\in\dot{\RB}^\imath$, we choose $\mu\in X^+$, such that $b\cdot v_\mu^+\neq 0$ in $L(\mu)$. Suppose $m_{b',b''}^b\neq 0$, for some $b',b''$ in $\dot{\RB}^\imath$. Then thanks to Theorem \ref{thm:stab}, we have $(b'\cdot b'')\cdot v_\mu^+\neq 0$. In particular, $b'\mid_{L(\mu)}\not\equiv 0$, and $b''\mid_{L(\mu)}\not\equiv 0$. Finally, note that for all but finitely many canonical basis element $c$, we have $c\mid_{L(\mu)}\equiv 0$. Then thanks to Lemma \ref{le:finit}, we only have finitely many such pair $(b',b'')$.
\end{proof}

\begin{remark}
Theorem~\ref{thm:stab} has been established by Watanabe \cite{Wa22} for all real rank one cases as well. Hence Lemma~\ref{le:finit} and Corollary~\ref{cor:fimu} hold for real rank one cases by similar proofs. One can also show all results in \S\ref{sec:SymGpSch} remain valid for real rank one cases. 

All results in \S\ref{sec:SymGpSch} hold if the strong compatibility conjecture in \cite{BW18a}*{Remark~6.18} is true.
\end{remark}

\subsection{The group scheme $\G^\imath$}
Let $A$ be a commutative ring (with unit) viewed as an $\A$-algebra via $q\mapsto 1$. We follow the same notations and conventions as in \S \ref{sec:nocp}. We write $_A\dot{\U}^\imath=A\otimes_\A\big({_\A\dot{\U}^\imath}\big)$. We will abuse the notations, and use $\dot{\RB}^\imath$ to denote the basis of $_A\dot{\U}^\imath$, consisting of the image of the $\imath$canonical basis elements after base change. This should not cause any confusion.

\subsubsection{}\label{sec:Ui1}

Let $_A\widehat{\mathrm{U}}^\imath$ be the $A$-module consisting of formal linear combinations
\[ \sum_{b\in\dot{\RB}^\imath} n_b b, \quad n_b\in A.
\]
Then thanks to Lemma \ref{le:finit}, $_A\widehat{\mathrm{U}}^\imath$ can be naturally embedded into $_A\widehat{\mathrm{U}}$. By the Corollary \ref{cor:fimu}, the product structure on $_A\widehat{\U}^\imath$ is well defined, and the embedding $_A\widehat{\U}^\imath\hookrightarrow{_A\widehat{\U}}$ is compatible with products.

Define $_A\widehat{\U}^{\imath,1}$ be the $A$-module consisting of all the formal linear combinations
\[
\sum_{(a,a')\in \dot{\RB}^\imath\times \dot{\RB}} n_{a,a'}a\otimes a', \quad n_{a,a'}\in A
\]
 Again, by Lemma \ref{le:finit}, we have natural embedding $_A\widehat{\U}^{\imath,1}\hookrightarrow{_A\widehat{\U}^{(2)}}$. Thanks to the Corollary \ref{cor:fimu} and \cite{Lu07}*{Lemma 1.8}, one can define a product structure on $_A\widehat{\U}^{\imath,1}$ in an evident way. Then the embedding is moreover compatible with products.

Since $\U^\imath$ is a right coideal subalgebra of $\U$, it follows that $\WD:{_A\widehat{\U}}\rightarrow{_A\widehat{\U}}$ will restrict to an $A$-algebra homomorphism
\begin{equation*} 
\WD: {_A\dot{\U}^\imath}\longrightarrow {_A\widehat{\U}^{\imath,1}}.
\end{equation*}


\subsubsection{}\label{sec:pro}

Define $\RO_A^\imath$ be the $A$-submodule of $(_A\dot{\U}^\imath)^*=\text{Hom}_A({_A\dot{\U}^\imath},A)$, spanned by the \emph{dual $\imath$canonical basis} $\{b^*\mid b\in \dot{\RB}^\imath\}$, where $b^*$ stands for the $A$-linear maps sending $b'$ to $\delta_{b', b}$, for any $b'\in\dot{\RB}^\imath$.


Recall \S\ref{sec:nocp} that $\RO_A$ is an $A$-submodule of $_A\dot{\U}^*$, spanned by the dual canonical basis. Define the $A$-linear map 
\[
r:\RO_A\rightarrow\RO_A^\imath, \quad b^* \mapsto \sum_{b_1\in\dot{\RB}^\imath}\iota_{b_1,b}b_1^*, \quad \text{for }b\in\dot{\RB}.
\]
 The summation is finite thanks to Lemma \ref{le:finit}. Then it is clear that for any $f\in \RO_A$, we have $r(f)=\hat{f}\mid_{_A\dot{\U}^\imath}$. Here $\hat{f}$ is the extension of $f$ to $_A\widehat{\U}$.


For any $\mu\in X$, let $\RO_{A,\mu}$ be the $A$-submodule of $\RO_A$, spanned by all $b^*$, such that $b\in \dot{\RB}\cap {_A\dot{\U}\one_\mu}$. Then $\RO_A=\bigoplus_{\mu\in X}\RO_{A,\mu}$. Moreover, we have $\RO_{A,\mu}\cdot\RO_{A,\mu'}\subseteq \RO_{A,\mu+\mu'}$. Via the restriction, we have a natural map $\RO_A\rightarrow(_A\dot{\U}\one_\mu)^*$. Under this restriction, the submodule $\RO_{A,\mu}$ is sent injectively to an $A$-submodule of $(_A\dot{\U}\one_\mu)^*$. We identify $\RO_{A,\mu}$ with its image in $(_A\dot{\U}\one_\mu)^*$.

Similarly, for any $\zeta\in X_\imath$, define $\RO_{A,\zeta}^\imath$ be the $A$-submodule of $\RO_A^\imath$, spanned by all the elements $b^*$, such that $b\in \dot{\RB}^\imath\cap {_A\dot{\U}^\imath\one_\zeta}$. Then $\RO_A^\imath=\bigoplus_{\zeta\in X_\imath}\RO_{A,\zeta}^\imath$. Similarly, the submodule $\RO_{A,\zeta}^\imath$ can be identified with an $A$-submodule of $(_A\dot{\U}^\imath\one_\zeta)^*$ via restriction.

By \S\ref{sec:fin} we have $A$-linear maps
\begin{equation*}
\begin{tikzcd}
{_A\dot{\U}\one_\mu}\arrow[r,bend right,"s_\mu"']& {_A\dot{\U}^\imath\one_{\overline{\mu}}} \arrow[l,"\iota_\mu"'], 
\end{tikzcd}, \quad \text{ such that }s_\mu\circ\iota_\mu=id, \quad \text{for any } \mu\in X.
\end{equation*}

Taking linear dual of these maps, and using Lemma \ref{le:finit}, we have
\begin{equation*}
\begin{tikzcd}
\RO_{A,\mu}\arrow[r,"\iota_\mu^*"] & \RO_{A,\overline{\mu}}^\imath \arrow[l,bend left,"s_\mu^*"]
\end{tikzcd}, \quad \text{such that }\iota_\mu^*\circ s_\mu^*=id.
\end{equation*}
In particular, we deduce that $\iota_\mu^*$ is surjective, and $s_\mu^*$ is injective.

It follows from the definition that $\iota_\mu^*(f)=r(f)$, for any $f\in\RO_{A,\mu}$. Hence we deduce 

(a) {\em the map $r:\RO_A\rightarrow\RO_A^\imath$ is surjective.}

\subsubsection{}\label{sec:HoOi}

Recall from \S \ref{sec:nocp} that $(\RO_A,\delta, \sigma,\epsilon)$ is a commutative $A$-Hopf algebra.



\begin{theorem}\label{thm:Hopfi}
The $A$-module $\RO_A^\imath$ has a structure of a commutative $A$-Hopf algebra, such that the surjection $r:\RO_A\rightarrow\RO_A^\imath$ is a Hopf algebra homomorphism. 
\end{theorem}

\begin{proof}
Note that the kernel $I_A$ of $r$ consists of linear forms $f$, such that $\hat{f}\mid_{{{}_A\dot{\U}^\imath}} = 0$. We firstly show that $I_A$ is an ideal. Take any $f\in I_A$, and $g\in \RO_A$. For any $x\in{_A\dot{\U}^\imath}$, we have $\widehat{fg}(x)=f\otimes g\circ(\widehat{\Delta}(x))=0$. The last equality follows from the fact $\widehat{\Delta}({_A\dot{\U}^\imath})\subseteq {_A\widehat{\U}^{\imath,1}}$. Hence $fg$ belongs to $I_A$. Since $\RO_A$ is commutative, $I_A$ is a two-sided ideal.

Since the embedding $_A\dot{\U}^\imath\hookrightarrow{_A\widehat{\U}}$ is compatible with products, it follows that the comultiplication $\delta:\RO_A\rightarrow\RO_A\otimes\RO_A$ induces $\RO_A/I_A\rightarrow \RO_A/I_A\otimes\RO_A/I_A$. Finally, it is direct to check that $\sigma(I_A)\subseteq I_A$, and $\epsilon(I_A)=0$. Hence $\RO_A/I_A$ admits a commutative $A$-Hopf algebra structure. We then endow $\RO_A^\imath$ with the commutative $A$-Hopf algebra structure via the isomorphism $\RO_A^\imath\cong \RO_A/I_A$ and the theorem is proved.
\end{proof}

Since $\RO_A^\imath$ is a commutative Hopf algebra over $A$, the set $\text{Hom}_{A-\text{alg}}(\RO_A^\imath, A)$ of $A$-algebra homomorphisms $\RO_A^\imath\rightarrow A$ preserving 1, has a group structure.

By definition, $\text{Hom}_{A-\text{alg}}(\RO_A^\imath,A)$ is an $A$-submodule of $\RO_A^{\imath,*}$, the $A$-linear duals of $\RO_A^\imath$. There is an ($A$-linear) bijective map  from $\RO_A^{\imath,*}$ to $_A\widehat{\U}^\imath$, given by $\phi\mapsto \sum_{b\in\dot{\RB}^\imath}\phi(b^*)b$. We write $G_A^\imath$ to be the image of $\text{Hom}_{A-\text{alg}}(\RO_A^\imath,A)$ under this bijection. Via the embedding $_A\widehat{\U}^\imath\hookrightarrow{_A\widehat{\U}}$, we view $G_A^\imath$ as a subset of ${_A\widehat{\U}}$. 

Recall the construction for the group $G_A$ from \S \ref{sec:regpq}. The following claim is immediate. 

(a) {\em As subsets of $_A\widehat{\U}$, we have $G_A^\imath={_A\widehat{\U}^\imath}\cap G_A$, which is a subgroup of $G_A$. And the bijection $\text{Hom}_{A-\text{alg}}(\RO_A^\imath,A)\xrightarrow{\sim} G_A^\imath$ is moreover a group isomorphism. }

\begin{definition}
We define $\G^\imath$ as the $\BZ$-group scheme, by setting $\G^\imath(A)=G_A^\imath$, for any $\BZ$-algebra A. We have $\G^\imath \cong Sp\, \RO_{\BZ}^\imath  $ as affine group schemes.
\end{definition}

By \S\ref{sec:pro} (a), $\G^\imath$ is a closed subsgroup scheme of $\G$, where $\G$ denotes the Chevalley group scheme over $\BZ$ associated to (the root data) of $G$.


\subsubsection{}
We next show that $\G^\imath_{\BZ[2^{-1}]}\rightarrow Sp\,\BZ[2^{-1}]$ has reduced geometric fibres.
\begin{proposition}\label{prop:GAi}
Suppose $A$ is an integral domain with characteristic not 2. Then $\RO_A^\imath$ is a reduced $A$-algebra.
\end{proposition}

\begin{proof}
For any $\mu,\mu'\in X$, we have the following commutative diagram by definitions:
\begin{equation*}
    \begin{tikzcd}
    {_A\dot{\U}^\imath}\one_{\overline{\mu+\mu'}} \arrow[r,"\iota_{\mu+\mu'}"] \arrow[d,"\Delta_{\overline{\mu},\overline{\mu'}}"'] & {_A\dot{\U}\one_{\mu+\mu'}} \arrow[r,"\pi_{\mu+\mu'}"] \arrow[d,"\Delta_{\mu,\mu'}"'] & {_A\U^-\one_{\mu+\mu'}} \arrow[d,"\Delta_{\mu,\mu'}"] \\ {_A\dot{\U}^\imath\one_{\overline{\mu}}}\otimes{}{_A\dot{\U}^\imath\one_{\overline{\mu'}}} \arrow[r,"\iota_\mu\otimes\iota_{\mu'}"] & {_A\dot{\U}\one_\mu}\otimes{} {_A\dot{\U}\one_{\mu'}} \arrow[r,"\pi_\mu\otimes\pi_{\mu'}"] & {_A\U^-\one_\mu}\otimes{} {_A\U^-\one_{\mu'}}. 
    \end{tikzcd}
\end{equation*}

Here $\Delta_{\mu,\mu'}$ (resp. $\Delta_{\overline{\mu},\overline{\mu'}}$) stands for the comultiplication restricting to the corresponding weight spaces (resp. $\imath$weight spaces). Then we have the commutative diagram:
\begin{equation}\label{dia:As}
    \begin{tikzcd}
    {_A\dot{\U}\one_{\mu+\mu'}} \arrow[r,"s_{\mu+\mu'}"] \arrow[d,"\Delta_{\mu,\mu'}"'] & {_A\dot{\U}^\imath\one_{\overline{\mu+\mu'}}} \arrow[d,"\Delta_{\overline{\mu},\overline{\mu'}}"]\\ {_A\dot{\U}\one_\mu}\otimes{}{_A\dot{\U}\one_{\mu'}} \arrow[r,"s_\mu\otimes s_{\mu'}"] & {_A\dot{\U}^\imath\one_{\overline{\mu}}}\otimes {_A\dot{\U}^\imath\one_{\overline{\mu'}}}.
    \end{tikzcd}
\end{equation}

By taking the linear dual of the diagram \eqref{dia:As} and restricting to the proper subspaces, we get the commutative diagram:
\begin{equation}\label{dia:OmA}
    \begin{tikzcd}
 \RO_{A,\mu+\mu'} & \RO_{A,\overline{\mu+\mu'}}^\imath \arrow[l,"s_{\mu+\mu'}^*"'] \\ \RO_{A,\mu}\otimes \RO_{A,\mu'} \arrow[u] & \RO_{A,\overline{\mu}^\imath}\otimes\RO_{A,\overline{\mu'}}^\imath \arrow[u] \arrow[l,"s_\mu^*\otimes s_{\mu'}^*"'].
 \end{tikzcd}
\end{equation}

Since $\RO_A^\imath$ is a flat $A$-module, it can be viewed as a subring of $\RO_K^\imath$, where $K$ is the fractional field of $A$. Hence we may assume that $A$ is a field. Set $p=\text{char }A$. Then by our assumption, $p\neq 2$.

Firstly assume $p>0$. Note that for any $\zeta,\zeta'\in X_\imath$, we have $\RO_{A,\zeta}^\imath\cdot\RO_{A,\zeta'}^\imath\subset\RO_{A,\zeta+\zeta'}^\imath$. So $\RO_A^\imath=\bigoplus_{\zeta\in X_\imath}\RO_{A,\zeta}^\imath$ is a $X_\imath$-graded algebra. We firstly prove that a homogeneous element cannot be nilpotent.

Take $f\in \RO_{A,\zeta}^\imath$, $g\in \RO_{A,\zeta'}^\imath$, such that $f\cdot g=0$. Choose $\mu,\mu'\in X$, such that $\overline{\mu}=\zeta$, $\overline{\mu'}=\zeta'$. By the commuting diagram \eqref{dia:OmA}, we deduce that $s_\mu^*(f)\cdot s_{\mu'}^*(g)=0$. Since $\RO_A$ is a integral domain, and $s_\mu^*$, $s_{\mu'}^*$ are injective, we deduce that either $f$ or $g$ is 0. Now suppose $f^n=0$, for some homogeneous element $f$, using the argument inductively, we deduce that $f=0$.

Suppose $f=\sum_{\zeta\in X_\imath} f_\zeta\in \RO_A^\imath$ be a nilpotent element, where $f_\zeta\in\RO_{A,\zeta}^\imath$. Then there exists $r\in \BN$, such that $f^{p^r}=0$. Note that $f^{p^r}=\sum_{\zeta\in X_\imath} f_\zeta^{p^r}$. Element $f_\zeta^{p^r}$ has degree $p^r\zeta$. Since $X_\imath$ has no $p$-torsion by Lemma \ref{le:nop}, we deduce that $f_\zeta^{p^r}$ and $f_{\zeta'}^{p^r}$ have different degrees whenever $\zeta\neq \zeta'$. Hence $f_\zeta^{p^r}=0$, for any $\zeta\in X_\imath$, which implies $f_\zeta=0$, thanks to the above argument. Therefore $f=0$.

The case when $p=0$ either follows from \cite{BK05}*{1.6.5}, or follows from a well-known result by Cartier (which asserts that any finitely generated Hopf algebra over a field with characteristic 0 is reduced).
\end{proof}

\begin{remarks}
The assumption that $\textrm{char}(A)\neq 2$ can not be dropped in general. In the case of of rank 1, one can show that $\RO_A^\imath\cong A[u,v]/(u^2-v^2-1)$, which is non-reduced if $A$ has characteristic 2.
\end{remarks}

\subsection{$G_A^\imath$ as a symmetric subgroup}
Let $A$ be an integral domain with characteristic not 2. Let $k$ be an algebraic closure of the quotient field of $A$. Then characteristic of $k$ in not 2. 
\subsubsection{}

 Recall \eqref{eq:thetaA} the group involution $\theta_A:G_A\rightarrow G_A$. Set $K_A=G_A^{\theta_A}$ be the subgroup of $G_A$ consisting of $\theta_A$ fixed points. It is called the \emph{symmetric subgroup} of $G_A$ corresponding to the involution $\theta_A$.  


Since $B_k$ is a $\theta_k$-anisotropic Borel subgroup of $G_k$, the maximal torus $T_k = B_k \cap \theta_k(B_k)$ contains a maximal $\theta_k$-anisotropic torus. We denote by $K_k^\circ$ the identity component of $K_k$. Then by \cite{Vu74}*{Proposition~7} (cf. \cite{Ri82}*{\S2.9}), we have \begin{equation}\label{eq:TK}
 K_k = T_k^{\theta_k} K^\circ_k.
\end{equation}

\begin{theorem}\label{thm:Oik}
As closed subgroups of $G_k$, we have $G_k^\imath=K_k$. In particular, the coordinate ring of the symmetric subgroup $K_k$ is isomorphic to $\RO_k^\imath$.
\end{theorem}

\begin{proof}
We first show $G_k^\imath\subseteq K_k$. It suffices to show the subalgebra $_k\dot{\U}^\imath \subset {}_k\widehat{\U}$ is fixed by $\theta_k$ pointwise. Since $_k\dot{\U}^\imath\cong k\otimes_\BZ {}_\BZ\dot{\U}^\imath$, and $_\BZ\dot{\U}^\imath$ is a subalgebra of $_\BC\dot{\U}^\imath$, we may assume $k=\BC$. Since $_\BC\dot{\U}^\imath$ is generated by $B_{i,\zeta}$ ($i\in\I$, $\zeta\in X_\imath$), we only need to check $\theta_\BC$ fixes these generators. This is direct, and we leave it for readers.

Next we compare the Lie algebras of $G_k^\imath$ and $K_k$. Let $\mathfrak{g}$ be the Lie algebra of $G$. Recall from \S \ref{sec:liealg} that we have a canonical embedding $\mathfrak{g}\hookrightarrow{}_k\widehat{\U}$. Also recall that $\mathfrak{g}$ has generators (as a Lie algebra)
\[
e_i=\sum_{\lambda\in X}E_i\one_\lambda,\qquad f_i=\sum_{\lambda\in X}F_i\one_\lambda,\qquad h_\mu=\sum_{\lambda\in X}\langle \mu,\lambda\rangle\one_\lambda,
\]
for any $i\in\I$, and $\mu\in Y$. Here we are abusing notations by using the same notations to denote the elements after base change.

The involution $\theta_k$ on $_k\widehat{\U}$ also restricts a (Lie algebra) involution on $\mathfrak{g}$, which is the same as the differentiation of the involution on the group $G_k$. We still write $\theta_k$ to denote the involution on $\mathfrak{g}$.


Let $\mathfrak{k}$ be the Lie algebra of $K_k$. By \cite{Bo91}*{\S9.4}, we have $\mathfrak{k}=\mathfrak{g}^{\theta_k}$. By the triangular decomposition of $\mathfrak{g}$ (cf. \cite{Ko14}*{Lemma~2.7}), it is easy to see that $\mathfrak{k}=\mathfrak{g}^{\theta_k}$ is the subalgebra of $\mathfrak{g}$ generated by $f_i+e_{\tau i}$ ($i\in\I$) and $h_\mu$ ($\mu\in Y^\imath$).

We write $\mathfrak{g}^\imath\subseteq \mathfrak{g}$ to be the Lie algebra of $G^\imath_k$. Then it follows from the definition that $\mathfrak{g}^\imath=\mathfrak{g}\cap{_k\widehat{\U}^\imath}$, as Lie subalgebras of $_k\widehat{\U}$. For any $i\in\I$, $\mu\in Y^\imath$, we have
\[
f_i+e_{\tau i}=\sum_{\zeta\in X_\imath}B_{i,\zeta},\qquad h_\mu=\sum_{\zeta\in X_\imath}\langle \mu,\zeta\rangle \one_\zeta.
\]
Again, we abuse notations here. Then we deduce that $\mathfrak{k}$ is contained in $\mathfrak{g}^\imath$. Combined with $G_k^\imath\subseteq K_k$, it follows that $K_k^\circ$ is contained in $G_k^\imath$. 

Thanks to \eqref{eq:TK}, it remains to show that $T_k^{\theta_k}$ is contained in $G_k^\imath$.
Take any $\xi=\sum_{\lambda\in X}n_\lambda\one_\lambda$ in $T_k^{\theta_k}$. Then $n_\lambda\in k^\times$  and $n_{\lambda'}n_{\lambda''}=n_{\lambda'+\lambda''}$, for any $\lambda',\lambda''\in X$. Moreover we also have $n_\mu=n_{\theta\mu}$, for any $\mu\in X$. Therefore $n_{\mu-\theta\mu}=1$. Hence $n_{\lambda'}=n_{\lambda''}$, whenever $\overline{\lambda'}=\overline{\lambda''}$ in $X_\imath$. For any $\zeta\in X_\imath$, we define $n_\zeta=n_\lambda$, for any $\lambda\in X$ with $\overline{\lambda}=\zeta$. Then we can write $\xi=\sum_{\zeta\in X_\imath}n_\zeta\one_\zeta \in {}_k\widehat{\U}^\imath$. 

We finish the proof now.
\end{proof}

\begin{corollary}
As subgroups of $G_A$, we have $G_A^\imath=K_A$.
\end{corollary}

\begin{proof}
The proof for $G_A^\imath\subseteq K_A$ is the same as the first part of the proof of Theorem \ref{thm:Oik}. For the other side inclusion, we embed all the objects into the ambient space $_k\widehat{\U}$. It follows from definitions that 
\[
K_A\subseteq {_A\widehat{\U}}\cap K_k={_A\widehat{\U}}\cap G_k^\imath={_A\widehat{\U}}\cap{_k\widehat{\U}^\imath}\cap G_k={_A\widehat{\U}^\imath}\cap G_A=G_A^\imath.
\]
We completed the proof.
\end{proof}
 
\subsubsection{}
 For any $\mu\in X^+$, write $V_k(\mu)=k\otimes_\A\big({_\A L(\mu)}\big)$ as before. Then $V_k(\mu)$ admits a $_k\widehat{\U}$-action, and hence a (rational) $G_k$-action by restriction. Via the embedding $_k\dot{\U}^\imath\hookrightarrow{_k\widehat{\U}}$ the space $V_k(\mu)$ also admits a $_k\dot{\U}^\imath$-action.  Since there are only finitely many $\imath$canonical basis elements acting non-trivially on $V_k(\mu)$, $V_k(\mu)$ admits a natural $_k\widehat{\U}^\imath$-action.

\begin{prop}\label{prop:KvsUi}
Let $M$ be a $k$-subspace of $V_k(\mu)$. Then $M$ is stable under the $K_k$-action if and only if it is stable under the $_k\dot{\U}^\imath$-action. 
\end{prop}

\begin{proof}
Suppose $M$ is $_k\dot{\U}^\imath$-stable. Since we  $K_k = G^\imath_k=G_k\cap{_k\widehat{\U}^\imath}$, it is clear $M$ is $K_k$-stable. 


Next suppose $M$ is $G_k^\imath$-stable. Take any $f$ in $_k\dot{\U}^\imath$  and $m$ in $M$. Then $f$ naturally induces a linear form on $\RO_k^\imath$, which we denote by $\tilde{f}$. It is direct to check that $f\cdot m=(id\otimes\tilde{f})\circ\Delta_M(m)$, where $\Delta_M:M\rightarrow M\otimes \RO_k^\imath$ is the comodule morphism corresponding to the $G_k^\imath$ action on $M$ (cf. \cite{Jan03}*{\S2.8}). It follows that $f\cdot m$ belongs to $M$. 
\end{proof}


	\section{Quantum Frobenius splittings}	\label{sec:qFri}
	
In this section, we assume given an $\imath$root datum $(\I,Y,X,A,(\alpha_i)_{i\in\I},(\coroot_i)_{i\in\I},\tau,\theta)$ of arbitrary (quasi-split) type.
	
	\subsection{$\imath$Quantum groups over $\mathcal{A}_2$}
 
	Set $\mathcal{A}_2=S^{-1}\mathcal{A}$, where $S\subseteq\mathcal{A}$ is the multiplicative system  generated by $q^a+q^{-a}$, for all $a\in\mathbb{Z}$. Then $\mathcal{A}_2 \subset \mathbb{Q}(q)$ is a subring containing $\mathcal{A}$. 

        This subsection is devoted to study the $\imath$quantum group over the ring $\A_2$.
     
	\subsubsection{} We define the following  \emph{q-double binomial coefficients} in $\Qq$:
	\begin{align*}
	\LR{m}{2k} &=\frac{[m][m-2]\cdots[m-2k+2]}{[2][4]\cdots[2k]} =\frac{\prod_{s=0}^{k-1} (q^{m-2s} - q^{-m+2s})}{\prod_{s=1}^{k}(q^{2s} - q^{-2s}) }, \text{for } m\in\mathbb{Z}, k\in\mathbb{Z}_{>0};\\
	\LR{m}{0}&=1, \text{for } m\in\mathbb{Z}; \qquad \LR{m}{2k}=0, \text{for }k<0.
	\end{align*}	
We denote by  $\LR{m}{2k}_{q^a}$ if $q$ is replaced by $q^a$ for some $a \in \BZ_{>0}$.  
	
	We collect some basic properties the $q$-double binomial coefficients from the definition ($m,k\in\mathbb{Z}$). 
		\begin{equation}\label{eq:mk1}
		\LR{m+2}{2k}=q^{-2k}\LR{m}{2k}+q^{m-2k+2}\LR{m}{2k-2},
		\end{equation}
			\begin{equation}\label{eq:mk2}
		\LR{m}{2k}=(-1)^k\LR{-m+2k-2}{2k},
		\end{equation}
		\begin{equation}\label{eq:mk3}
		\LR{2m}{2k}=\qbinom{m}{k}_{q^2}.
		\end{equation}

	\begin{lem}\label{lem:DoubleInA2}	
		(a)  Let  $m,k\in\mathbb{Z}$. We have  $\LR{m}{2k}\in\mathcal{A}_2$.
		
		(b)  Let  $m', m'',k\in\mathbb{Z}$. We have 
		\[
		\LR{m'+m''}{2k}=\sum_{k'+k''=k} q^{e(m'k''-m''k')}\LR{m'}{2k'}\LR{m''}{2k''}, \quad \text{for } e = \pm 1.
		\]
	\end{lem}
	\begin{proof}
		We show (a). It suffices to consider the case when $k>0$ and $m$ is odd. Thanks to \eqref{eq:mk1} and \eqref{eq:mk2}, it suffices to show $\LR{-1}{2k} \in \CA_2$. Note that $[2n] = [n] (q^n + q^{-n})$ for $n \in \BZ_{>0}$. We have 
			\begin{align*}
		\LR{-1}{2k}&= (-1)^k \LR{2k-1}{2k} = (-1)^k\frac{[2k-1][2k-3]\cdots[1]}{[2k][2k-2]\cdots [2]}\\&=(-1)^k\frac{[2k]!}{[2]^2[4]^2\cdots[2k]^2}=(-1)^k\qbinom{2k}{k}\prod_{i=1}^k\frac{1}{(q^i+q^{-i})^2} \in  \CA_2.
		\end{align*}		 
%
%

We prove (b). When both $m'$, $m''$ are even, the claim is clear  by \eqref{eq:mk3} and \cite[\S1.3.1 (e)]{Lu93}. 
	For fixed $k$, we may regard (b) as an identity of rational functions in three variables: $q, q^{m'}, q^{m''}$. Since the identity holds for all even $m' $, $m'' $, it must hold as a formal identity in the three variables. This proves the proposition. 
	\end{proof}

	\subsubsection{} \label{sec:qsetup}
	We recall the setup in \S\ref{sec:qFr}.
	Fix an odd number $l>1$. We require that $l$ is relatively prime to all the root length $\epsilon_i$. Let $f_l\in \mathcal{A}$ be the $l$-th cyclotomic polynomial. Set $\mathcal{A}'=\mathcal{A}/(f_l)$, and let $\mathbf{F}$ be the quotient ring of $\mathcal{A}'$. Let $\phi:\mathcal{A}\rightarrow\mathcal{A}'$ be the natural quotient map. Then $\phi$ extends to a ring homomorphism from the local ring $\mathcal{A}_{(f_l)}$ to $\mathbf{F}$. Let $c:\mathcal{A}\rightarrow\mathcal{A'}$ be the ring homomorphism sending $q$ to 1. Then $c$ also extends to $\A_{(f_l)}$.
	
	  Write $\bq_i = \phi (q_i)\in\A '$, for any $i\in\I$. For $m,k\in\mathbb{Z}$, set
	\begin{equation*}
	\lr{m}{2k} =c(\LR{m}{2k})=\frac{m\cdot (m-2)\cdots (m-2k+2)}{2\cdot 4\cdots (2k)}\in \mathbb{Z}[2^{-1}].
	\end{equation*}
	
	The following lemma is an analogue of \cite[Lemma 34.1.2 (c)]{Lu93}.
	
	\begin{lemma}\label{le:qBinomAtUnity}
		For $n,k\in\mathbb{Z}$, write $n=n_0+n_1l$, with $n_0,n_1\in\mathbb{Z}$ such that $n_0\in\{0,2,\cdots,2l-2\}$, and $k=k_0+k_1l$, with $k_0,k_1\in\mathbb{Z}$ such that $k_0\in\{0,1,\cdots, l-1\}$. We have
		\begin{equation*}
		\phi\left(\LR{n}{2k}\right)=\lr{n_1}{2k_1}\phi\left(\LR{n_0}{2k_0}\right).
		\end{equation*}
		\begin{proof}
			Let $n$ be even. Then by Lemma \ref{lem:DoubleInA2}, and \cite[Lemma 34.1.2 (c)]{Lu93}, we have
			\begin{align*}
			    \phi\left(\LR{n}{2k}\right)=\phi\left(\qbinom{n/2}{k}_{q^2}\right)=\phi\left(\qbinom{n_0/2}{k_0}_{q^2}\right)\binom{n_1/2}{k_1}=\phi\left(\LR{n_0}{2k_0}\right)\lr{n_1}{2k_1}.
			\end{align*}
			
			Now suppose $n$ is odd. The lemma holds when $n=l$, since
		\[
		\phi\left(\LR{l}{2k}\right) =
		\begin{cases}
		0, & \text{if } l \nmid k;\\
		\lr{1}{2k/l}, & \text{otherwise}.
		\end{cases}
		\]
		
		Note that $n-l$ is even. By the Lemma \ref{lem:DoubleInA2}, we have
			\begin{align*}
			    \phi\left(\LR{n}{2k}\right)&=\sum_{k'+k''=k}\bq^{nk''}\phi\left(\LR{n-l}{2k'}\LR{l}{2k''}\right)\\&=\sum_{k'+lk''=k}\phi\left(\LR{n-l}{2k'}\right)\lr{1}{2k''}\\&=\phi\left(\LR{n_0}{2k_0}\right)\sum_{0\leqslant k''\leqslant k_1}\lr{n_1-1}{2k_1-2k''}\lr{1}{2k''}\\&=\phi\left(\LR{n_0}{2k_0}\right)\lr{n_1}{2k_1}.
			\end{align*}
		Hence the lemma is proved. 
		\end{proof}
	\end{lemma}

	
	
\subsubsection{}\label{sec:baid}

Recall the construction for \emph{$\imath$divided powers} in \S\ref{sec:iqp}.
%
%
Let $i \in \I$ with $\tau i=i$. We define the  \emph{balanced $\imath$divided powers} for $n \ge 0$:
\[
B_i^{(n)}  = B_{i,\overline{n+1}}^{(n)} = \frac{(B_i- [-n+1]_i )(B_i- [-n+3]_i)  \cdots (B_i- [n-3]_i)(B_i- [n-1]_i)}{[n]_i^!}.
\]
Here $B_i^{(0)} = 1$ by definition.


	\begin{lem} \label{lem:ff}
		For $i\in\I$ with $\tau i=i$, and any $n\in \BZ_{\ge 0}$, we have
		\begin{equation}\label{eq:ff}
		 B_i^{(n)}=\sum_{t\geqslant 0}\LR{-1}{2t}_i B_{i,\overline{n}}^{(n-2t)}, \qquad
		 B_{i,\overline{n}}^{(n)} =\sum_{t\geqslant 0}\LR{1}{2t}_i B_i^{(n-2t)}.
		\end{equation}
	\end{lem}
	
	\begin{proof}
		We prove the first equation by induction on $n$. The second equation can be proved similarly.  The claim is trivial for $n = 0, 1$.  Note that 
		\begin{align*}
		B_i^2B_i^{(n)} &=[n+1]_i[n+2]_iB_i^{(n+2)}+[n+1]_i^2B_i^{(n)},\\
		B_i^2B_{i,\overline{n}}^{(n-2t)} &=[n-2t]_i^2B_{i,\overline{n}}^{(n-2t)} +[n-2t+1]_i[n-2t+2]_iB_{i,\overline{n}}^{(n-2t+2)} .
		\end{align*}
		Suppose the equation holds for $n$. Multiplying $B_i^2$ on both sides of the first equation in \eqref{eq:ff}, we obtain  
		\begin{align*}
		&[n+1]_i[n+2]_iB_i^{(n+2)}\\
		= & -[n+1]_i^2 \sum_{t\geqslant 0} \LR{-1}{2t}_i B^{(n-2t)}_{i,\overline{n}}+ \sum_{t\geqslant 0} \LR{-1}{2t}_i [n-2t]_i^2 B^{(n-2t)}_{i,\overline{n}} \\
		& + \sum_{t\geqslant 0} \LR{-1}{2t}_i [n-2t+1]_i[n-2t+2]_i B^{(n-2t+2)}_{i,\overline{n}} \\
		= & [n+1]_i[n+2]_iB^{(n+2)}_{i,\overline{n}} \\
		+ & \sum_{t\geqslant 0} \LR{-1}{2t+2}_i \Big( ([n-2t]_i^2-[n+1]_i^2)\frac{[2t+2]_i}{[-1-2t]_i}+[n-2t-1]_i[n-2t]_i \Big) B^{(n-2t)}_{i,\overline{n}}\\
		= & [n+1]_i[n+2]_iB^{(n+2)}_{i,\overline{n}} + \sum_{t\geqslant 0} \LR{-1}{2t+2}_i \Big([n+1]_i[n+2]_i \Big) B^{(n-2t)}_{i,\overline{n}}\\
		= &  [n+1]_i[n+2]_i\sum_{t\geqslant 0} \LR{-1}{2t}_i B^{(n+2-2t)}_{i,\overline{n}}.
		\end{align*}
		This finishes the proof.	
		\end{proof}
%
%

\begin{lem}\label{lem:balanced}
		For  $i\in\I$ with $\tau i=i$, and $a,k\in \mathbb{Z}_{\ge 0}$, we have
		\begin{align*} 
		B_i^{(a)}B_i^{(k)}&=\sum_{t\geqslant 0}\qbinom{a+k}{a}_i\prod_{m=1}^t\frac{[a-2m+2]_i[k-2m+2]_i}{[a+k-2m+1]_i[2m]_i}B_i^{(a+k-2t)}\\
		& =\sum_{t\geqslant 0}\qbinom{a+k}{a}_i\frac{\LR{a}{2t}_i \LR{k}{2t}_i}{\LR{a+k-1}{2t}_i}B_i^{(a+k-2t)}
		\end{align*}
		 In particular, we have $B_iB_i^{(n)}=\sum_{t\geqslant 0}[n+1]_i\LR{1}{2t}_iB_i^{(n+1-2t)}$.
	\end{lem}
	
	\begin{proof}
	We prove by induction on $a + k$. The base cases are $a+k \le 2$, which can be checked directly. We write $c_{a,k;t} = \qbinom{a+k}{a}_i\displaystyle \prod_{m=1}^t\frac{[a-2m+2]_i[k-2m+2]_i}{[a+k-2m+1]_i[2m]_i}$.
	 
%
		
		Multiplying $B_i^2$ on both sides of the equation, we obtain, by the induction hypothesis,  that 
		\begin{align*}
		    &([a+1]_i[a+2]_iB_i^{(a+2)}+[a+1]_i^2B_i^{(a)})B_i^{(k)}\\&=\sum_{t\geqslant 0}c_{a,k;t}\Big([a+k+1-2t]_i [a+k+2-2t]_i B_i^{(a+k+2-2t)}+[a+k+1-2t]_i ^2 B_i^{(a+k-2t)}\Big).
		\end{align*}
		
		Therefore we have $B_i^{(a+2)} B_i^{(k)} = \sum_{t\geqslant 0}c'_{a+2,k;t}B_i^{(a+2+k - 2t)} $ such that 
		\[
		 [a+1]_i[a+2]_i c'_{a+2,k;t} = [a+k-2t+1]_i[a+k-2t+2]_ic_{a,k;t}+[2a+k-2t+4]_i[k-2t+2]_ic_{a,k;t-1}
		\]
		
		The right hand side is 
		\begin{align*}
		    &[a+k-2t+1]_i[a+k-2t+2]_i\qbinom{a+k}{a}_i \prod_{m=1}^t\frac{[a-2m+2]_i[k-2m+2]_i}{[a+k-2m+1]_i[2m]_i}\\&+[2a+k-2t+4]_i[k-2t+2]_i\qbinom{a+k}{a}_i \prod_{m=1}^{t-1}\frac{[a-2m+2]_i[k-2m+2]_i}{[a+k-2m+1]_i[2m]_i}\\
		    &=[a+1]_i[a+2]_ic_{a+2,k;t} \Big( \frac{ [a+k-2t+2]_i  [a-2t+2]_i}{ [a+k+2]_i[a+2]_i } + \frac{[2a+k-2t+4]_i  [2t]_i}{[a+k+2]_i  [a+2]_i }\Big) \\
		    &=[a+1]_i[a+2]_ic_{a+2,k;t}.
		\end{align*}
  
	Therefore $c_{a+2,k;t} = c'_{a+2,k;t}$. This completes the proof.
	\end{proof}
	\begin{remark}
	This formula is motivated by the formulas for the structure constants for $\imath$divided powers in \cite{CW22}. 
	\end{remark}
	

 

 \subsubsection{} 
 Define $_{\A_2}\dot{\U}^\imath$ be the $\A_2$-span of the free $\A$-submodule of $_\A\dot{\U}^\imath$ inside $\dot{\U}^\imath$. It is then a free $\A_2$-submodule, as well as an $\A_2$-subalgebra.
For $i\in\I$ with $\tau i=i$, and $\zeta\in X_\imath$, define $'B_{i,\zeta}^{(n)}=B_{i}^{(n)}\one_\zeta\in\dot{\U}^\imath$.  By Lemma~\ref{lem:ff}, we have the following claim.

(a) {\em 
the algebra $_{\A_2}\dot{\U}^\imath$ is generated by $'B_{i,\zeta}^{(n)}$ ($\tau i=i$, $\zeta \in X_\imath$, $n\in\BN$), and $B_{i,\zeta}^{(n)}$ ($\tau i\neq i$, $\zeta \in X_\imath$, $n\in\BN$) as an $\A_2$-algebra. }
 	

		\subsubsection{}
		
	Fix $i\neq j\in\I$ with $\tau i=i$ in this subsection. For any $m,n\in\mathbb{Z}$, with $n>0$ and $e=\pm 1$, we define elements $y_{n,m,e}=y_{i,j;n,m,e}$ in $\Ui$ inductively as follows: for $m<0$, set $y_{n,m,e}=0$, and set $y_{n,0,e}=B_j^{(n)}$; for $m\geqslant0$, $y_{n,m,e}$ are determined by the following formula:
		\begin{align}\label{fo:recursion}
		q_i^{-e(2m+na_{ij})}&B_iy_{n,m,e}-y_{n,m,e}B_i\\&=-[m+1]_iy_{n,m+1,e}+[m+na_{ij}-1]_iq_i^{-e(2m+na_{ij}-1)}y_{n,m-1,e}\notag.
		\end{align}

	 The elements are analogues of Lusztig's higher relations for $\imath$quantum groups. They were first studied in \cite{CLW21}, where explicit formulas were obtained in \cite[\S6.1]{CLW21}. We give a simpler expression in the $\CA_2$-form.
%
%
%
	
	\begin{prop}\label{prop:HigherSerreRe}
		For $m,n\in\mathbb{N}$, with $n>0$, and $e=\pm 1$, we have
		\begin{equation}\label{eq:HigherSerreRe}
		 y_{i,j;n,m,e}=\sum_{\substack{r+s+2t=m\\t\geqslant 0}}(-1)^rq_i^{-e(m+na_{ij}-1)(r+t)}\LR{m+na_{ij}}{2t}_iB_i^{(r)}B_j^{(n)}B_i^{(s)}.
		\end{equation}
	\end{prop}
	
	\begin{proof}
		We prove by induction on $m$. When $m= 0, 1$, one can check this by direct computation.  We check the right hand side of \eqref{eq:HigherSerreRe} satisfies the recursion formula \eqref{fo:recursion}. 
		
		By Lemma \ref{lem:balanced}, we have $
		B_iB_i^{(n)}=\sum_{u\geqslant 0}[n+1]_i\LR{1}{2u}B_i^{(n+1-2u)}$. Hence by induction hypothesis, we have
		 \begin{align*}
		 &B_iy_{n,m,e}\\
		 = &\sum_{\substack{r+s+2t=m\\t\geqslant 0}}\sum_{u\geqslant 0}(-1)^rq_i^{-e(m+na_{ij}-1)(r+t)}[r+1]_i\LR{m+na_{ij}}{2t}_i\LR{1}{2u}_iB_i^{(r+1-2u)}B_j^{(n)}B_i^{(s)}\\
		 = &\sum_{\substack{r+s+2t=m+1\\t\geqslant 0}}\sum_{u\geqslant 0}(-1)^{r+1}q_i^{-e(m+na_{ij}-1)(r+t+u-1)}[r+2u]_i\LR{m+na_{ij}}{2t-2u}_i\LR{1}{2u}_iB_i^{(r)}B_j^{(n)}B_i^{(s)}.
		 \end{align*}
		 	
		 We write $J_{r,s}$ to denote the coefficient before $B_i^{(r)}B_j^{(n)}B_i^{(s)}$. Then
		 \begin{align*}
		 J_{r,s} =&\sum_{u\geqslant 0}(-1)^{r+1}q_i^{-e(m+na_{ij}-1)(r+t+u-1)}(q_i^{er}[2u]_i+q_i^{-2eu}[r]_i)\LR{m+na_{ij}}{2t-2u}_i\LR{1}{2u}_i\\
		 =&(-1)^{r+1}q_i^{-e(m+na_{ij}-1)(r+t-1)}[r]_i\sum_{u\geqslant 0}q_i^{-e(m+na_{ij}+1)u}\LR{m+na_{ij}}{2t-2u}_i\LR{1}{2u}_i\\
		 &+(-1)^{r+1}q_i^{-e(m+na_{ij}-1)(r+t-1)+er}\sum_{u\geqslant 0}q_i^{-e(m+na_{ij}-1)u}[2u]_i\LR{m+na_{ij}}{2t-2u}_i\LR{1}{2u}_i\\
		  \stackrel{(\heartsuit 1)}{=}&(-1)^{r+1}q_i^{-e(m+na_{ij}-1)(r+t-1)-et}[r]_i\LR{m+na_{ij}+1}{2t}_i \\
		 &+(-1)^{r+1}q_i^{-e(m+na_{ij}-1)(r+t-1)+er}\sum_{u\geqslant 0}q_i^{-e(m+na_{ij}-1)u}\LR{m+na_{ij}}{2t-2u}_i\LR{-1}{2u-2}_i\\
		   \stackrel{(\heartsuit 2)}{=}&(-1)^{r+1}q_i^{-e(m+na_{ij}-1)(r+t-1)-et}[r]_i\LR{m+na_{ij}+1}{2t}_i\\
		 &+(-1)^{r+1}q_i^{-e(m+na_{ij}-1)(r+t-1)+e(r+t-m-na_{ij})}\LR{m+na_{ij}-1}{2t-2}_i,  
		 \end{align*}
where both $(\heartsuit 1)$ and $(\heartsuit 2)$ are by Lemma~\ref{lem:DoubleInA2}.		 
		 Hence 
		 \begin{align*}
		 B_iy_{n,m,e}=&\sum_{\substack{r+s+2t=m+1\\t\geqslant 0}}(-1)^{r+1}q_i^{-e(m+na_{ij}-1)(r+t-1)}\\&\cdot\left(q_i^{-et}[r]_i\LR{m+na_{ij}+1}{2t}_i+q_i^{e(r+t-m-na_{ij})}\LR{m+na_{ij}-1}{2t-2}_i\right)B_i^{(r)}B_j^{(n)}B_i^{(s)}.
		 \end{align*}
		
		 Similarly, one can compute
		 \begin{align*}
		 y_{n,m,e}B_i=&\sum_{\substack{r+s+2t=m+1\\t\geqslant 0}}(-1)^rq_i^{-e(m+na_{ij}-1)(r+t)}\\&\cdot\left(q_i^{et}[s]_i\LR{m+na_{ij}+1}{2t}_i+q_i^{e(m+na_{ij}-t-s)}\LR{m+na_{ij}-1}{2t-2}_i\right)B_i^{(r)}B_j^{(n)}B_i^{(s)}.
		 \end{align*}
		
		 Hence one has
		 \begin{align*}
		 q_i^{-e(2m+na_{ij})}B_iy_{n,m,e}-y_{n,m,e}B_i=\sum_{\substack{r+s+2t=m+1\\t\geqslant 0}}(-1)^{r+1}q_i^{-e(m+na_{ij})(r+t)}a_{r,s}B_i^{(r)}B_j^{(n)}B_i^{(s)},
		 \end{align*}
		 where 
		 \begin{align*} a_{r,s}=&(q_i^{-e(s+2t)}[r]_i+q_i^{e(2t+r)}[s]_i)\LR{m+na_{ij}+1}{2t}_i\\&+(q_i^{-e(m+na_{ij}+s-r)}+q_i^{e(m+na_{ij}-s+r)})\LR{m+na_{ij}-1}{2t-2}_i.
		 \end{align*}
		 	
		 On the other hand, by definition, we have 
		 \begin{align*}
		 -[m+1]_iy_{n,m+1,e}+&[m+na_{ij}-1]_iq_i^{-e(2m+na_{ij}-1)}y_{n,m-1,e}\\&=\sum_{\substack{r+s+2t=m+1\\t\geqslant 0}}(-1)^{r+1}q_i^{-e(m+na_{ij})(r+t)}b_{r,s}B_i^{(r)}B_j^{(n)}B_i^{(s)},
		 \end{align*}
		 where
		 $ b_{r,s}=[m+1]_i\LR{m+na_{ij}+1}{2t}_i-q_i^{e(r-s)}[m+na_{ij}-1]_i\LR{m+na_{ij}-1}{2t-2}_i$.
		 
		It remains to verify $a_{r,s} = b_{r,s}$. Since $m+1=r+s+2t$, we have 
		 $$[m+1]_i=q_i^{e(r+2t)}[s]_i+q_i^{-e(s+2t)}[r]_i+q_i^{e(r-s)}[2t]_i.$$
		 Hence 
		 \begin{align*}
		 a_{r,s}-b_{r,s}=&q_i^{e(r-s)}(q_i^{-e(m+na_{ij})}+q_i^{e(m+na_{ij})}+[m+na_{ij}-1]_i)\LR{m+na_{ij}-1}{2t-2}_i\\&-q_i^{e(r-s)}[2t]_i\LR{m+na_{ij}+1}{2t}_i\\=&q_i^{e(r-s)}[m+na_{ij}+1]_i\LR{m+na_{ij}-1}{2t-2}_i-q_i^{e(r-s)}[2t]_i\LR{m+na_{ij}+1}{2t}_i\\=&0.
		 \end{align*}
		 The proposition is proved.
	\end{proof}
	
It follows from \cite[Theorem 6.3]{CLW21} that $y_{i,j;n,m,e}=0$, whenever $m>-na_{ij}$, $e=\pm 1$. Letting $n=1$, $m=1-a_{ij}$, we have
	\begin{equation}\label{eq:iSerreNew}
	\sum_{\substack{r+s+2t=1-a_{ij}\\t\geqslant 0}}(-1)^r\LR{1}{2t}_iB_i^{(r)}B_jB_i^{(s)}=0.
	\end{equation}
	

 
	\subsubsection{} \label{subsec:rel}
	Let $k$ be any field of characteristic zero. We view $k$ as an $\A$-module on which $q \mapsto 1$.
	Write $\udoti=k\otimes_\A (_{\A}\dot{\U}^\imath)$. For any $i\in \I$, $\zeta\in X_\imath$ and $n\in\mathbb{N}$, we denote by $b_{i,\zeta}^{(n)}$ the image of $B_{i,\zeta}^{(n)}$. It is clear that all     $b_{i,\zeta}^{(n)}$ ($i\in \I$, $\zeta\in X_\imath$, $n\geqslant 1$) is generated by $b_{i,\zeta}$ ($i\in\I$, $\zeta\in X_\imath$). Hence the elements $b_{i,\zeta}$ ($i\in\I$, $\zeta\in X_\imath$), $\one_\zeta$ ($\zeta\in X_\imath$) generate $\udoti$.  
	 
		The algebra $\udoti$ is generated by elements $\one_\zeta$ ($\zeta\in X_\imath$), $b_{i,\zeta}$ ($i\in\I$, $\zeta\in X_\imath$) which subject to the following relations ($i\neq j\in\I$, $\zeta,\zeta'\in X_\imath$): 
		\begin{align}
		&\quad b_{i,\zeta}\one_{\zeta'}=\delta_{\zeta,\zeta'}b_{i,\zeta}, \qquad \one_{\zeta'}b_{i,\zeta}=\delta_{\zeta-\overline{\alpha_i} \zeta'}b_{i,\zeta}, \quad \one_{\zeta'} \one_{\zeta} = \delta_{\zeta, \zeta'} \one_{\zeta}, \label{re:wt}\\
		&\sum_{r+s=1-a_{ij}}(-1)^rb_{i,\zeta-s\overline{\alpha_i}-\overline{\alpha_j}}^{(r)}b_{j,\zeta-s\overline{\alpha_i}}b_{i,\zeta}^{(s)}=0, \quad \text{if } j \neq \tau i \neq i, \label{re:iSerre}\\
		&\!\!\!\sum_{r+s=1-a_{i,\tau i}}\!\!\!(-1)^rb_{i,\zeta-s\overline{\alpha_i}-\overline{\alpha_{\tau i}}}^{(r)}b_{\tau i,\zeta-s\overline{\alpha_i}}b_{i,\zeta}^{(s)}=\left\{\begin{array}{ll}
		\langle \coroot_i-\coroot_{\tau i}, \zeta\rangle\one_\zeta, &  \text{if }a_{i,\tau i}=0,\\
		-2b_{i,\zeta}, & \text{if }a_{i,\tau i}=-1,\\
		0,  & \text{if }a_{i,\tau i}\leqslant -2, 
		\end{array}\right. \label{re:iwithtaui} \text {if }\tau i\neq i, \\
			&\sum_{\substack{r+s+2t=1-a_{ij}\\t\geqslant 0}}(-1)^r\lr{1}{2t}{'b_{i,\zeta-s\overline{\alpha_i}-\overline{\alpha_j}}^{(r)}}b_{j,\zeta-s\overline{\alpha_i}}{'b_{i,\zeta}^{(s)}}=0, \quad \text{if } \tau i = i.\label{re:iSerre2}
		\end{align}
	 The presentation was obtained in  \cite[Theorem~3.1]{CLW18}. Passing to the modified form is straightforward (cf. \cite{Lu93}*{\S31.1.3}). We replace the relation \cite[Theorem~3.1(3.9)]{CLW18} by the equivalent relation \eqref{re:iSerre2} thanks to \eqref{eq:iSerreNew}.
	
%


\subsection{Main theorem}\label{sec:iFr} Retain the setup in \S \ref{sec:qsetup}. Let $\CA'_2 = \CA_2 / (f_l) \supset \CA'$.  Define
	\begin{align*}  
	_{\mathcal{A}'}\dot{\mathfrak{U}}^{\imath}&={\mathcal{A}'}\otimes_c ({_{\mathcal{A}}\dot{\mathrm{U}}^\imath}),\quad 
	_{\mathcal{A}'}\dot{\mathrm{U}}^\imath=\mathcal{A}'\otimes_\phi (_\mathcal{A}\dot{\mathrm{U}}^\imath),\\
	_{\mathcal{A}_2'}\dot{\mathfrak{U}}^{\imath}&={\mathcal{A}_2'}\otimes_c ({_{\mathcal{A}}\dot{\mathrm{U}}^\imath}),\quad 
	_{\mathcal{A}'_2}\dot{\mathrm{U}}^\imath=\mathcal{A}_2'\otimes_\phi (_\mathcal{A}\dot{\mathrm{U}}^\imath),\\
	_{\mathbf{F}} \dot{\mathfrak{U}}^{\imath}&={{\mathbf{F}}}\otimes_c ({_{\mathcal{A}}\dot{\mathrm{U}}^\imath}),\quad 
	{_{\mathbf{F}}}\dot{\mathrm{U}}^\imath={{\mathbf{F}}}\otimes_\phi (_\mathcal{A}\dot{\mathrm{U}}^\imath).
	\end{align*}

	Then we have $_{\mathcal{A}'}\dot{\mathfrak{U}}^{\imath}\subseteq {_{\mathcal{A}_2'}\dot{\mathfrak{U}}^{\imath}}\subseteq {_{\mathbf{F}}\dot{\mathfrak{U}}^{\imath}}$, and $_{\mathcal{A}'}\dot{\U}^{\imath}\subseteq {_{\mathcal{A}_2'}\dot{\U}^{\imath}}\subseteq {_{\mathbf{F}}\dot{\U}^{\imath}}$.
	
	For any $i\in\I$, $n\in\mathbb{N}$, and $\zeta\in X_\imath$, we use $\mathfrak{b}_{i,\zeta}^{(n)}$ (resp. $\mathrm{B}_{i,\zeta}^{(n)}$) to denote the image of $B_{i,\zeta}^{(n)}$ in $_{\mathcal{A}'}\dot{\mathfrak{U}}^{\imath}$ (resp. $_{\mathcal{A}'}\dot{\mathrm{U}}^\imath$). For $i=\tau i$, set $'\mathfrak{b}_{i,\zeta}^{(n)}$ (resp. $'\mathrm{B}_{i,\zeta}^{(n)}$) be the image of $'B_{i,\zeta}^{(n)}$ in $_{\mathcal{A}'_2}\dot{\mathfrak{U}}^\imath$ (resp. $_{\mathcal{A}'_2}\dot{\mathrm{U}}^\imath$). We state the main theorem of this section.
	
	\begin{theorem}\label{thm:iFr}
		There exists a unique $\mathcal{A}_2'$-algebra homomorphim
		$\ifr: {_{\mathcal{A}_2'}\dot{\mathfrak{U}}^{\imath}}\longrightarrow {_{\mathcal{A}_2'}\dot{\mathrm{U}}^\imath},$
		such that: for $\tau i\neq i$, $n\in\mathbb{N}$, we have
		$\ifr(\mathfrak{b}_{i,\zeta}^{(n)})=\mathrm{B}_{i,l\zeta}^{(nl)}$;
		for $\tau i=i$, $n\in\mathbb{N}$, we have 
		$\ifr('\mathfrak{b}_{i,\zeta}^{(n)})={'\mathrm{B}_{i,\zeta}^{(nl)}}.$
	\end{theorem}
	
	The uniqueness is clear. It suffices to prove the existence over the field of fraction $\mathbf{F}$. This theorem will be proved in \S\ref{subsec:pf}. Let us assume this theorem for now.  

	\begin{corollary}\label{cor:iFr}
	For $\tau i=i$, we have 
		\[
		\ifr(\mathfrak{b}_{i,\zeta}^{(n)}) =
			\begin{cases} \mathrm{B}_{i,l\zeta}^{(ln)}, &\text{if } \langle \coroot_i,\zeta\rangle  \neq n \text{ {\rm mod }}2;\\
			\sum_{t=0}^{(l-1)/2}\phi\left(\qbinom{(l-1)/2}{t}_{q_i^2}\right)\mathrm{B}_{i,l\zeta}^{(nl-2t)}, &\text{if } \langle \coroot_i,\zeta\rangle  = n \text{ {\rm mod }}2.
		\end{cases}
		\]
		In particular, the morphism $\ifr$ restricts to an $\A'$-algebra homomorphism $\ifr: {_{\mathcal{A}'}\dot{\mathfrak{U}}^{\imath}}\longrightarrow {_{\mathcal{A}'}\dot{\mathrm{U}}^\imath}$. 
	\end{corollary}
	
	\begin{remarks}
	Even though we do not need to consider the localization $\CA'_2$ in light of Corollary~\ref{cor:iFr}, it seems impossible to obtain the formulas without Theorem~\ref{thm:iFr}. For applications of Frobenius splittings, we only consider fields of characteristic not $2$. So we believe the localization in Theorem~\ref{thm:iFr} is conceptual. 
	\end{remarks}
	
	\begin{proof}
		If $\langle \coroot_i,\zeta\rangle \neq n$ {\rm mod} $2$, then $\mathfrak{b}_{i,\zeta}^{(n)}={'\mathfrak{b}_{i,\zeta}^{(n)}}$  and $\mathrm{B}_{i,\zeta}^{(n)}={'\mathrm{B}_{i,\zeta}^{(n)}}$. Then the identity follows directly.
		
		Suppose $\langle \coroot_i,\zeta\rangle = n$ {\rm mod} $2$. By Lemma \ref{lem:ff}  and Theorem \ref{thm:iFr}, we have
		\begin{align*}
		\ifr(\mathfrak{b}_{i,\zeta}^{(n)})&=\ifr(\sum_{t\geqslant 0}\lr{1}{2t}{'\mathfrak{b}_{i,\zeta}^{(n-2t)}})\\
		&=\sum_{t \geqslant 0}\lr{1}{2t}{'\mathrm{B}_{i,l\zeta}^{(nl-2tl)}}\\
		&=\sum_{t\geqslant 0}\lr{1}{2t}\sum_{s\geqslant 0}\phi\left(\LR{-1}{2s}_i\right)\mathrm{B}_{i,l\zeta}^{(nl-2tl-2s)}\\
		&=\sum_{k\geqslant 0}\sum_{\substack{tl+s=k\\ t\geqslant 0, s\geqslant 0}}\lr{1}{2t}\phi\left(\LR{-1}{2s}_i\right)\mathrm{B}_{i,l\zeta}^{(nl-2k)}.
		\end{align*}
		By Lemma \ref{le:qBinomAtUnity}, $\phi\left(\LR{l}{2k}_i\right)=\lr{1}{2k/l}$ if $l \mid k$, and $\phi\left(\LR{l}{2k}_i\right) = 0$ if $l \nmid  k$. Hence we have 
		\begin{equation*}
		\sum_{\substack{tl+s=k\\ t\geqslant 0, s\geqslant 0}}\lr{1}{2t}\phi\left(\LR{-1}{2s}_i\right)=\sum_{t'+s=k}\bq_i^{t'}\phi\left(\LR{l}{2t'}_i\LR{-1}{2s}_i\right)=\phi\left(\LR{l-1}{2k}_i\right).
		\end{equation*}
		Note that $\phi\left(\LR{l-1}{2k}_i\right)= 0$, unless $0\leqslant k\leqslant (l-1)/2$. Hence we have 
		\begin{equation*}
		\ifr(\mathfrak{b}_{i,\zeta}^{(n)})=\sum_{k=0}^{(l-1)/2}\phi\left(\LR{l-1}{2k}_i\right)\mathrm{B}_{i,l\zeta}^{(nl-2k)}=\sum_{k=0}^{(l-1)/2}\phi\left(\qbinom{(l-1)/2}{k}_{q_i^2}\right)\mathrm{B}_{i,l\zeta}^{(nl-2k)}.
		\end{equation*}
		We finish the proof.
	\end{proof}


	\subsection{Proof of Theorem \ref{thm:iFr}}\label{subsec:pf}
	For $i\in\I$, $\zeta\in X_\imath$, we define 
\begin{align*}	
	\ifr: \cAtUi &\rightarrow \pAtUi,\\
	 '\fb_{i,\zeta} &\mapsto {'\B_{i,l\zeta}^{(l)}},\quad  \text{if } \tau i =i;\\
	 \fb_{i,\zeta} &\mapsto  {\B_{i,l\zeta}^{(l)}},  \quad \text{if } \tau i \neq i.
\end{align*}
 
	We show $\ifr$ is a well-defined algebra homomorphism satisfying Theorem~\ref{thm:iFr}. This proof is organized as follows:
		\begin{itemize}
			\item we check the relation \eqref{re:wt} is preserved under $\ifr$ (this is trivial and will be skipped);
			\item we show   $\ifr(\fb_{i,\zeta}^{(a)})=\B_{i,l\zeta}^{(al)}$ if $\tau i \neq i$ in \S\ref{subsec:pf1};
			\item we show $\ifr('\fb_{i,\zeta}^{(a)})={}'\B_{i,l\zeta}^{(al)}$ if $\tau i = i$ in \S\ref{subsec:pf2}
			\item we check the relation \eqref{re:iwithtaui} is preserved under $\ifr$ in \S\ref{subsec:pf3};
			\item we check the relation \eqref{re:iSerre} is preserved under $\ifr$ in \S\ref{subsec:pf4}; 
			\item we check the relation \eqref{re:iSerre2} is preserved under $\ifr$ in \S\ref{subsec:pf5};
%
	\end{itemize}
	
	\subsubsection{} \label{subsec:pf1}
	Let $\zeta\in X_\imath$ and $ i\in \I$ be such that  $\tau i\neq i$. We have
	\begin{equation*}
	    \ifr(\fb_{i,\zeta}^{(a)})=\frac{(\B_i^{(l)})^a}{a!}\one_{l\zeta}=\frac{\phi\left({[al]_i!}/{([l]_i!)^a}\right)}{a!}\B_i^{(al)}\one_{l\zeta}=\B_{i,l\zeta}^{(al)}.
	\end{equation*}
	
	\subsubsection{} \label{subsec:pf2} 
	Let $\zeta\in X_\imath$ and $ i\in \I$ be such that  $\tau i =  i$.  We proceed by induction on $a$. Thanks to Lemma \ref{lem:balanced}, we have 
	\begin{equation*}
	    '\fb_{i,\zeta-a \overline{\alpha_i}}{'\fb_{i,\zeta}}^{(a)}=(a+1)\sum_{u\geqslant 0}\lr{1}{2u}{'\fb_{i,\zeta}}^{(a+1-2u)},
	\end{equation*}
	and
	\begin{equation*}
	    '\B_{i,l\zeta-al \overline{\alpha_i}}^{(l)}{'\B}_{i,l\zeta}^{(al)}=\phi\left(\qbinom{(a+1)l}{l}_i\prod_{m=1}^u\frac{[al-2m+2]_i[l-2m+2]_i}{[al+l-2m+1]_i[2m]_i}\right){'\B}_{i,l\zeta}^{(al+l-2u)}.
	\end{equation*}
	
	It follows by direct computation that 
	\begin{equation*}
	    \phi\left(\qbinom{(a+1)l}{l}_i\prod_{m=1}^u\frac{[al-2m+2]_i[l-2m+2]_i}{[al+l-2m+1]_i[2m]_i}\right)
	    = 
	    \begin{cases}
	    	(a+1)\lr{1}{2u/l}, &\text{if } l\mid u;\\
		0, &\text{otherwise}.
	    \end{cases}
	\end{equation*}
	Therefore $\ifr('\fb_{i,\zeta}^{(a)})={}'\B_{i,l\zeta}^{(al)}$.
	
	\subsubsection{} \label{subsec:pf3} 
	
	Let $\zeta\in X_\imath$, $i\in\I$ with $\tau i\neq i$. We write $\alpha=-a_{i,\tau i}=-a_{\tau i,i}$. 
	 
	\begin{lemma}\label{lem:rel3}
		For $a,b\geqslant 0$, with $0\leqslant a+b<l$, we have the following identity in $_{\A'}\U^-$:
		\begin{equation*}
	\sum_{\substack{r+s=1+\alpha\\r,s\geqslant 0}}(-1)^rF_i^{(rl-a)}F_{\tau i}^{(l-a-b)}F_i^{(sl-b)}=0.
		\end{equation*}
	\end{lemma} 
		\begin{proof}
		The identity is trivial for $\alpha=0$. If $a=b=0$, the identity follows by \cite[35.2.3]{Lu93}. 
	
	We assume $\alpha \ge 1$ and $a > 0$ (the case when $ b > 0$ is similar). Then we have $\alpha l-b>\alpha(l-a-b)$. Then by the higher Serre relation \cite[\S7.1.1\&Proposition~7.1.5]{Lu93}, we have
		\[
		    \sum_{u+v=\alpha l-b}(-1)^u\bq_i^{(\alpha(a+b) - b -1)u}F_i^{(u)}F_{\tau i}^{(l-a-b)}F_i^{(v)}=0.
		\]
		
For any $ u\ge 0$, we write $l-a+u = u_a + u_1l$ for some $u_1 \in \BZ$ and $0\le u_a \le l-1$. Since $0 < l-a < l$, by \cite[Lemma~34.1.2]{Lu93}, we have 
\[
	 F^{(l-a)}_i F^{(u)}_i =\phi\left( \qbinom{l-a+u}{l-a}_i \right)F^{(l-a+u)}_i = \binom{u_a}{l-a} F^{(l-a+u)}_i.
\]
Therefore, we have 
	\begin{align*}
		0 =   &F^{(l-a)}_i \left(  \sum_{u+v=\alpha l-b}(-1)^u\bq_i^{(\alpha(a+b) - b -1)u}F_i^{(u)}F_{\tau i}^{(l-a-b)}F_i^{(v)}\right) \\
		 = &\sum_{u+v=\alpha l-b}(-1)^u\bq_i^{(\alpha(a+b) - b -1)u} \binom{u_a}{l-a}  F_i^{(l-a+u)}F_{\tau i}^{(l-a-b)}F_i^{(v)}\\
		 = & \sum_{u+v=\alpha l-b, u_a = l-1}(-1)^u\bq_i^{(\alpha(a+b) - b -1)u} \binom{l-1}{l-a}  F_i^{(l-a+u)}F_{\tau i}^{(l-a-b)}F_i^{(v)} \\
		& + \sum_{u+v=\alpha l-b, u_a = l-2}(-1)^u\bq_i^{(\alpha(a+b) - b -1)u} \binom{l-2}{l-a}  F_i^{(l-a+u)}F_{\tau i}^{(l-a-b)}F_i^{(v)} \\
		& + \cdots \cdots \\
		 = & \sum_{u+v=\alpha l-b, u_a = l-1}(-1)^u\bq_i^{(\alpha(a+b) - b -1)u} \binom{l-1}{l-a}  F_i^{(rl-1)}F_{\tau i}^{(l-a-b)}F_i^{(sl -b-a+1)} \\
		& + \sum_{u+v=\alpha l-b, u_a = l-2}(-1)^u\bq_i^{(\alpha(a+b) - b -1)u} \binom{l-2}{l-a}  F_i^{(rl-2)}F_{\tau i}^{(l-a-b)}F_i^{(sl-b-a+2)} \\
		& + \cdots \cdots .
	\end{align*}
	Note that $ \binom{u_a}{l-a} = 0$ if $u_a < l-a$.  The lemma follows by induction on $a$ now. 
	\end{proof}
	
	We show  the following equalities hold  in $\fpdotU^\imath$:
	\begin{equation}
	\sum_{r+s=1+\alpha}(-1)^r\B_{i,*}^{(rl)}\B_{\tau i,*}^{(l)}\B_{i,l\zeta}^{(sl)}=\left\{\begin{array}{ll}
	\langle \coroot_i-\coroot_{\tau i}, \zeta\rangle\one_{l\zeta} &  \text{ if }\alpha=0\\
	-2\B_{i,l\zeta}^{(l)} & \text{ if }\alpha=1\\
	0  & \text{ if }\alpha\geqslant 2 
	\end{array}\right. \label{eq:checkiwithtaui}
	\end{equation}
	  Here $*$ on the subscripts stand for the appropriate elements in $X_{\imath}$. 
	
	Take $\lambda\in X$ such that $\bar{\lambda}=\zeta$. Recall \eqref{eq:pimA} the $\A$-module isomorphism   $p_\imath:{_\A\dot{\U}^\imath}\one_{l\zeta}\rightarrow  {_\A\dot{\U}^-}\one_{l\lambda}$. Hence it induces an isomorphism $p_\imath': {_{\A'}\dot{\U}^\imath}\one_{l\zeta}\rightarrow  {_{\A'}\dot{\U}^-}\one_{l\lambda}$. 
	 
	 For any $n\in\BN$, we have the following identity in $\U$:
	 \begin{equation}\label{eq:spandi}
    B_i^{(a)}=\sum_{t=0}^aq_i^{t(a-t)-(t(t-1)/2)\cdot \alpha}\varsigma_i^tF_i^{(a-t)}E_{\tau i}^{(t)}\Tilde{K}_i^{-t}
\end{equation}

For any $a,b\in\BN$, and $\mu\in X$, we have the following identity in $_\A\dot{\U}$ by \cite{Lu93}*{23.1.3}:
\begin{equation}\label{eq:quoUp}
\begin{split}
    E_i^{(a)}F_i^{(b)}\one_\mu &=\sum_{t\geqslant 0}\qbinom{a-b+\langle \coroot_i,\mu\rangle}{t}_iF_i^{(b-t)}E_i^{(a-t)}\one_\mu\\ &\equiv \qbinom{a-b+\langle \coroot_i,\mu\rangle}{a}_iF_i^{(b-a)}\one_\mu\quad \text{mod }{_\A\dot{\U}}{_\A\U}^+. 
\end{split}
\end{equation}
As usual, we understand $F_i^{(n)}=0$ if $n<0$.

Hence in $_\A\U^-\one_\lambda$ (which is viewed as the quotient ${_\A\dot{\U}\one_\lambda}/{_\A\dot{\U}}{_\A\U}^+\one_\lambda$), we have:
\begin{align*}
    &p_\imath(\sum_{r+s=1+\alpha}(-1)^rB_{i,*}^{(rl)}B_{\tau i,*}^{(l)}B_{i,l\zeta}^{(sl)})=\sum_{r+s=1+\alpha}(-1)^rB_{i}^{(rl)}B_{\tau i}^{(l)}F_{i}^{(sl)}\one_{l\lambda}\\
    &\stackrel{(\heartsuit 1)}{=}\sum_{r+s=1+\alpha}(-1)^rB_{i}^{(rl)}\sum_{b=0}^lq_i^{b(l-b)-b(b-1)/2\cdot\alpha}\varsigma_{\tau i}^bF_{\tau i}^{(l-b)}E_i^{(b)}\tilde{K}_{\tau i}^{-b}F_{i}^{(sl)}\one_{l\lambda}\\
    &\stackrel{(\heartsuit 2)}{=}\sum_{r+s=1+\alpha}(-1)^rB_{i}^{(rl)}\sum_{b=0}^lq_i^{b(l-b)-b(b-1)/2\cdot\alpha-al(\langle \coroot_{\tau i},\lambda\rangle-s\alpha)}\varsigma_{\tau i}^b\\
    &\cdot \qbinom{b-sl+\langle \coroot_i,l\lambda\rangle}{b}_iF_{\tau i}^{(l-b)}F_i^{(sl-b)}\one_{l\lambda}\\
    &\stackrel{(\heartsuit 3)}{=}\sum_{r+s=1+\alpha}(-1)^r\sum_{a=0}^{rl}q_i^{a(rl-a)-a(a-1)/2\cdot \alpha}\varsigma_i^aF_i^{(rl-a)}E_{\tau i}^{(a)}\tilde{K}_i^{-a}\\
    &\cdot \sum_{b=0}^lq_i^{b(l-b)-b(b-1)/2\cdot\alpha-al(\langle \coroot_{\tau i},\lambda\rangle-s\alpha)}\varsigma_{\tau i}^b\qbinom{b-sl+\langle \coroot_i,l\lambda\rangle}{b}_iF_{\tau i}^{(l-b)}F_i^{(sl-b)}\one_{l\lambda}\\
    &\stackrel{(\heartsuit 4)}{=}\sum_{\substack{0\leqslant a+b\leqslant l\\a,b\geqslant 0}}\sum_{\substack{r+s=1+\alpha\\r,s\geqslant 0}}(-1)^rq_i^{-(a+b)^2+\alpha(a+b)(a+b-1)/2+lN} \\&\cdot \varsigma_i^a\varsigma_{\tau i}^b\qbinom{a+b-l+(sl-b)\alpha+\langle \coroot_{\tau i},l\lambda\rangle}{a}_i\qbinom{b-sl+\langle \coroot_i,l\lambda\rangle}{b}_i\notag
	\\&\cdot  F_i^{(rl-a)}F_{\tau i}^{(l-a-b)}F_i^{(sl-b)}\one_{l\lambda}.\notag
\end{align*}
Here $N$ is some integer. The equality $(\heartsuit 1)$ and $(\heartsuit 3)$ follow from \eqref{eq:spandi} The equality $(\heartsuit 2)$ and $(\heartsuit 4)$ follow from \eqref{eq:quoUp}.

	Therefore in $_{\A'}\U^-\one_{l\lambda}$, we have
	\begin{align*}
	&p_\imath'(\sum_{r+s=1+\alpha}(-1)^r\B_{i,*}^{(rl)}\B_{\tau i,*}^{(l)}\B_{i,l\zeta}^{(sl)})
	=\sum_{\substack{0\leqslant a+b\leqslant l\\a,b\geqslant 0}}\sum_{\substack{r+s=1+\alpha\\r,s\geqslant 0}}(-1)^r\bq_i^{-(a+b)^2+\alpha(a+b)(a+b-1)/2} \\&\cdot \phi\left(\varsigma_i^a\varsigma_{\tau i}^b\qbinom{a+b-l+(sl-b)\alpha+\langle \alpha^\vee_{\tau i},l\lambda\rangle}{a}_i\qbinom{b-sl+\langle \alpha^\vee_i,l\lambda\rangle}{b}_i\right)\notag
	\\&\cdot  F_i^{(rl-a)}F_{\tau i}^{(l-a-b)}F_i^{(sl-b)}\one_{l\lambda}.\notag
	\end{align*}

	The coefficients $\phi\left(\varsigma_i^a\varsigma_{\tau i}^b\qbinom{a+b-l+(sl-b)\alpha+\langle \alpha^\vee_{\tau i},l\lambda\rangle}{a}_i\qbinom{b-sl+\langle \alpha^\vee_i,l\lambda\rangle}{b}_i\right)$ are independent of $s$ and $r$ by \cite[Lemma~34.1.2]{Lu93}. 
	Hence by Lemma~\ref{lem:rel3}, it suffices to consider the summands when $a+b = l$. Hence we have
	\begin{align*}
	&p_\imath'(\sum_{r+s=1+\alpha}(-1)^r\B_{i,*}^{(rl)}\B_{\tau i,*}^{(l)}\B_{i,l\zeta}^{(sl)})\notag\\
	=&\sum_{a=0}^l\sum_{\substack{r+s=1+\alpha\\r,s\geqslant 0}}(-1)^r\phi\left(\varsigma_i^a\varsigma_{\tau i}^{-a}\qbinom{(sl-l+a)\alpha+\langle \coroot_{\tau i},l\lambda\rangle}{a}_i\qbinom{l-a-sl+\langle \coroot_i,l\lambda\rangle}{l-a}_i\right)\\
	&\cdot F_i^{(rl-a)}F_i^{(sl-l+a)}\one_{l\lambda}\notag\\
	=&\sum_{r=0}^\alpha(-1)^r(r-\alpha+\langle \coroot_i,\lambda\rangle)\binom{\alpha}{r}F_i^{(\alpha l)}\one_{l\lambda} +\sum_{s=0}^\alpha(-1)^{1+\alpha-s}(s\alpha+\langle \coroot_{\tau i},\lambda\rangle)\binom{\alpha}{s}F_i^{(\alpha l)}\one_{l\lambda}
	\end{align*}
	Note that 
	\[
	F_i^{(rl-a)}F_i^{(sl-l+a)} =
	\begin{cases}
	 \binom{\alpha}{s}F_i^{(\alpha l)}, &\text{if } a =l;\\
	 \binom{\alpha}{r}F_i^{(\alpha l)}, &\text{if } a =0;\\
	 0, &\text{otherwise}.
	 \end{cases}
	\]
	
	By the identities
	\begin{equation*}
	\sum_{t=0}^n(-1)^t\binom{n}{t}=0\quad \text{ if }n>0,\qquad \sum_{t=0}^n(-1)^tt\binom{n}{t}=0\quad\text{ if }n>1,
	\end{equation*}
	we deduce that 
	\[
		p_\imath'(\sum_{r+s=1+\alpha}(-1)^r\B_{i,*}^{(rl)}\B_{\tau i,*}^{(l)}\B_{i,l\zeta}^{(sl)})
		= \begin{cases}
		0, &\text{if } \alpha\ge 2;\\
		\langle \coroot_i-\coroot_{\tau i}, \lambda\rangle \one_{l\lambda}, &\text{if } \alpha=0;\\
		-2F_i^{(l)}\one_{l\lambda}, &\text{if } \alpha=1.
		\end{cases}
	\]
	Now \eqref{eq:checkiwithtaui} follows by the isomorphism $p_\imath': {_{\A'}\dot{\U}}^\imath\one_{l\zeta}\rightarrow  {_{\A'}\dot{\U}}^-\one_{l\lambda}$.


	\subsubsection{}\label{subsec:pf4} Let $\zeta\in X_\imath$ and $i \in \I$ with $\tau i \neq i$. We write $\alpha=-a_{ij}$. We need to show the following equality in $_{\mathbf{F}}\dot{\U}^\imath$:
	\[
	\sum_{r+s=1+\alpha}(-1)^r\B_{i,*}^{(rl)}{'\B_{j,*}^{(l)}}\B_{i,l\zeta}^{(sl)}=0.
	\]
	Here $'\B_{j,*}^{(l)}$ stands for the standard $\imath$divided power $\B_{j,*}^{(l)}$ if $\tau j\neq j$. As before $*$ denotes appropriate elements in $X_\imath$.
	
	Let $\lambda\in X$ such that $\bar{\lambda}=\zeta$ and  let $p'_\imath=p'_{\imath,l\lambda}$ be the isomorphism in \eqref{eq:pimA}. Then we have
	\begin{equation}\label{eq:r}
	p_\imath'(\sum_{r+s=1+\alpha}(-1)^r\B_{i,*}^{(rl)}{'\B_{j,*}^{(l)}}\B_{i,l\zeta}^{(sl)})=\sum_{t=0}^{(l-1)/2}c_t\sum_{r+s=1+\alpha}(-1)^rF_i^{(rl)}F_j^{(l-2t)}F_i^{(sl)}\one_{l\lambda}
	\end{equation}
	where $c_t\in\A_2'$.

	 The following claim can be proved similar to \cite[\S35.2.3]{Lu93}.
	 
		(a) {\it For $0\leqslant a \leqslant l$, we have  $\sum_{r+s=1+\alpha}(-1)^rF_i^{(rl)}F_j^{(a)}F_i^{(sl)}=0$ in $_{\A'}\U^-$.}

	By \eqref{eq:r}, we complete the proof in this case.

	\subsubsection{}\label{subsec:pf5}
	Let  $i \neq j\in\I$ with $\tau i=i$. We write $\alpha=-a_{ij}$. We show  the following equality in $_{\mathbf{F}}\dot{\U}^\imath$:
	 \[
 \displaystyle \sum_{\substack{r+s+2t=1+\alpha\\t\geqslant 0}}(-1)^r\lr{1}{2t}{'\B_{i,*}^{(rl)}}{'\B_{j,*}^{(l)}}{'\B_{i,l\zeta}^{(sl)}}=0.
 \]
	Here $'\B_{j,*}^{(l)}$ stands for the standard $\imath$divided power $\B_{j,*}^{(l)}$ if $\tau j\neq j$, and $*$ denotes appropriate elements in $X_\imath$.

	For any $n,a\in\mathbb{N}$, $e=\pm 1$, set
	\begin{equation*}
	D_{i;a,e}^{(n)}=\sum_{u\geqslant 0}q_i^{-e(a-1)u}\LR{a}{2u}_iB_i^{(n-2u)}\in \U^\imath.
	\end{equation*}
	In particular, $D_{i;0,e}^{(n)}=B_i^{(n)}$. We write $D_i^{(n)}=D_{i; 1,\pm 1}^{(n)}$. We have $B_iB_i^{(a)}=[a+1]_iD_i^{(a+1)}$ by Lemma ~\ref{lem:balanced} for any $a\in\mathbb{N}$.

	Then for $n,m\in\mathbb{N}$, thanks to Proposition~ \ref{prop:HigherSerreRe}, we have
	\begin{equation*}
	y_{n,m}=y_{i,j;n,m}=\sum_{r=0}^m(-1)^rq_i^{-(m-n\alpha-1)r}B_i^{(r)}B_j^{(n)}D_{i;m-n\alpha,1}^{(m-r)}.
	\end{equation*}
	
	Let $\U_i^\imath\subseteq \U^\imath$ be the subalgebra (with 1) generated by $B_i$. It is a polynomial ring with one variable over $\mathbb{Q}(q)$. Elements $\{B_i^{(n)}\mid n\in\mathbb{N}\}$ form a $\mathbb{Q}(q)$-basis for $\U_i^\imath$. Define linear operators $\delta=\delta_i$, $E_e=E_{i,e}$ ($e=\pm 1$) on $\U_i^\imath$, such that $\delta(B_i^{(n+1)})=B_i^{(n)}$, $E_e(B_i^{(n)})=q_i^{en}D^{(n)}$, for any $n\in\mathbb{N}$. We understand $B_i^{(n)}$ as $0$ if $ n < 0$.

	By a proof similar to Lemma~\ref{lem:balanced}, we obtain 
\begin{align}
    B_i^{(a-1)}D_i^{(n)} &=\sum_{t\geqslant 0}\qbinom{a+n-1}{n}_i\prod_{m=1}^t\frac{[a-2m+2]_i[n-2m+2]_i}{[a+n-2m+1]_i[2m]_i}B_i^{(a+n-2t-1)} \notag\\
    & = \sum_{t\geqslant 0}\qbinom{a+n-1}{n}_i \frac{\LR{a}{2t}_i \LR{n}{2t}_i}{\LR{a+n-1}{2t}_i}B_i^{(a+n-2t-1)}. \label{eq:BD}
\end{align}

	\begin{proposition}\label{prop:Rank1Operators}
		
		(a) For any $n,a \in\mathbb{N}$, we have $E_e(D_{i;a,e}^{(n)})=q_i^{en}D_{i;a+1,e}^{(n)}$.
		
		(b) As operators on $\U_i^\imath$, we have $\delta E_e=q_i^e E_e\delta$.
		
		(c) For any $f,g\in\U^\imath_i$, $k\in\mathbb{N}$, we have 
		\begin{equation*}
		\delta^k(fg)=\sum_{s=0}^k\qbinom{k}{s}_i(E_e^{k-s}\delta^sf)(E_{-e}^s\delta^{k-s}g).
		\end{equation*}
		
		(d) For $k,N\in\mathbb{N}$, we have 
			\begin{equation*}
		\delta_i^N(D_i^{(k)}B_i^{(N)})=q_i^{Nk}\sum_{s=0}^N\qbinom{N}{s}_iq_i^{-(N+k)s}D_{i;N-s+1,1}^{(k-s)}D_{i;s,-1}^{(s)}.
		\end{equation*}
		
	\end{proposition}
	
	\begin{proof}
 
 (a) By direct computation, we have
 \begin{align*}
     E_{i,e}(D_{i;a,e})&=E_{i,e}(\sum_{u\geqslant 0}q_i^{-e(a-1)u}\LR{a}{2}_iB_i^{(n-2u)})\\
     &=\sum_{u\geqslant 0}q_i^{-e(a-1)u}\LR{a}{2u}_iq_i^{e(n-2u)}D_i^{(n-2u)}\\
     &=\sum_{u\geqslant 0}q_i^{-e(a-1)u+e(n-2u)}\LR{a}{2u}_i\sum_{t\geqslant 0}\LR{1}{2t}_iB_i^{(n-2u-2t)}\\
     &=\sum_{m\geqslant 0}\sum_{u+t=m}q_i^{en-e(a+1)u}\LR{a}{2u}_i\LR{1}{2t}_iB_i^{(n-2m)}\\
     &=q_i^{en}\sum_{m\geqslant 0}q_i^{-eam}\LR{a+1}{2m}_iB_i^{(n-2m)}\\&=q_i^{en}D_{i;a+1,e}^{(n)}.
 \end{align*}
 
 (b) For any $a\in\mathbb{N}$, we have
\begin{align*}
    \delta_i E_{i,e}(B_i^{(a)})=q_i^{ae}\delta_i(\sum_{u\geqslant 0}\LR{1}{2u}_iB_i^{(a-2u)})=q_i^{ae}\sum_{u\geqslant 0}\LR{1}{2u}_iB_i^{(a-2u-1)}.
\end{align*}
Also, we have
\begin{equation*}
    E_{i,e}\delta_i(B_i^{(a)})=E_{i,e}(B_i^{(a-1)})=q_i^{(a-1)e}\sum_{u\geqslant 0}\LR{1}{2u}_iB_i^{(a-1-2u)}=q_i^{-e}\delta_iE_{i,e}(B_i^{(n)}).
\end{equation*}

(c) We may assume $k\geqslant 1$. We firstly prove the formula for $k=1$. It will suffice to check that for $a,c\in\mathbb{N}$:
\begin{equation*}
\begin{split}
    \delta_i(B_i^{(a)}B_i^{(c)})&= E_{e} ( B_i^{(a)} ) \delta_i ( B_i^{(c)}) +  \delta_i ( B_i^{(a)} ) E_{-e}( B_i^{(c)})\\&  = q_i^{ea}D_i^{(a)}B_i^{(c-1)}+q_i^{-ec}B_i^{(a-1)}D_i^{(c)}.
\end{split}
\end{equation*}

Note that 
\begin{align*}
    q_i^{ea}D_i^{(a)}B_i^{(c-1)}+q_i^{-ec}B_i^{(a-1)}D_i^{(c)}&=q_i^{ea}\frac{B_iB_i^{(a-1)}B_i^{(c-1)}}{[a]_i}+q_i^{-ec}B_i^{(a-1)}D_i^{(c)}\\
    &=\left(q_i^{ea}\frac{[c]_i}{[a]_i}+q_i^{-ec}\right)B_i^{(a-1)}D_i^{(c)}\\
    &=\frac{[a+c]_i}{[a]_i}B_i^{(a-1)}D_i^{(c)}.
\end{align*}

By Lemma \ref{lem:balanced} and \eqref{eq:BD}, we have
\begin{equation*}
    \delta_i(B_i^{(a)}B_i^{(c)})=\sum_{t\geqslant 0}\qbinom{a+c}{a}_i\prod_{m=1}^t\frac{[a-2m+2]_i[c-2m+2]_i}{[a+c-2m+1]_i[2m]_i}B_i^{(a+c-2t-1)} = B_i^{(a-1)}D_i^{(c)}.
\end{equation*}
%
%
Now suppose the formula holds for $k$. Then
\begin{align*}
    \delta_i^{k+1}(fg)&=\sum_{s=0}^k\qbinom{k}{s}_i\delta_i((E_{i,e}^{k-s}\delta_i^sf)(E_{i,-e}^s\delta_i^{k-s}g))\\&=\sum_{s=0}^k\qbinom{k}{s}_i((\delta_iE_{i,e}^{k-s}\delta_i^sf)(E_{i,-e}^{s+1}\delta_i^{k-s}g)+(E_{i,e}^{k+1-s}\delta_i^sf)(\delta_iE_{i,-e}^s\delta_i^{k-s}g))\\
    &=\sum_{s=0}^k\qbinom{k}{s}_i((q_i^{e(k-s)}E_{i,e}^{k-s}\delta_i^{s+1}f)(E_{i,-e}^{s+1}\delta_i^{k-s}g)+(E_{i,e}^{k+1-s}\delta_i^sf)(q_i^{-es}E_{i,-e}^s\delta_i^{k+1-s}g))\\
    &=\sum_{s=0}^{k+1}\left(\qbinom{k}{s-1}_iq_i^{e(k-s+1)}+q_i^{-es}\qbinom{k}{s}_i\right)(E_{i,e}^{k+1-s}\delta_i^sf)(E_{i,-e}^s\delta_i^{k+1-s}g)\\
    &=\sum_{s=0}^{k+1}\qbinom{k+1}{s}_i(E_{i,e}^{k+1-s}\delta_i^sf)(E_{i,-e}^s\delta_i^{k+1-s}g).
\end{align*}
Hence the formula holds for $k+1$. We complete the proof by induction.

(d) We have
	\begin{align*}
	    \delta_i^N(D_i^{(k)}B_i^{(N)})&=\sum_{s=0}^N\qbinom{N}{s}_i(E_{i,1}^{N-s}\delta_i^sD_i^{(k)})(E_{i,-1}^s\delta_i^{N-s}B_i^{(N)})\\
	    &=q_i^{-k}\sum_{s=0}^N\qbinom{N}{s}_i(E_{i,1}^{N-s}\delta_i^sE_{i,1}B_i^{(k)})(E_{i,-1}^sB_i^{(s)})\\
	    &=q_i^{-k}\sum_{s=0}^Nq_i^{s}\qbinom{N}{s}_i(E_{i,1}^{N-s+1}\delta_i^sB_i^{(k)})(q_i^{-s^2}D_{i,s,-1}^{(s)})\\
	    &=q_i^{Nk}\sum_{s=0}^N\qbinom{N}{s}_iq_i^{-(N+k)s}D_{i,N-s+1,1}^{(k-s)}D_{i,s,-1}^{(s)}.
	\end{align*}
	This finishes the proof.
	\end{proof}
	
	For $0\leqslant k<l$, we have $y_{l,(1+\alpha)l-k} = 0$ by higher Serre relations. We have 
	\begin{align}
	0&= \sum_{k=0}^{l-1}(-1)^kq_i^{k} y_{l,(1+\alpha)l-k}  D_{i;k,-1}^{(k)} \notag  \\
		&=\sum_{k=0}^{l-1}(-1)^kq_i^{k} \Big( \sum_{r=0}^{(1+\alpha)l-k}(-1)^rq_i^{-(l-k-1)r}B_i^{(r)}B_j^{(l)}D_{i;l-k,1}^{((1+\alpha)l-k-r)}\Big) D_{i;k,-1}^{(k)}\notag\\
	&=\sum_{r=0}^{(1+\alpha)l}(-1)^rq_i^{-(l-1)r}B_i^{(r)}B_j^{(l)}J_r,\label{eq:S}
	\end{align}
	where 
	$J_r=\sum_{k=0}^{l-1}(-1)^kq_i^{k(r+1)}D_{i;l-k,1}^{((1+\alpha)l-k-r)}D_{i;k,-1}^{(k)}\in \U^\imath_i$.
%
%
	
		Let $_{\A_2}\U^\imath_i$ be the $\A_2$ submodule of $\U^\imath_i$ spanned by $B_i^{(n)}$ for $n\in\mathbb{N}$. Then it is also an $\A_2$-subalgebra. Note that elements $D_{i;a,e}^{(n)}$ ($a,n\in\mathbb{N}$, $e=\pm 1$) belong to this subalgebra. And set $_{\A_2'}\U_i^\imath={\A_2'}\otimes_{\A_2}(_{\A_2}\U^\imath_i)$ where $\A_2\rightarrow\A_2'$ is the canonical quotient. Note that operators $\delta_i$, $E_{i,e}$ preserve $_{\A_2}\U^\imath_i$, so they induce linear operators on  $_{\A_2'}\U_i^\imath$. We use the same notation to denote the corresponding operator after base change.


	By Proposition~\ref{prop:Rank1Operators}, we have
	\begin{align}\label{eq:J}
	&\delta_i^{l-1}(D_i^{((1+\alpha)l-r)}B_i^{(l-1)}) \notag \\
	&= \bq_i^{(l-1)((1+\alpha)l-r)}\sum_{s=0}^{l-1}\phi\left(\qbinom{l-1}{s}_i\right)\bq_i^{-((2+\alpha)l-r-1)s}D_{i;l-s,1}^{((1+\alpha)l-r-s)}D_{i;s,-1}^{(s)}\notag \\
	&\stackrel{\heartsuit}{=}\bq_i^r\sum_{s=0}^{l-1}(-1)^s\bq_i^{(r+1)s}D_{i;l-s,1}^{((1+\alpha)l-r-s)}D_{i;s,-1}^{(s)} =\bq_i^rJ_r \quad \text{in } _{\A_2'}\U_i^\imath
	\end{align}
	
	Here $\heartsuit$ follows by $\phi\left(\qbinom{l-1}{s}_i\right)=(-1)^s$, for $0\leqslant s<l$. By the proof of Proposition \ref{prop:Rank1Operators} (c), we have
	\begin{equation}\label{eq:l}
	[(2+\alpha)l-r]_iD_i^{((1+\alpha)l-r)}B_i^{(l-1)}=[l]_i\delta_i(B_i^{(l)}B_i^{((1+\alpha)l-r)}) \quad \text{in } _{\A_2}\U^\imath_i
	\end{equation}
	
	Note that $\phi([l]_i)=0$ in $\A'_2$. Suppose $l\nmid r$, then  $\phi([(2+\alpha)l-r]_i)\neq 0$. Hence 
	\begin{align*}
		J_r &= \bq_i^{-r} \delta_i^{l-1}(D_i^{((1+\alpha)l-r)}B_i^{(l-1)}) \\
		&= \bq_i^{-r}\phi([(2+\alpha)l-r]_i)^{-1}\phi([l]_i) \delta_i^{l}(B_i^{(l)}B_i^{((1+\alpha)l-r)}) =0 \quad \text{in } _{\A'_2}\U^\imath_i.
	\end{align*}

	Suppose $r=lr'$, with $0\leqslant r'\leqslant 1+\alpha$. Then by \eqref{eq:l} and \S\ref{subsec:pf2}, we have 
	\begin{align*} 
	&(2+\alpha-r')D_i^{((1+\alpha-r')l)}B_i^{(l-1)} \\
	& =\delta_i(B_i^{(l)}B_i^{((1+\alpha-r')l)}) \\
	& =  (2+\alpha-r')\sum_{u\geqslant 0}\lr{1}{2u}B_i^{((2+\alpha-r'-2u)l-1)}
	\quad \text{in } _{\A'_2}\U^\imath_i.
	\end{align*}
%
%
	
	Hence by \eqref{eq:J}, we deduce 
	\begin{equation*}
	J_r=\sum_{u\geqslant 0}\lr{1}{2u}B_i^{((1+\alpha-r'-2u)l)} \quad \text{in } _{\A_2'}\U^\imath_i.
	\end{equation*}
	
	For any $\zeta\in X_\imath$, combining with \eqref{eq:S}, we have the following equality in $_{\A_2'}\dot{\U}^\imath$
	\begin{align*}
	    0&=\sum_{r'=0}^{1+\alpha}(-1)^{r'}B_i^{(r'l)}B_j^{(l)}\sum_{u\geqslant 0}\lr{1}{2u}B_i^{((1+\alpha-r'-2u)l)}\one_{l\zeta}\\
	    &=\sum_{\substack{r+s+2u=1+\alpha\\u\geqslant 0}}(-1)^r\lr{1}{2u}B_i^{(rl)}B_j^{(l)}B_i^{(sl)}\one_{l\zeta}\\
	    &=\sum_{\substack{r+s+2u=1+\alpha\\u\geqslant 0}}(-1)^r\lr{1}{2u}{'\B_{i,*}^{(rl)}}{'\B_{j,*}^{(l)}}{'\B_{i,l\zeta}^{(sl)}}.
	\end{align*}
	We finish the proof in this case. 

	
	\subsection{The splitting on modules}\label{sec:splitmod}
	
    \subsubsection{}
    
	For any $\lambda\in X^+$, recall \S\ref{sec:qanniR} that $L(\lambda)$ denotes the irreducible $\mathrm{U}$-module of highest weight $\lambda$.  It follows from \cite[3.8]{BW18a} that the $\mathcal{A}$-linear map $_\mathcal{A}\dot{\mathrm{U}}^\imath\rightarrow {_\mathcal{A}L(\lambda)}$ sending $u$ to $u\cdot v_\lambda ^+$ is surjective.  
	
\begin{lem}\label{lem:annil}
The map $\pi:{_\mathcal{A}\dot{\mathrm{U}}^\imath}\rightarrow {_\mathcal{A}L(\lambda)}$ sending $u$ to $u\cdot v_\lambda ^+$ has kernel
\[
    J(\lambda)=\sum_{\zeta\in X_\imath,\;\zeta\neq\bar{\lambda}}{_\mathcal{A}\dot{\mathrm{U}}^\imath}\one_\zeta+\sum_{i\in\I,\;n>\langle \coroot_i,\lambda\rangle}{_\mathcal{A}\dot{\mathrm{U}}^\imath}B_{i,\bar{\lambda}}^{(n)}.
\]
\end{lem}

\begin{proof}
Note that ${_\mathcal{A}\dot{\mathrm{U}}^\imath}=\bigoplus_{\zeta\in X_\imath}{_\mathcal{A}\dot{\mathrm{U}}^\imath}\one_\zeta$, and $\one_\zeta\cdot v_\lambda^+=0$ unless $\zeta=\bar{\lambda}$. Hence it will suffice to show the restricting map $\pi:{_\mathcal{A}\dot{\mathrm{U}}^\imath}\one_{\bar{\lambda}} \rightarrow {_\mathcal{A}L(\lambda)}$ has kernel 
$$
Q(\lambda)=\sum_{i\in\I,\;n>a_i}{_\mathcal{A}\dot{\mathrm{U}}^\imath}B_{i,\bar{\lambda}}^{(n)}.
$$

The action map factors through the ${_\mathcal{A}\dot{\mathrm{U}}^\imath}$-module isomorphism  \eqref{eq:pimA} $\psi: {_\mathcal{A}\dot{\mathrm{U}}^\imath}\one_{\bar{\lambda}}\xrightarrow{\sim}{_\mathcal{A}\dot{\mathrm{U}}^-}\one_\lambda \xrightarrow{\sim} {}_\CA M(\lambda)$. Here ${}_\CA M(\lambda)$ denotes the $\A$-form of the Verma module $M(\lambda)$ with highest weight $\lambda$.

It follows from \cite[23.3.3 \& 23.3.6]{Lu93} that the $\mathcal{A}$-linear map $ {}_\CA M(\lambda) \rightarrow {_\mathcal{A}L(\lambda)}$ sending the highest weight vector to $v_\lambda^+$ has kernel 
\[
   P(\lambda)=\sum_{i\in\I,\; n>a_i}{_\mathcal{A}\dot{\mathrm{U}}}\cdot F_i^{(n)}v^+_\lambda =\sum_{i\in\I,\; n>a_i}{_\mathcal{A}\mathrm{U}}^- F_i^{(n)}v^+_\lambda.
\]

Therefore $Q(\lambda) = \psi^{-1} (P(\lambda) ) = \sum_{i\in\I,\;n>a_i}{_\mathcal{A}\dot{\mathrm{U}}^\imath}  \psi^{-1}(F_i^{(n)}v^+_\lambda)$. If $\tau i \neq i$, then we have $\psi^{-1}(F_i^{(n)}v^+_\lambda) = B_{i,\bar{\lambda}}^{(n)} $ by direct computation. If $\tau i = i$, then the $\CA$-span of $\{B_{i,\bar{\lambda}}^{(n)}\}_{n > a_i}$ is the same as the $\CA$-span of  $\{\psi^{-1}(F_i^{(n)}v^+_\lambda)\}_{n > a_i}$ by \cite[Theorem~2.10\&3.6]{BerW18}. This finishes the proof.
\end{proof}

For any $\mathcal{A}$-algebra $R$ and $\lambda\in X^+$, recall $_R L(\lambda)=R\otimes_{\A} {_\mathcal{A}L(\lambda)}$. Since ${_\mathcal{A}L(\lambda)}$ is a free (hence flat) $\CA$-module, we obtain the following corollary by base change.

\begin{cor}\label{cor:annlR}
The $R$-linear map $_R\dot{\mathrm{U}}^\imath \rightarrow {_RL(\lambda)}$ given by $r\otimes u\mapsto (r\otimes u)(1\otimes v_\lambda^+)$ is surjective and the kernel is the left ideal of $_R\dot{\mathrm{U}}^\imath$ generated by $1\otimes\one_\zeta$ ($\zeta\neq\bar{\lambda}$), $1\otimes B_{i,\bar{\lambda}}^{(n)}$ for various $i\in\I,\, n>\langle \coroot_i,\lambda \rangle$.
\end{cor}

\subsubsection{}\label{se:L}
Recall \S\ref{sec:qFr}.

\begin{proposition}\label{prop:fa}
    For any $\lambda\in X^+$, there exists a unique $_{\A'}\dot{\mU}^\imath$-module homomorphism
    \[
    \psi_\lambda:{_{\A'}\mathfrak{L}}(\lambda)\longrightarrow \big({_{\A'}L} (l\lambda)\big)^\ifr, \quad v_{\lambda}^+ \mapsto v_{l\lambda}^+.
    \]
     Here $\big({_{\A'}L} (l\lambda)\big)^{\ifr}$ stands for the $_{\A'}\dot{\mU}^\imath$-module, which is the same as ${_{\A'}L} (l\lambda)$ as $\A'$-modules, and the action of $_{\A'}\dot{\mU}^\imath$ is twisted by $\ifr$.
\end{proposition}

\begin{proof}
Since ${_{\A'}\mathfrak{L}}(\lambda)$ is generated by $v_\lambda^+$ as $\cAUi$-module, the uniqueness is clear. Thanks to Corollary \ref{cor:annlR}, in order to show such map is well-defined, it will suffice to show that $\ifr (\fb_{i,\bar{\lambda}}^{(n)})v_{l\lambda}^+=0$, for any $i\in\I$ and $n>\langle \alpha^\vee_i,\lambda\rangle$.

Suppose $\tau i\neq i$. Then  $n>\langle \alpha^\vee_i,\lambda\rangle$ implies $nl>\langle \alpha^\vee_i,l\lambda\rangle$. Thanks to Corollary \ref{cor:annlR} and Theorem~\ref{thm:iFr}, we have 
$\ifr (\fb_{i,\bar{\lambda}}^{(n)})v_{l\lambda}^+=\RB_{i,\overline{l\lambda}}^{(nl)}v_{l\lambda}^+=0$.

Suppose $\tau i=i$. By Corollary \ref{cor:iFr}, we have 
$$
\ifr(\fb_{i,\bar{\lambda}}^{(n)}) = \sum_{t=0}^{(l-1)/2}  c_t \RB_{i,\overline{l\lambda}}^{(ln-2t)}, \quad c_t \in \CA'.
$$ 
Since $n>\langle \alpha^\vee_i,\lambda\rangle$, we have $ln\geqslant \langle \alpha^\vee_i,l\lambda\rangle + l$. Therefore $ln-2t\geqslant \langle \alpha^\vee_i,l\lambda\rangle +1$ for $t = 0, \dots, (l-1)/2$. By Corollary \ref{cor:annlR} again, we deduce that $\ifr(\fb_{i,\bar{\lambda}}^{(n)})v_{l\lambda}^+=0$. 

The proposition is proved.
\end{proof}

\subsubsection{}\label{sec:propJ}
For a subset $J$ of the index set $\I$, we call $J$ has \emph{property (*)} if: 

(a) $\tau j=j$, for any $j\in J$; (b) $\langle \alpha^\vee_i,\alpha_j\rangle \geqslant -1$, for any $j\in J$, and $i\in \I$; (c) for any fixed $i\in \I$, there are at most one $j\in J$, with $\langle \alpha^\vee_i, \alpha_j\rangle \neq 0$.

For a subset $J\subseteq \I$ with property (*), we define
\begin{equation*}
\RB_{J,\overline{\lambda}}^{(l-1)} =\prod_{j\in J}{\RB}_j^{(l-1)}\one_{\bar{\lambda}} \in \pAUi, \quad \text{for any }\lambda\in X.
\end{equation*}
We see that the product is independent of the order by condition (c). Note that $\bq_j^k+\bq_j^{-k}$ is invertible in $\A'$, whenever $l$ does not divide $k$. Hence by Lemma \ref{lem:ff}, the element $\RB_{J,\overline{\lambda}}^{(l-1)}$ indeed belongs to $\pAUi$.

\begin{proposition}\label{prop:faJ}
    For any $\lambda\in X^+$, and a subset $J$ of $\I$ which has property (*), there exists a unique $\cAUi$-module homomorphism
    \[
    \psi^J_\lambda: {_{\A'}\mathfrak{L}}(\lambda)\longrightarrow \big({_{\A'}L}(l\lambda)\big)^\ifr, \quad v_\lambda^+ \mapsto \RB_{J,\overline{\lambda}}^{(l-1)}v_{l\lambda}^+\,  \text{ (the untwisted action}).
    \]
    Here $\big({_{\A'}L} (l\lambda)\big)^\ifr$ stands for the $_{\A'}\dot{\mU}^\imath$-module, which is the same as ${_{\A'}L} (l\lambda)$ as $\A'$-modules, and the action of $_{\A'}\dot{\mU}^\imath$ is twisted by $\ifr$.
\end{proposition}

\begin{proof}
It will suffice to prove that $\ifr(\fb_{i,\bar{\lambda}}^{(n)})\RB_{J,\overline{l\lambda}}^{(l-1)}\cdot v_{l\lambda}^+=0$, for any $i\in \I$ and $n>\langle \alpha^\vee_i, \lambda\rangle$ by Lemma~\ref{lem:annil}. We divide the proof into two cases.

Case (a): $i \not \in J$. 

If $\langle  \alpha^\vee_i,\alpha_j\rangle=0$  for all $j\in J$, then $\ifr(\fb_{i,\bar{\lambda}}^{(n)})$ commutes with $\RB_{J,\overline{l\lambda}}^{(l-1)}$. By the proof of Proposition \ref{prop:fa}, we have  $\ifr(\fb_{i,\bar{\lambda}}^{(n)})v_{l\lambda}^+=0$ if $n>\langle \alpha^\vee_i,\lambda\rangle$. Hence $\ifr(\fb_{i,\bar{\lambda}}^{(n)})\RB_{J,\overline{l\lambda}}^{(l-1)} v_{l\lambda}^+=0$

Otherwise we can find a unique $j'\in J$  such that $\langle \alpha^\vee_{i},\alpha_{j'}\rangle =-1$. Then $\ifr(\fb_{i,\bar{\lambda}}^{(n)})$ commutes with $\RB_{j,\overline{l\lambda}}^{(l-1)}$, for any $j' \neq j\in J$. Hence
\[
\ifr(\fb_{i,\bar{\lambda}}^{(n)})\RB_{J,\overline{l\lambda}}^{(l-1)} v_{l\lambda}^+=\RB_{J\backslash\{j'\},*}^{(l-1)}\ifr(\fb_{i,\bar{\lambda}}^{(n)}){'\RB_{j',\overline{l\lambda}}^{(l-1)}}v_{l\lambda}^+.
\]

 We claim $\ifr(\fb_{i,\bar{\lambda}}^{(n)}){'\RB_{j',\overline{l\lambda}}^{(l-1)}}v_{l\lambda}^+=0$. Since $'\RB_{j',\overline{l\lambda}}^{(l-1)}v_{l\lambda}^+$ is an $\A'$-linear combinations of $ \RF_{j'}^{(t)}v_{l\lambda}^+$ with $ 0 \leqslant t \leqslant l-1$,  it suffices to show $\ifr(\fb_{i,\bar{\lambda}}^{(n)})\RF_{j'}^{(t)}v_{l\lambda}^+=0$ for any $ 0 \leqslant t \leqslant l-1$.  Note that $\RF_{j'}^{(t)}v_{l\lambda}^+$ is a highest weight vector of highest weight $\langle \alpha^\vee_i, l\lambda-t\alpha_{j'}\rangle < nl$, with respect to the rank one quantum group associated to $i \in \I$. By Corollary~\ref{cor:iFr}, we have 
	\[
		\ifr(\mathfrak{b}_{i,\overline{\lambda}}^{(n)}) =
			\begin{cases} 
			\RB_{i,\overline{l\lambda}}^{(ln)},  &\text{if } \tau i \neq i;\\
			\mathrm{B}_{i,\overline{l\lambda}}^{(ln)}, &\text{if } \tau i =i \text{ and } n = \langle \coroot_i,\overline{\lambda}\rangle + 1;\\
			\sum_{t=0}^{(l-1)/2}c_t\mathrm{B}_{i,\overline{l\lambda}}^{(nl-2t)} (c_t \in \A'), &\text{if } \tau i =i \text{ and } n > \langle \coroot_i,\overline{\lambda}\rangle  +1.
		\end{cases}
		\]
		The claim follows by the proof of Lemma~\ref{lem:annil} or \cite{BerW18}*{Theorem~2.10\&3.6}. We hence finish Case (a). 

Case (b): $i \in J$. 

Then by the definition of the property (*), $\ifr(\fb_{i,\bar{\lambda}}^{(n)})$ commutes with $\RB_{J,\overline{l\lambda}}^{(l-1)}$. Then by similar computations, we have
\[
\ifr(\fb_{i,\bar{\lambda}}^{(n)})\RB_{J,\overline{l\lambda}}^{(l-1)} v_{l\lambda}^+=\RB_{J,\overline{l\lambda}}^{(l-1)}\ifr(\fb_{i,\bar{\lambda}}^{(n)}) v_{l\lambda}^+=0, \quad \text{if }n>\langle \alpha^\vee_i,\lambda\rangle.
\]

We finish the proof.
\end{proof}


\section{Frobenius splittings of algebraic groups}\label{sec:iFrO}

In this section, we assume given an $\imath$root datum $(\I,Y,X,A,(\alpha_i)_{i\in\I},(\coroot_i)_{i\in\I},\tau,\theta)$ of finite (quasi-split) type. Retain the settings in \S\ref{sec:qFr}. 
\subsection{The maps $\ofr$ and $\iofr$}

\subsubsection{}

Recall \S \ref{sec:iqp} the embedding $_\A\dot{\U}^\imath\hookrightarrow{_\A\widehat{\U}}$. We have induced embeddings $_{\A'}\dot{\U}^\imath\hookrightarrow{_{\A'}\widehat{\U}}$, and $_{\A'}\dot{\mathfrak{U}}^\imath\hookrightarrow{_{\A'}\widehat{\mathfrak{U}}}$. It follows from \cite{BS21}\footnote{In \cite{BS21}, the authors use a different completion, but the same proof applies in our setting.} that $\ofr:{_{\A'}\widehat{\U}}\rightarrow {_{\A'}\widehat{\mU}}$ restricts to $\iofr:{_{\A'}\dot{\U}^\imath}\rightarrow {_{\A'}\dot{\mU}^\imath}$. 

For any canonical basis element $b$ of $_{\A'}\dot{\U}$, we write
\[
\ofr(b)=\sum_{b'\in\dot{\RB}}x_{b,b'}b', \quad \text{for } x_{b,b'}\in \A', b' \in {}_{\A'}\dot{\mU}.
\]
For any $\imath$canonical basis element $b$ of $_{\A'}\dot{\U}^\imath$, we write
\[
\iofr(b)=\sum_{b'\in\dot{\RB}^\imath}x^\imath_{b,b'}b', \quad \text{for } x^\imath_{b,b'}\in\A', b' \in {}_{\A'}\dot{\mathfrak{\U}}^\imath.
\]

\begin{lemma}\label{le:fiFr}
 (1) Let $b'$ be a canonical basis element of $_{\A'}\dot{\mU}$. There are only finitely many canonical basis element $b$ of $_{\A'}\dot{\U}$, such that $x_{b,b'}\neq 0$.

 (2)  Let $b'$ be a $\imath$canonical basis element of $_{\A'}\dot{\mU}^\imath$. There are only finitely many $\imath$canonical basis element $b$ of $_{\A'}\dot{\U}^\imath$, such that $x^\imath_{b,b'}\neq 0$.
\end{lemma}

\begin{proof}
We show part (1). Choose some $\mu_1,\mu_2\in X^+$, such that $b'\cdot (v_{-\mu_1}^-\otimes v_{\mu_2}^+)\neq 0$ in $^\omega {_{\A'}\mathfrak{L}}(\mu_1)\otimes_{\A'} {_{\A'}\mathfrak{L}}(\mu_2)$. 

By \S\ref{sec:qanniR}, it is direct to check that we have the $_{\A'}\dot{\U}$-module homomorphism:
\[
^\omega {_{\A'}L}(l\mu_1)\otimes_{\A'} {_{\A'}L}(l\mu_2)\longrightarrow \left({^\omega {_{\A'}\mathfrak{L}}(\mu_1)\otimes_{\A'} {_{\A'}\mathfrak{L}}(\mu_2)}\right)^\ofr,\quad v_{-l\mu_1}^-\otimes v_{l\mu_2}^+ \mapsto v_{-\mu_1}^-\otimes v_{\mu_2}^+.
\]
Here $\left({^\omega {_{\A'}\mathfrak{L}}(\mu_1)\otimes_{\A'} {_{\A'}\mathfrak{L}}(\mu_2)}\right)^\ofr$ is defined similar to that in Proposition~\ref{prop:ga}.

Suppose $b$ is some canonical basis element of $_{\A'}\dot{\U}$, such that $x_{b,b'}\neq 0$. Then thanks to the map above, we deduce that $b\cdot( v_{-l\mu_1}^-\otimes v_{l\mu_2}^+)\neq 0$. There are only finitely many such $b$ by \cite{Lu93}*{\S25.2}.

We now show part (2). For any $\mu\in X^+$, it follows from Corollary \ref{cor:annlR} that there is an $_{\A'}\dot{\U}^\imath$-module homomorphism :
\[
{_{\A'}L}(l\mu)\rightarrow {_{\A'}\mathfrak{L}}(\mu)^\iofr, \quad v_{l\mu}^+ \mapsto v_\mu^+.
\]
Here ${_{\A'}\mathfrak{L}}(\mu)^\iofr$ is defined similar to that in Proposition~\ref{prop:ga}.

For any fixed $\imath$canonical basis element $b'$ of $_{\A'}\dot{\mU}^\imath$, we can find $\mu\in X^+$, such that $b'\cdot v_\mu^+\neq 0$. Here $b'\cdot v_\mu^+$ stands for the standard action (not twisted by $\iofr$). 
Suppose $b$ is some $\imath$canonical basis element of $_{\A'}\dot{\U}^\imath$, such that $x_{b,b'}^\imath\neq 0$. Then thanks to the map above and Theorem \ref{thm:stab}, we deduce that $b\cdot v_{l\mu}^+\neq 0$. By Theorem~\ref{thm:stab} again, there are only finitely many such $b$.
\end{proof}

\subsubsection{}\label{sec:oFrbach}

Let $l=p$ be a prime number, which is relatively prime to all the root length. 
Let $\mathbb{F}_p$ be the finite field consisting of $p$ elements. Note that the $p$-th cyclotomic polynomial $f_p(x)$ is $1+x+\cdots+x^{p-1}$. So $f_p(1)=0$ modulo $p$. Hence there is a ring homomorphism $v:\A'\rightarrow \BF_p$, sending $\bq$ to 1. It follows that the compositions
\begin{equation*}
    \begin{tikzcd}
    \mathcal{A} \arrow[r,"c\text{,}\,\phi"] & \mathcal{A}' \arrow[r,"v"] & \mathbb{F}_p.
    \end{tikzcd}
\end{equation*}
are given by $q\mapsto 1$, regardless of the first map. We shall view $\BF_p$ as an $\A$-algebra in either way without ambiguity.

The we have canonical isomorphisms
\[
\BF_p\otimes_{\A'} {_{\A'}\dot{\mU}}\cong\BF_p\otimes_{\A'} {_{\A'}\dot{\U}}\cong {_{\BF_p}\!\dot{\U}}\quad\text{and}\quad\BF_p\otimes_{\A'} {_{\A'}\dot{\mU}^\imath}\cong\BF_p\otimes_{\A'} {_{\A'}\dot{\U}^\imath}\cong {_{\BF_p}\!\dot{\U}^\imath}.
\]
Therefore the $\A'$-algebra homomorphisms
\[
\ofr: {_{\A'}\dot{\U}}\longrightarrow{_{\A'}\dot{\mU}}\quad \text{ and }  \quad  \iofr:{_{\A'}\dot{\U}^\imath}\longrightarrow{_{\A'}\dot{\mU}^\imath}
\]
induce $\BF_p$-algebra homomorphisms
\[
\ofr:{_{\BF_p}\!\dot{\U}}\longrightarrow{_{\BF_p}\!\dot{\U}}\quad  \text{ and }\quad  \iofr:{_{\BF_p}\!\dot{\U}^\imath}\rightarrow{_{\BF_p}\!\dot{\U}^\imath}.
\]
Thanks to Lemma \ref{le:fiFr}, we further have well-defined $\BF_p$-algebra homomorphisms
\[
\ofr:{}_{\BF_p}\!\widehat{\U} \longrightarrow {}_{\BF_p}\!\widehat{\U} \quad  \text{ and }\quad  \iofr:{}_{\BF_p}\!\widehat{\U}^\imath \rightarrow {}_{\BF_p}\!\widehat{\U}^\imath.
\]

For any $i\in\I$, $n\in\BN$, $\lambda\in X$, we write $e_i^{(n)}\one_\lambda$ and $f_i^{(n)}\one_\lambda$ to denote the images of $E_i^{(n)}\one_\lambda$ (or $\mathfrak{e}_i^{(n)}\one_\lambda$) and $F_i^{(n)}\one_\lambda$ (or $\mathfrak{f}_i^{(n)}\one_\lambda$) in ${}_{\BF_p}\!\dot{\U}$, respectively. Also, for any $i\in\I$, $n\in\BN$, $\zeta\in X_\imath$, we write $b_{i,\zeta}^{(n)}$ and $'b_{i,\zeta}^{(n)}$ to denote the image of $B_{i,\zeta}^{(n)}$ (or $\mathfrak{b}_{i,\zeta}^{(n)}$) and $'B_{i,\zeta}^{(n)}$ (or $'\mathfrak{b}_{i,\zeta}^{(n)}$) in ${}_{\BF_p}\!\dot{\U}^\imath$.


\subsubsection{}
Recall the commutative Hopf algebras $\RO_{\BF_p}$ and $\RO^\imath_{\BF_p}$ in \S\ref{sec:nocp} and \S\ref{sec:pro}, respectively.
%
Thanks to Lemma \ref{le:fiFr}, the maps $\ofr$ and $\iofr$ induce $\BF_p$-linear maps
\[
{\ofr}^*:\RO_{\BF_p}\longrightarrow\RO_{\BF_p}\quad \text{ and } \quad  {\iofr}^*:\RO_{\BF_p}^\imath\longrightarrow\RO_{\BF_p}^\imath.
\]

It follows by definitions that $r\circ{\ofr^*}={\iofr^*}\circ r$, where $r:\RO_{\BF_p}\rightarrow\RO_{\BF_p}^\imath$ denotes the restriction defined in \ref{sec:pro}.

\begin{proposition}\label{prop:pthFr}
(1) For any $f\in \RO_{\BF_p}$, we have ${\ofr^*}(f)=f^p$. 

(2) For any $f\in \RO_{\BF_p}^\imath$, we have ${\iofr^*}(f)=f^p$.
\end{proposition}

\begin{proof}
We prove part (1). Part (2) follows from the identity $r\circ{\ofr^*}={\iofr^*}\circ r$ and the fact that $r$ is a surjective algebra homomorphism.

Let ${}_{\BF_p}\!\widehat{\U}^{(p)}$ be the $\BF_p$-linear space consisting of all the formal $\BF_p$-linear combinations
\[
\sum_{b_i\in\dot{\RB};\;1\leqslant i\leqslant p }n_{b_1,\cdots,b_p}b_1\otimes\cdots\otimes b_p, \quad \text{for }n_{b_1,\cdots,b_p}\in \BF_p.
\]
 Then ${}_{\BF_p}\!\widehat{\U}^{(p)}$ is naturally endowed with an structure of $\BF_p$-algebra. By applying comultiplication $p$ times, we obtain an $\BF_p$-algebra homomorphism $\widehat{\Delta}^p:{{}_{\BF_p}\!\widehat{\U}}\rightarrow {{}_{\BF_p}\!\widehat{\U}^{(p)}}$. For any $f\in\RO_A$, let $f^{\otimes p}:{{}_{\BF_p}\!\widehat{\U}^{(p)}}\rightarrow \BF_p$ be the $\BF_p$-linear map defined in the obvious way. 
 
 Then it suffices to check that the following diagram commutes:
\begin{equation}\label{dia:pthFr}
\begin{tikzcd}
{}_{\BF_p}\!\dot{\U} \arrow[r,"\ofr"] \arrow[d,"\widehat{\Delta}^p"'] & {}_{\BF_p}\!\widehat{\U} \arrow[d,"f"] \\ {}_{\BF_p}\!\widehat{\U}^{(p)} \arrow[r,"f^{\otimes p}"] & \BF_p
\end{tikzcd}
\end{equation}

Let $x=E_{i_1}^{(a_1)}\cdots E_{i_s}^{(a_s)}F_{j_1}^{(b_1)}\cdots F_{j_t}^{(b_t)}\one_\lambda\in {{}_{\BF_p}\!\dot{\U}}$. We have
\begin{align*}
\widehat{\Delta}^{p}(x)=&\sum_{\substack {a_{n,1}+a_{n,2}+\cdots a_{n,p}=a_n \\ 1\leq n \leq s \\ b_{m,1}+b_{m,2}+\cdots+b_{m,p}=b_m \\ 1\leq m\leq t\\ \lambda_1+\lambda_2+\cdots \lambda_p=\lambda}} E_{i_1}^{(a_{1,1})}\cdots E_{i_s}^{(a_{s,1})}F_{j_1}^{(b_{1,1})}\cdots F_{j_t}^{(b_{t,1})}\one_{\lambda_1} \otimes \cdots\\& \otimes  E_{i_1}^{(a_{1,p})}\cdots E_{i_s}^{(a_{s,p})}F_{j_1}^{(b_{1,p})}\cdots F_{j_t}^{(b_{t,p})}\one_{\lambda_p}
\end{align*}

The symmetric group $S_p$ acts on ${}_{\BF_p}\!\widehat{\U} $ by permutation. The element $\widehat{\Delta}^{p}(x)$ is $S_p$-invariant.  If a summand $\widehat{\Delta}^{p}(x)$ is not fixed by $S_p$, then the number of elements in its $S_p$-orbit will be divisible by $p$. Therefore the evaluation of $f^{\otimes p}$ over non-trivial $S_p$-orbits is $0$. 

Recall $z^p = z$ for any $z \in \BF_p$. Hence we have 
\begin{align*}
f^{\otimes p}\circ \widehat{\Delta}^p(x)=f\left(E_{i_1}^{(a_1/p)}\cdots E_{i_s}^{(a_s/p)}F_{j_1}^{(b_1/p)}\cdots F_{j_t}^{(b_t/p)}\one_{\lambda/p}\right)
\end{align*}
if $p$ divides $a_n$ ($1\leq n\leq s$), $p$ divides $b_m$ ($1\leq m\leq t$), and $\lambda\in p X$; 
and 
\[
f^{\otimes p}\circ \widehat{\Delta}^p(x)=0, \quad \text{otherwise}.
\]

Finally, by the definition of $\ofr$, we conclude that $f^{\otimes p}\circ \widehat{\Delta}^p(x)=f\circ{\ofr}(x)$. Hence the diagram \eqref{dia:pthFr} commutes.
\end{proof}

\subsection{The maps $\fr$ and $\ifr$}

\subsubsection{}

Recall the $\A'$-algebra homomorphisms $\fr:{_{\A'}\dot{\mU}}\rightarrow {_{\A'}\dot{\U}}$ in \S\ref{sec:qFr} and $\ifr:{_{\A'}\dot{\mU}^\imath}\rightarrow{_{\A'}\dot{\U}^\imath}$ in \S\ref{sec:iFr}.
For any canonical basis element $b$ of $_{\A'}\dot{\mU}$, we write
\[
\fr (b) =\sum_{b'\in \dot{\RB}}y_{b,b'}b', \quad \text{for }y_{b,b'}\in \A', b'\in{_{\A'}\dot{\U}}.
\]
Similarly, for any $\imath$canonical basis element $b$ of $_{\A'}\dot{\U}^\imath$, we write
\[
\ifr (b) =\sum_{b'\in \dot{\RB}^\imath}y^\imath_{b,b'}b', \quad \text{for }y^\imath_{b,b'}\in \A', b'\in{_{\A'}\dot{\U}^\imath}.
\]

\begin{lemma}\label{le:fiFrs}
(1) Let $b'$ be a canonical basis element of ${_{A'}\dot{\U}}$. Then there are only finitely many canonical basis element $b$ of $_{\A'}\dot{\mU}$, such that $y_{b,b'}\neq 0$.
 
 (2) Let  $b'$ be an $\imath$canonical basis element of ${_{A'}\dot{\U}^\imath}$. Then there are only finitely many $\imath$canonical basis element $b$ of $_{\A'}\dot{\mU}^\imath$, such that $y^\imath_{b,b'}\neq 0$.
\end{lemma}

\begin{proof}
We show part (1). Let $\mu_1,\mu_2\in X^+$ be such that $b'\cdot( v_{-l\mu_1}^-\otimes v_{l\mu_2}^+)\neq 0$ in ${^\omega _{\A'}L}(l\mu_1)\otimes_{\A'} {_{\A'}L}(l\mu_2)$. Thanks to \S\ref{sec:qanniR}, we have an $_{\A'}\dot{\mU}$-module homomorphism:
\[
{^\omega _{\A'}\mathfrak{L}}(\mu_1)\otimes_{\A'} {_{\A'}\mathfrak{L}}(\mu_2)\longrightarrow \left( {^\omega _{\A'}L}(l\mu_1)\otimes_{\A'} {_{\A'}L}(l\mu_2)\right)^\fr,\quad v_{-\mu_1}^-\otimes v_{\mu_2}^+ \mapsto v_{-l\mu_1}^-\otimes v_{l\mu_2}^+.
\]

Suppose $b$ is some canonical basis element of $_{\A'}\dot{\mU}$, such that $y'_{b,b'}\neq 0$. Then we deduce that $\fr(b)\cdot (v_{-l\mu_1}^-\otimes v_{l\mu_2}^+)\neq 0$. Hence $b\cdot (v_{-\mu_1}^-\otimes v_{\mu_2}^+)\neq 0$. There are only finitely many such $b$ by \cite{Lu93}*{\S25.2}.

We now show part (2). Let $\mu\in X^+$ be such that $b'\cdot v_{l\mu}^+\neq 0$. By  Proposition \ref{prop:fa}, we have the $_{\A'}\dot{\mU}^\imath$-module homomorphism:
\[
{_{\A'}\mathfrak{L}}(\mu)\longrightarrow \left({_{\A'}L}(l\mu)\right)^\ifr, \quad v_\mu^+ \mapsto v_{l\mu}^+.
\]

Suppose $b$ is some $\imath$canonical basis element of $_{\A'}\dot{\mU}^\imath$, such that $y^\imath_{b,b'}\neq 0$. Then we deduce that $\ifr(b)\cdot v_{l\mu}^+\neq 0$. Hence $b\cdot v_\mu^+\neq 0$. There are only finitely many such $b$ by Theorem~\ref{thm:stab}.
\end{proof}

As consequences, we obtain well-defined $\A'$-algebra homomorphisms  
\begin{equation*}
\fr:{_{\A'}\widehat{\mU}}\rightarrow {_{\A'}\widehat{\U}}\quad \text{and} \quad   \ifr:{_{\A'}\widehat{\mU}^\imath}\rightarrow {_{\A'}\widehat{\U}^\imath}.
\end{equation*}
 
 
\subsubsection{}
Recall \S\ref{sec:nocp} and \S\ref{sec:qFr}. Write ${_{\A'}\widehat{\U}^{(2)}}={_{\A'_\phi}\widehat{\U}^{(2)}}$ and ${_{\A'}\widehat{\mU}^{(2)}}={_{\A'_c}\widehat{\U}^{(2)}}$. Let ${_{\A'}\widehat{\U}}\widehat{\otimes}{_{\A'}\widehat{\mU}}$ be the $\A'$-module consisting of formal $\A'$-linear combinations
\[
\sum_{(b,b')\in\dot{\RB}\times\dot{\RB}} n_{b,b'}b\otimes b', \quad \text{for }n_{b,b'} \in \A'.
\]
This is moreover an $\A'$-algebra. 

Thanks to Lemma \ref{le:fiFr} and Lemma \ref{le:fiFrs}, we have well-defined $\A'$-algebra homomorphisms  
\[
id\otimes \ofr:{_{\A'}\widehat{\U}^{(2)}}\rightarrow {_{\A'}\widehat{\U}}\widehat{\otimes}{_{\A'}\widehat{\mU}} \quad \text{ and } \quad \fr\otimes id:{_{\A'}\widehat{\mU}^{(2)}}\rightarrow{_{\A'}\widehat{\U}}\widehat{\otimes}{_{\A'}\widehat{\mU}}.
\]

\begin{prop}
The following diagram commutes:
\begin{equation}\label{dia:frspl}
        \begin{tikzcd}
        _{\mathcal{A}'}\dot{\mathfrak{U}} \arrow[r,"\fr"] \arrow[d,"\widehat{\Delta}"'] & _{\mathcal{A}'}\dot{\mathrm{U}} \arrow[r,"\widehat{\Delta}"] & _{\mathcal{A}'}\widehat{\mathrm{U}}^{(2)} \arrow[d,"id\otimes\ofr"] \\
        _{\mathcal{A}'}\widehat{\mathfrak{U}}^{(2)}
        \arrow[rr, "\fr\otimes id"] & & _{\mathcal{A}'}\widehat{\mathrm{U}}\widehat{\otimes}{_{\A'}\widehat{\mU}}.
        \end{tikzcd}
    \end{equation}

\end{prop}

\begin{proof}
 It suffices to check the commutative diagram over the field of fractions $\mathbf{F}$. Since $_{\mathbf{F}}\dot{\mU}$ is generated by $\one_\lambda$ ($\lambda\in X$), $\mathfrak{e}_{i}\one_\lambda$ ($i\in\I$, $\lambda\in X$) and $\mathfrak{f}_{i}\one_\lambda$ ($i\in\I$, $\lambda\in X$), it will suffice to check the identity  
\begin{equation*}
(id\otimes\ofr)\circ\widehat{\Delta}\circ\fr(u)=(\fr\otimes id)\circ \widehat{\Delta}(u), \quad \text{for any such generator } u \in {}_{\mathbf{F}}\dot{\mU}.
\end{equation*}
By direct computations, one has
\[
(id\otimes\ofr)\circ\widehat{\Delta}\circ\fr(\one_\lambda)=\sum_{\lambda_1+\lambda_2=l\lambda}\one_{\lambda_1}\otimes \ofr(\one_{\lambda_2})=\sum_{\lambda_1'+\lambda_2'=\lambda}\one_{l\lambda_1'}\otimes\one_{\lambda_2'}=(\fr\otimes id)\circ \widehat{\Delta}(\one_\lambda)
\]


By \cite{Lu93}*{\S3.1.5}, for any $i\in\I$, $n\in\BN$, and $\lambda\in X$,  we have 
\begin{align*}
    &\widehat{\Delta}(E_i^{(n)}\one_\lambda)=\sum_{\substack{t+s=n\\ \mu+\mu'=\lambda}}q_i^{(t+\langle \alpha_i^\vee,\mu\rangle)s}E_i^{(t)}\one_\mu\otimes E_i^{(s)}\one_{\mu''},\\ &\widehat{\Delta}(F_i^{(n)}\one_{\lambda})=\sum_{\substack{t+s=n\\\mu+\mu'=\lambda}}q_i^{(s-\langle \alpha_i^\vee,\mu'\rangle)t} F_i^{(t)}\one_\mu\otimes  F_i^{(s)}\one_{\mu'}.
\end{align*}

Therefore 
\begin{align*}
    &\widehat{\Delta}\circ\fr(\mathfrak{e}_i\one_\lambda)=\widehat{\Delta}(\RE_i^{(l)}\one_{l\lambda})=\sum_{\substack{t+s=l\\ \mu+\mu'=l\lambda}}\bq_i^{(t+\langle \alpha_i^\vee,\mu\rangle)s}\RE_i^{(t)}\one_\mu\otimes \RE_i^{(s)}\one_{\mu'},\\
    &\widehat{\Delta}\circ\fr(\mathfrak{f}_i\one_\lambda)=\widehat{\Delta}(\RF_i^{(l)}\one_{l\lambda})=\sum_{\substack{t+s=l\\\mu+\mu'=l\lambda}}\bq_i^{(s-\langle \alpha_i^\vee,\mu'\rangle)t} \RF_i^{(t)}\one_\mu\otimes  \RF_i^{(s)}\one_{\mu'}
\end{align*}

Composing with $id\otimes \ofr$, the summands on the right hand sides are 0, unless $\{t,s\}=\{0,l\}$ and $\mu,\mu' \in l X$. Hence we obtain
\begin{align*}
(id\otimes \ofr)\circ\widehat{\Delta}\circ\fr(\mathfrak{e}_i\one_\lambda)=\sum_{\mu+\mu'=\lambda}(\one_{l\mu}\otimes\mathfrak{e}_i\one_{\mu'}+\RE_i^{(l)}\one_{l\mu}\otimes\one_{\mu'})=(\fr\otimes id)\circ \widehat{\Delta}(\mathfrak{e}_i\one_\lambda),\\
(id\otimes \ofr)\circ\widehat{\Delta}\circ\fr(\mathfrak{f}_i\one_\lambda)=\sum_{\mu+\mu'=\lambda}(\one_{l\mu}\otimes\mathfrak{f}_i\one_{\mu'}+\RF_i^{(l)}\one_{l\mu}\otimes\one_{\mu'})=(\fr\otimes id)\circ \widehat{\Delta}(\mathfrak{f}_i\one_\lambda).
\end{align*}
This completes the proof.
\end{proof}

\subsubsection{}


For any $i\in\I$ with $\tau i=i$, and any $n\in\BN$, recall \S \ref{sec:baid} the definition of the elements $B_i^{(n)}$ in $\U^\imath$. Also recall the elements $'\RB_{i,\zeta}^{(n)}$ ($\zeta\in X_\imath$, $n\in\BN$) in $_{\A_2'}\dot{\U}^\imath$ in \S \ref{sec:iFr}.

\begin{lemma}\label{le:conew}
 For any $i\in\I$ with $\tau i=i$, and any $n\in\BN$, we have
 \[
 \Delta(B_i^{(n)})=\sum_{r=0}^n B_i^{(n-r)}\otimes M_{n,r}
 \]
 where $M_{n,r}\in \U$, such that for any $\lambda\in X$, we have
 \[
 M_{n,r}\one_\lambda=\sum_{\substack{a+a'+2c=r\\c\geqslant 0}}q_i^{s_i(a,a',\lambda,n,c)}\LR{a'-a-\lambda}{2c}_iE_i^{(a)}F_i^{(a')} \one_\lambda
 \]
 where ($\lambda_i=\langle\coroot_i,\lambda\rangle$)
\[
    s_i(a,a',\lambda,n,c)=(a-a'+n+\lambda_i)a'-n(\lambda_i+a)+(1+a-a'+\lambda_i)c.
\]
\end{lemma}


\begin{proof}
In the algebra $\U^\imath$, by the comultiplication formulas \cite{CW22}*{Theorem 4.2 \& Theorem 5.1}, we have
 \[
     \Delta(B_i^{(n)})=\Delta(B_{i,\overline{n+1}}^{(n)})=\sum_{r=0}^n B_{i,\overline{n+1}}^{(n-r)}\otimes S_{n,r},
 \]
 with
 \[
     S_{n,r}=\sum_{\substack{a+a'+2c=r\\c\geqslant 0}}q_i^{\binom{2c+1}{2}+(a+a')(r-n)-aa'}\Check{E_i}^{(a)}\qbinom{h;-\lfloor \frac{r-1}{2}\rfloor}{c}_iK_i^{r-n}F_i^{(a')},
 \]
 where
 \[
 \check{E}_i^{(a)}=\frac{(q_i^{-1}E_iK_i^{-1})^a}{[a]_i!}, \qquad \qbinom{h;a}{n}_i=\prod_{s=1}^n\frac{q_i^{4a+4s-4}K_i^{-2}-1}{q_i^{4s}-1}.
 \]
 In particular, we have
 \[
 \qbinom{h;a}{n}_i\one_\lambda =q_i^{(2a-2-\lambda_i)n}\LR{2a+2n-\lambda_i-2}{2n}_i\one_\lambda, \quad \text{where }\lambda_i=\langle \coroot_i,\lambda\rangle.
 \]

 Then it is direct to compute
 \[
  S_{n,r}\one_\lambda= 
  \begin{cases}
  \displaystyle\sum_{\substack{a+a'+2c=r\\c\geqslant 0}}q_i^{s_i(a,a',\lambda,n,c)}\LR{a'-a-\lambda_i}{2c}_iE_i^{(a)}F_i^{(a')}\one_\lambda, &\text{if $r$ is even};\\
    \displaystyle\sum_{\substack{a+a'+2c=r\\c\geqslant 0}}q_i^{s_i(a,a',\lambda,n,c)-c}\LR{a'-a-\lambda_i-1}{2c}_iE_i^{(a)}F_i^{(a')}\one_\lambda, &\text{if $r$ is odd}.
  \end{cases}
 \]
 
 Hence by Lemma \ref{lem:ff} , we have
 \begin{align*}
 &\Delta(B_{i}^{(n)})= \sum_{\substack{0\leqslant r\leqslant n\\r\,\text{even}}}\left( {B_{i}^{(n-r)}}\otimes \sum_{\substack{a+a'+2c=r\\c\geqslant 0}}q_i^{s_i(a,a',\lambda,n,c)}\LR{a'-a-\lambda_i}{2c}_iE_i^{(a)}F_i^{(a')}\one_\lambda\right)+ \\
        & \sum_{\substack{0\leqslant r\leqslant n\\r\,\text{odd}}} \left(\left(\sum_{t\geqslant 0}\LR{1}{2t}_i{B_{i}}^{(n-r-2t)}\right)\otimes \sum_{\substack{a+a'+2c=r\\c\geqslant 0}}q_i^{s_i(a,a',\lambda,n,c)-c}\LR{a'-a- \lambda_i-1}{2c}_iE_i^{(a)}F_i^{(a')}\one_\lambda\right).
\end{align*}

It remains to compute the coefficient of $B_{i}^{(n-r)}\otimes E_i^{(a)}F_i^{(a')}\one_\lambda$ when both $r$ and $a+a'$ are odd. The coefficient equals
\begin{align*}
    &\sum_{c+t=(r-a-a')/2}q_i^{s_i(a,a',\lambda,n,c)-c}\LR{1}{2t}_i\LR{a'-a-\lambda-1}{2c}_i\\
    =&q_i^{(a-a'+n+\lambda)a'-n(\lambda+a)}\sum_{c+t=(r-a-a')/2}q_i^{(a-a'+\lambda)c}\LR{1}{2t}_i\LR{a'-a-\lambda-1}{2c}_i\\
    =&q_i^{s_i(a,a',\lambda,n,(r-a-a')/2)}\LR{a'-a-\lambda_i}{r-a-a'}_i. \qquad \text{(By Lemma~\ref{lem:DoubleInA2}.)}
\end{align*}

This completes the proof.
\end{proof}

Following the notations in \S\ref{sec:nocp} and \S\ref{sec:Ui1}, we write ${_{\A'}\widehat{\U}^{\imath,1}}={_{\A'_\phi}\widehat{\U}^{\imath,1}}$ and ${_{\A'}\widehat{\mU}^{\imath,1}}={_{\A'_c}\widehat{\U}^{\imath,1}}$. Let  ${_{\A'}\widehat{\U}^\imath}\widehat{\otimes}{_{\A'}\widehat{\mU}}$ be the $\A'$-module consisting of formal $\A'$-linear combinations
\[
\sum_{(b,b')\in\dot{\RB}_\phi\times\dot{\RB}_c}n_{b,b'}b\otimes b', \quad \text{for }n_{b,b'} \in \A'
\]
This is moreover an $\A'$-algebra. By Lemma \ref{le:fiFr} and Lemma \ref{le:fiFrs} we have well-defined $\A'$-algebra homomorphisms 
 \[
 id\otimes \ofr:{_{\A'}\widehat{\U}^{\imath,1}}\rightarrow {_{\A'}\widehat{\U}^\imath}\widehat{\otimes}{_{\A'}\widehat{\mU}} \quad \text{ and  } \quad \ifr\otimes id:{_{\A'}\widehat{\mU}^{\imath,1}}\rightarrow{_{\A'}\widehat{\U}^\imath}\widehat{\otimes}{_{\A'}\widehat{\mU}}.
 \]

\begin{prop}\label{prop:cosp}
    The following diagram commutes:
    \begin{equation}\label{dia:ifrspl}
        \begin{tikzcd}
        _{\mathcal{A}'}\dot{\mathfrak{U}}^\imath \arrow[r,"\ifr"] \arrow[d,"\widehat{\Delta}"'] & _{\mathcal{A}'}\dot{\mathrm{U}}^\imath \arrow[r,"\widehat{\Delta}"] & _{\mathcal{A}'}\widehat{\mathrm{U}}^{\imath,1} \arrow[d,"id\otimes\ofr"] \\
        _{\mathcal{A}'}\widehat{\mathfrak{U}}^{\imath,1}
        \arrow[rr, "\ifr\otimes id"] & & {_{\A'}\widehat{\U}^\imath}\widehat{\otimes}{_{\A'}\widehat{\mU}}
        \end{tikzcd}
    \end{equation}
\end{prop}

\begin{proof}
It suffices to check 
\[
(id\otimes \ofr) \circ \widehat{\Delta}\circ \ifr(u)=(\ifr\otimes id)\circ \widehat{\Delta}(u)
\]
for $u=\one_\zeta$ ($\zeta\in X_\imath$), and $u=\mathfrak{b}_{i,\zeta}$ ($i\in\I$, $\zeta\in X_\imath$). The first case is straightforward.


Assume $u=\mathfrak{b}_{i,\zeta}$. Then by definitions, we have  
\[
(\ifr\otimes id)\circ \widehat{\Delta}(\mathfrak{b}_{i,\zeta})=\sum_{\substack{\zeta'+\overline{\lambda}=\zeta\\\zeta'\in X_\imath,\lambda\in X}}\left('\RB_{i,l\zeta'}^{(l)}\otimes\one_\lambda+\one_{l\zeta'}\otimes(\fe_i\one_\lambda+\ff_i\one_\lambda)\right).
\]
We then compute $(id\otimes \ofr) \circ \widehat{\Delta}\circ \ifr(\fb_{i,\zeta})$, for any $i\in\I$ and $\zeta\in X_\imath$. We divide into two cases.


 Case (a): $\tau i=i$. 
 
 Recall Theorem~\ref{thm:ClaFr} that 
\[
\ofr(\RE_i^{(a)}\RF_i^{(a')}\one_{l\lambda}) =
\begin{cases}
\fe_i^{(a/l)}\ff_i^{(a'/l)}\one_\lambda, &\text{if } l\mid a \text{ and }l\mid a';\\
0, &\text{otherwise.}
\end{cases}
\]
Therefore by the Lemma \ref{le:conew}, we have
 \begin{align*}
 &(id\otimes \ofr) \circ \widehat{\Delta}\circ \ifr(\mathfrak{b}_{i,\zeta}) = (id\otimes \ofr) \circ \widehat{\Delta}('\B_{i,l\zeta}^{(l)}) \\
 =\, &\bq_i^{s_i(0,0,l\lambda,l,0)}{'\RB_{i,\zeta'}^{(l)}}\otimes \one_{\lambda}+\bq_i^{s_i(l,0,l\lambda,l,0)}\one_{\zeta'}\otimes \fe_i\one_\lambda+\bq_i^{s_i(0,l,l\lambda,l,0)}\one_{\zeta'}\otimes \ff_i\one_\lambda.
 \end{align*}

It is direct to check that $s_i(0,0,l\lambda,l,0)$, $s_i(l,0,l\lambda,l,0)$, and $s_i(0,l,l\lambda,l,0)$ are all divisible by $l$. We deduce that
\begin{align*}
(id\otimes \ofr) \circ \widehat{\Delta}\circ \ifr(\mathfrak{b}_{i,\zeta}) &=\sum_{\substack{\zeta'+\overline{\lambda}=\zeta\\\zeta'\in X_\imath,\lambda\in X}}\left('\RB_{i,l\zeta'}^{(l)}\otimes\one_\lambda+\one_{l\zeta'}\otimes(\fe_i\one_\lambda+\ff_i\one_\lambda)\right) \\
&=(\ifr\otimes id)\circ \widehat{\Delta}(\mathfrak{b}_{i,\zeta}).
\end{align*}

 Case (b): $\tau i\neq i$. 
 
 By direct computation in $\mathrm{U}^\imath$, we have
\[
    \Delta(B_i)=B_i\otimes \Tilde{K}_i^{-1}+1\otimes F_i+\Tilde{K}_i^{-1}\Tilde{K}_{\tau i}\otimes \varsigma_i E_{\tau i}\Tilde{K}_i^{-1}.
\]

Therefore we have 
\begin{equation*}
\begin{split}
    &\Delta(B_i^{(l)})=(B_i\otimes \Tilde{K}_i^{-1}+1\otimes F_i+\Tilde{K}_i^{-1}\Tilde{K}_{\tau i}\otimes \varsigma_i E_{\tau i}\Tilde{K}_i^{-1})^l/[l]_i^!\\
    &=B_i^{(l)}\otimes \Tilde{K}_i^{-l}+1\otimes F_i^{(l)}+(\Tilde{K}_i^{-l}\Tilde{K}_{\tau i}^l)\otimes \varsigma_i^l(E_{\tau i}\Tilde{K}_i^{-1})^{l}/[l]_i^!+\text{other terms},
\end{split}
\end{equation*}
where the “other terms" stands for the terms, such that the degree of the second factor does not belong to $l\BZ[\I]$ (recall that $\U$ is naturally $\BZ[\I]$-graded). Hence they are in the kernel of $id \otimes \ofr$.

Therefore
\begin{align*}
    &(id\otimes \ofr) \circ \widehat{\Delta}\circ \ifr(\mathfrak{b}_{i,\zeta})= (id\otimes \ofr) \circ \widehat{\Delta}(\B_{i,l\zeta}^{(l)})\\
    =\, &(id\otimes \ofr) \Big( \sum_{\substack{\zeta'+\overline{\lambda}=\zeta\\\zeta'\in X_\imath,\lambda\in X}}\big(\B_{i,l\zeta'}^{(l)}\otimes \one_{l\lambda}+\one_{l\zeta'}\otimes \RF_i^{(l)}\one_{l\lambda}+\one_{l\zeta'}\otimes \RE_{\tau i}^{(l)}\one_{l\lambda}+\text{other terms}\big) \Big)\\
    =&\sum_{\substack{\zeta'+\overline{\lambda}=\zeta\\\zeta'\in X_\imath,\lambda\in X}}\left('\RB_{i,l\zeta'}^{(l)}\otimes\one_\lambda+\one_{l\zeta'}\otimes(\fe_i\one_\lambda+\ff_i\one_\lambda)\right)\\
    =\, &(\ifr\otimes id)\circ \widehat{\Delta}(\mathfrak{b}_{i,\zeta}).
\end{align*}

This finishes the proof.
\end{proof}

\subsection{Frobenius splittings of algebraic groups}\label{sec:Frbach}
We assume $l=p$ is an odd prime number, which is relatively prime to all the root lengths.  We follow the notations in \S \ref{sec:oFrbach}.

\subsubsection{}
 Tensoring the $\A'$-algebra homomorphisms $\fr: {_{\A'}\dot{\mU}}\longrightarrow{_{\A'}\dot{\U}}$ and  $\ifr:{_{\A'}\dot{\mU}^\imath}\longrightarrow{_{\A'}\dot{\U}^\imath}$ with the $\A'$-algebra $\BF_p$,
we obtain $\BF_p$-algebra homomorphisms
\[
 \fr:{{}_{\BF_p}\!\dot{\U}}\longrightarrow{{}_{\BF_p}\!\dot{\U}}\quad \text{ and } \quad  \ifr:{{}_{\BF_p}\!\dot{\U}^\imath}\longrightarrow{{}_{\BF_p}\!\dot{\U}^\imath}.
\]
Thanks to Lemma \ref{le:fiFrs}, we also have $\BF_p$-algebra homomrophisms 
\[
\fr:{{}_{\BF_p}\!\widehat{\U}}\longrightarrow{{}_{\BF_p}\!\widehat{\U}}\quad \text{ and } \quad  \ifr:{{}_{\BF_p}\!\widehat{\U}^\imath}\longrightarrow{{}_{\BF_p}\!\widehat{\U}^\imath}.
\]
Recall the ${\BF_p}$-Hopf algebras $\RO_{\BF_p}$ in \S\ref{sec:nocp} and $\RO_{\BF_p}^\imath$ in \S\ref{sec:pro}. We hence have the induced well-defined $\BF_p$-linear maps 
\[
{ \fr}^{,*}:\RO_{\BF_p}\longrightarrow \RO_{\BF_p}\quad \text{ and } \quad  { \ifr}^{,*}:\RO_{\BF_p}^\imath\longrightarrow\RO_{\BF_p}^\imath.
\]

\begin{prop}\label{prop:splp}
(1) The map $\fr^{,*}$ is a Frobenius splitting of $\RO_{\BF_p}$, that is, we have $\fr^{,*}(fg^p)={ \fr^{,*}}(f)\cdot g$  for any $f$, $g$ in $\RO_{\BF_p}$, and $\fr^{,*}(1)=1$.

(2) The map $\ifr^{,*}$ is a Frobenius splitting of $\RO_{\BF_p}^\imath$. That is, we have $\ifr^{,*}(fg^p)={\ifr^{,*}}(f)\cdot g$ for any $f$, $g$ in $\RO_{\BF_p}^\imath$, and $\ifr^{,*}(1)=1$.
\end{prop}

\begin{proof}
We show part (1). It follows from the commuting diagram \eqref{dia:frspl} that the following diagram commutes
\begin{equation*}
    \begin{tikzcd}
        {}_{\BF_p}\!\dot{\U} \arrow[r,"\fr"] \arrow[d,"\widehat{\Delta}"'] & {}_{\BF_p}\!\widehat{\U} \arrow[r,"\widehat{\Delta}"] & {}_{\BF_p}\!\widehat{\U}^{(2)} \arrow[d,"id\otimes\ofr"] \\
        {}_{\BF_p}\!\widehat{\U}^{(2)}
        \arrow[rr, "\fr\otimes id"] & & {}_{\BF_p}\!\widehat{\U}^{(2)}
        \end{tikzcd}
\end{equation*}

Taking ${\BF_p}$-linear duals, we obtain the commutative diagram
\begin{equation}\label{dia:opsplit}
    \begin{tikzcd}
    \RO_{\BF_p} & \RO_{\BF_p} \arrow[l,"\fr^{,*}"'] & \RO_{\BF_p}\otimes\RO_{\BF_p} \arrow[l,"m"'] \\ \RO_{\BF_p} \otimes\RO_{\BF_p} \arrow[u,"m"] & & \RO_{\BF_p}\otimes\RO_{\BF_p} \arrow[ll,"\fr^{,*}\otimes id"'] \arrow[u,"id\otimes\ofr^{*}"']
    \end{tikzcd}
\end{equation}
We have seen that $\ofr^*$ is the $p$-th power map. The diagram \eqref{dia:opsplit} implies that $\fr^{,*}(fg^p)=\fr^{,*}(f)g$, for any $f,g\in \RO_{\BF_p}$. Since $\ofr\circ\fr=id$, we deduce that $\fr^{,*}(f^p)=f$, for any $f\in \RO_{\BF_p}$. Hence $\fr^{,*}$ is a Frobenius splitting of $\RO_{\BF_p}$.

Part (2) follows similarly using the commutative diagram \eqref{dia:ifrspl}.
\end{proof}



\subsubsection{}\label{sec:kFr}
 Let $k$ be an algebraically closed field with characteristic $p$. Recall $p$ is an odd prime number which is relatively prime to all the root lengths.


The algebra $\RO_k\cong k\otimes_{\BF_p}\RO_{\BF_p}$ is the coordinate ring  of the algebraic group $G_k$ associated to the root datum by \S\ref{sec:nocp}. The algebra $\RO_k^\imath\cong k\otimes_{\BF_p}\RO_{\BF_p}^\imath$ is coordinate ring of the symmetric subgroup $G_k^\theta$ associated to the $\imath$-root datum by Theorem~\ref{thm:Oik}. 


Let $(\;\cdot\;)^p:k\rightarrow k$ be the $\BF_p$-linear map sending any $t$ in $k$ to $t^p$. We define
\begin{align*}
_k\ofr^*& =(\;\cdot\;)^p\otimes_{\BF_p} { \ofr^*}: \RO_k\longrightarrow\RO_k\\ _k\iofr^*&=(\;\cdot\;)^p\otimes_{\BF_p}{\iofr^*}:\RO^\imath_k\longrightarrow\RO_k^\imath.
\end{align*}
 Here we add a subscript $_k({\cdot})$ to indicate the non-trivial automorphism $(\;\cdot\;)^p$ on $k$.



Let $(\;\cdot\;)^{1/p}:k\rightarrow k$ be the $\BF_p$-linear map  sending any $t$ in $k$ to $t^{1/p}$, the unique $p$-th root of $t$. We define
\begin{align*}
_k\fr^{,*}& =(\;\cdot\;)^{1/p}\otimes_{\BF_p}{\fr^{,*}}:\RO_k\longrightarrow\RO_k,\\ 
_k\ifr^{,*}& =(\;\cdot\;)^{1/p}\otimes_{\BF_p}{\ifr^{,*}}:\RO^\imath_k\longrightarrow\RO_k^\imath.
\end{align*}
 We again add a subscript $_k({\cdot})$ here to indicate the non-trivial automorphism $(\;\cdot\;)^{1/p}$ on $k$. 

 The following theorem is immediate by Proposition~\ref{prop:pthFr} and Proposition \ref{prop:splp}.
 
\begin{thm}\label{thm:splitO}
(1) We have 
\begin{align*}
 {_k\ofr^*}(f)&=f^p, \quad \text{for any } f\in \RO_k,\\
  {_k\iofr^*}(f)&=f^p, \quad \text{for any } f\in \RO_k^\imath.
\end{align*}

(2) The maps $_k\fr^{,*}$ and $_k\ifr^{,*}$ provide explicit Frobenius splittings for algebraic groups $G_k$ and $G^\imath_k$, respectively.
\end{thm}

\begin{remark}
(1) The existence of such splittings is guaranteed by \cite{BK05}*{Proposition~1.1.6}. Our construction is explicit and uniform for all positive characteristics (not $2$ for $G^\imath_k$ of course). 

(2) By \cite{BK05}*{Proposition~1.1.6}, there exists a splitting of $G_k$ that compatibly splits $G_k^\imath$. We expect our method can be used to construct explicitly such an Frobenius splitting. This will be addressed in future works.
\end{remark}

\section{Frobenius splittings of flag varieties}\label{sec:Frflag}
In this section, we assume given an $\imath$root datum $(\I,Y,X,A,(\alpha_i)_{i\in\I},(\coroot_i)_{i\in\I}),\tau,\theta)$ of finite (quasi-split) type. Let $p$ be an odd prime number, and relatively prime to all the root lengths (so $p>3$ if the underlying root datum has a $G_2$ factor). We fix an algebraically closed field $k$ of characteristic $p$. 

Let $(G_k, \theta_k,B_k)$ be an anisotropic triple assocaited to the $\imath$root datum following Proposition \ref{prop:cft}. We identify this triple as we constructed in \S \ref{sec:class} (using quantum groups). Let $\{x_i,y_i;i\in\I\}$ be the corresponding anisotropic pinning, and $\mathcal{B}_k=G_k/B_k$ be the flag variety.

	
	

\subsection{The algebra $R_{\BF_p}$ and $R_k$}

	We follow the notations in \S \ref{sec:oFrbach} and \S \ref{sec:Frbach}.
	For any $\lambda\in X^+$, we write
	\[
	\baV(\lambda) =\BF_p\otimes_{\A}{_\A L(\lambda)}.
	\]
	It is a ${}_{\BF_p}\!\dot{\U}$-module with maximal vector $v_\lambda^+=1\otimes v_\lambda^+$ (again we abuse notations here).
	
	Set 
	\[
	R_{\BF_p}=\bigoplus_{\lambda\in X^+}\baV(\lambda)^*.
	\]
	It has a $\BF_p$-algebra structure defined in the following way. For any $\lambda_1,\lambda_2\in X^+$, there is a unique ${}_{\BF_p}\!\dot{\U}$-module homomorphism
	\[
	\baV(\lambda_1+\lambda_2)\longrightarrow \baV(\lambda_1)\otimes \baV(\lambda_2)
	\]
	sending $v_{\lambda_1+\lambda_2}^+$ to $v_{\lambda_1}^+\otimes v_{\lambda_2}^+$.
	Taking dual of this map, we get $\BF_p$-linear maps
	\[
	\baV(\lambda_1)^*\otimes \baV(\lambda_2)^*\longrightarrow \baV(\lambda_1+\lambda_2)^*,
	\]
	which defines the multiplication in $R_{\BF_p}$. The unit of $R_{\BF_p}$ is the linear dual of $\baV(0)=\BF_p\cdot v_0^+$, sending $v_0^+$ to 1.

	Set $R_k = k\otimes_{\BF_p} R_{\BF_p} \cong \bigoplus_{\lambda\in X^+}V_k(\lambda)^*$. 
	For any $\lambda\in X^+$, recall the construction for the $k$-algebra $R_{\mathcal{L}_\lambda}$ in \S \ref{sec:algFr}. By definitions, $R_{\mathcal{L}_\lambda}$ is a $k$-subalgebra of $R_k$. 
	
	
	

	\subsection{The $p$-th power maps}
	\subsubsection{}

Retain the notations in \S \ref{sec:oFrbach} and \S \ref{sec:Frbach}. 
For any $\lambda \in X^+$, recall \S\ref{sec:qFr} the modules ${_{\A'}\mathfrak{L}}(\mu)$ and ${_{\A'}L}(\mu)$.
By the arguments in \S \ref{sec:oFrbach}, we have canonical isomorphisms (as ${}_{\BF_p}\!\dot{\U}$-modules)
\[
\BF_p\otimes_{\A'}{{_{\A'}\mathfrak{L}}(\mu)}\cong \BF_p\otimes_{\A'}{{_{\A'}L}(\mu)}\cong \baV(\mu).
\]
Hence by Proposition \ref{prop:ga}, we have a  ${}_{\BF_p}\!\dot{\U}$-module homomorphism
\[
\gamma_\mu: \baV(p\mu)\longrightarrow \baV(\mu)^{\ofr}, \quad v_{p\mu}^+ \mapsto v_\mu^+.
\]
 


Define $\BF_p$-linear map $\Gamma: R_{\BF_p}\rightarrow R_{\BF_p}$, such that $\Gamma(f)=\gamma_\mu^*(f)$, for $f\in \baV(\mu)^*$. Here $\gamma_\mu^*: (\baV(\mu)^{\ofr})^* \rightarrow \baV(p\mu)^*$ denotes the ${\BF_p}$-linear dual of $\gamma_\mu$.
 
 Recall \S\ref{sec:kFr} the map $(\;\cdot \;)^p:k\rightarrow k$, mapping $ t$ to $t^p$. We further define 
\[
\Gamma_k =(\;\cdot\;)^p\otimes \Gamma: R_k\longrightarrow R_k,
\]
\begin{proposition}[\cite{KL1}*{Theorem~1}] \label{prop:Ga}
(1) The  map $\Gamma: R_{\BF_p}\rightarrow R_{\BF_p}$ is the $p$-th power map, that is, $\Gamma(f)=f^p$, for any $f\in R_{\BF_p}$. 

(2) The  map $\Gamma_k: R_{k}\rightarrow R_{k}$ is the $p$-th power map, that is, $\Gamma(f)=f^p$, for any $f\in R_{k}$.
\end{proposition}

\begin{proof}
It suffices to show part (1). Let $f\in \baV(\mu)^*$ for some $\mu\in X^+$. By definition, the equality $\Gamma(f)=f^p$ is equivalent to the commutativity of the following diagram:
\begin{equation}\label{dia:Fr}
    \begin{tikzcd}
        \baV(p\mu) \arrow[r,"\gamma_\mu"] \arrow[d] & \baV(\mu)^{\ofr}  \arrow[d,"f"]\\
        \baV(\mu)^{\otimes p} \arrow[r,"f^{\otimes p}"] & \mathbb{F}_p.
    \end{tikzcd}
\end{equation}
Note that $\baV(\mu)^{\ofr}= \baV(\mu)$ as $\BF_p$-vector spaces.

Let $I_\mu\subset \baV(\mu)^{\otimes p}$ be the subspace spanned by elements $v_1\otimes \cdots \otimes V_{\BF_p}-v_{\sigma(1)}\otimes \cdots \otimes v_{\sigma(p)}$, for all $v_i\in V_{\BF_p}(\mu)$ and $\sigma\in S_n$. Then $I_\mu$ is moreover a ${}_{\BF_p}\!\dot{\U}$-submodule. Set $S^p\baV(\mu)=\baV(\mu)^{\otimes p}/I_\mu$ be the quotient module. 
Note that the $\BF_p$-linear map $f^{\otimes p}$ factors through $S^p\baV(\mu)$, denoted by $f^p: S^p\baV(\mu)\longrightarrow \mathbb{F}_p$. Define a $\BF_p$-linear map
\[
\vartheta: \baV(\mu)\longrightarrow \baV(\mu)^{\otimes p} \longrightarrow S^p\baV(\mu), \quad v \mapsto v^{\otimes p} \mapsto  v^p=\overline{v^{\otimes p}}. 
\]
Since $(f(v) )^p = f(v)$ for any $v \in \baV(\mu)$, we have the following commuting diagram
\begin{equation*}
    \begin{tikzcd}
    \big(\baV(\mu)^{\ofr}=\big)\baV(\mu) \ar[r, "\vartheta"] \ar[dr, "f"]& S^p\baV(\mu) \ar[d, "f^p"]\\
        &  \BF_p
    \end{tikzcd}
\end{equation*}

 Hence in order to check the commutativity of the diagram \eqref{dia:Fr}, it will suffice to show the commutativity of the following diagram:
\begin{equation}\label{eq:prop:Ga}
    \begin{tikzcd}
    \baV(p\mu) \arrow[r,"\gamma_\mu"] \arrow[d] & \baV(\mu)^{\ofr} \arrow[d,"\vartheta"] \\
    \baV(\mu)^{\otimes p} \arrow[r] & S^p\baV(\mu).
    \end{tikzcd}
\end{equation}

We first show that $\vartheta: \baV(\mu)^{\ofr}\rightarrow S^p\baV(\mu)$ is ${}_{\BF_p}\!\dot{\U}$-equivariant, that is, 
\begin{equation}\label{id:the}
    \vartheta({ \ofr(u)}\cdot v)=u\cdot \vartheta(v), \quad \text{for any } u\in{{}_{\BF_p}\!\dot{\U}}, v\in \baV(\mu).
\end{equation}
 We may assume $v$ is homogeneous of weight $\gamma$, and $u$ equals $e_i^{(a)}\one_\lambda$, or $f_i^{(a)}\one_\lambda$, for some $a\in\BN$ and $\lambda\in X$.  If $\lambda\neq p\gamma$, then it is clear that both sides of \eqref{id:the} are 0.  If $ \lambda = p \gamma$, then entirely similar to the proof of Proposition~\ref{prop:pthFr}, we have 
  \[
 e_i^{(a)}\one_{p\gamma}\cdot v^p= 
 \begin{cases} (e_i^{(a/p)}v)^p, &\text{if } p \mid a;\\
 0, &\text{otherwise}.
 \end{cases}
 \]
  Hence $ \vartheta (\ofr(e_i^{(a)}\one_{p\gamma}) \cdot v)= e_i^{(a)}\one_{p\gamma}\cdot v^p $, whence the identity \eqref{id:the}. Similar computation holds when $u=f_i^{(a)}\one_{p\gamma}$.  
 
  Now all the maps in the diagram \eqref{eq:prop:Ga} are ${{}_{\BF_p}\!\dot{\U}}$-equivariant. Since $\baV(p\mu)$ is generated by $v^+_{p\mu}$ as ${{}_{\BF_p}\!\dot{\U}}$-module and $v^+_{p\mu} \in \baV(p\mu)$ maps to the same element via the two maps, we conclude the diagram is commutative. 
\end{proof}
\subsection{The splitting $\Psi_k$}\label{sec:Phi}

\subsubsection{}

For any $\mu\in X^+$, recall \S \ref{se:L} the $\A'$-linear maps $\psi_\mu:{_{\A'}\mathfrak{L}}(\mu)\rightarrow ({_{\A'}L}(p\mu))^{\ifr}$. Applying base change to $\BF_p$, we obtain the ${}_{\BF_p}\!\dot{\U}^\imath$-module homomorphisms:
\[
\psi_\mu: \baV(\mu)\longrightarrow \baV(p\mu)^{\ifr}, \quad v_\mu^+ \mapsto v_{p\mu}^+.
\]

Define the $\mathbb{F}_p$-linear map $\Psi: R_{\BF_p} \longrightarrow R_{\BF_p} $ by
\[
\Psi(f)=
\begin{cases}
\psi_\mu^*(f), &\text{if } f\in \baV(p\mu)^*;\\
0,  &\text{if }f\in \baV(\mu)^* \text{ with }\mu\not\in pX.
\end{cases}
\]

\begin{prop}\label{prop:Psi}
The map $\Psi$ is a Frobenius splitting for the algebra $R_{\BF_p}$. Namely, it is an additive map satisfying: 

(a) $\Psi(f^pg)=f\Psi(g),$  for $f,g\in R$; 
    
(b) $\Psi(1)=1$.  

\end{prop}

\begin{proof}
Condition (b) follows easily from the definition. Now we check the condition (a).

Let $f\in \baV(\lambda)^*$ and $g\in \baV(p\mu)^*$. By definition, the equality $\Psi(f^pg)=f\Psi(g)$
holds if and only if the following diagram commutes
\begin{equation*}
    \begin{tikzcd}
    \baV(\mu+\lambda) \arrow[r] \arrow[d,"\psi_{\mu+\lambda}"'] & \baV(\mu)\otimes \baV(\lambda) \arrow[r,"\psi_\mu\otimes id"] & \baV(p\mu)^{\ifr}\otimes \baV(\lambda) \arrow[d,"g\otimes f"] \\
    \baV(p\mu+p\lambda)^{\ifr} \arrow[r] & \big(\baV(p\mu)\otimes \baV(\lambda)^{\otimes p}\big)^{\ifr} \arrow[r,"g\otimes f^{\otimes p}"] & \mathbb{F}_p.
    \end{tikzcd}
\end{equation*}
Note that $\baV(p\mu)^{\ifr} = \baV(p\mu)$ as $\BF_p$-vector spaces.

Following the proof of Proposition \ref{prop:Ga}, we define $S^p\baV(\lambda)$, and the $\BF_p$-linear map $\vartheta: \baV(\lambda)\rightarrow S^p\baV(\lambda)$ sending $v$ to $v^p=\overline{v^{\otimes p}}$. Note that the map $g\otimes f^{\otimes p}$ factors through $\baV(p\mu)\otimes S^p\baV(\lambda)$. Moreover, we have the following commuting diagram by Proposition~\ref{prop:Ga}:
\begin{equation*}
    \begin{tikzcd}
    \baV(p\mu)\otimes S^p\baV(\lambda) \arrow[rd,"g\otimes f^p"] & \baV(p\mu)\otimes \baV(\lambda) \arrow[l,"id \otimes \vartheta"'] \arrow[d,"g\otimes f"] \\ \baV(p\mu)\otimes \baV(\lambda)^{\otimes p} \arrow[u] \arrow[r,"g\otimes f^{\otimes p}"'] & \mathbb{F}_p.
    \end{tikzcd}
\end{equation*}

Hence it will suffice to check the following diagram commutes:
\begin{equation}\label{dia:v}
    \begin{tikzcd}
    \baV(\mu+\lambda) \arrow[r] \arrow[d,"\psi_{\mu+\lambda}"'] & \baV(\mu)\otimes \baV(\lambda) \arrow[r,"\psi_\mu\otimes id"] & \baV(p\mu)^{\ifr}\otimes \baV(\lambda) \arrow[d,"id\otimes \vartheta"] \\
    \baV(p\mu+p\lambda)^{\ifr} \arrow[r] & (\baV(p\mu)\otimes \baV(\lambda)^{\otimes p})^{\ifr} \arrow[r] & (\baV(p\mu)\otimes S^p\baV(\lambda))^{\ifr}.
    \end{tikzcd}
\end{equation}

Note that each space of the diagram admits a ${}_{\BF_p}\!\dot{\U}^\imath$-action, ${}_{\BF_p}\!\widehat{\U}^\imath$-action, as well as a ${}_{\BF_p}\!\widehat{\U}$-action. 
We claim that each morphism is ${}_{\BF_p}\!\widehat{\U}^\imath$-equivariant. We will only check this claim for the map 
\[
id\otimes\vartheta: \baV(p\mu)^{ \ifr}\otimes \baV(\lambda)\longrightarrow (\baV(p\mu)\otimes S^pV(\lambda))^{ \ifr}.
\]
The checkings for other maps are trivial.

Let $v\in \baV(p\mu)$, $w\in \baV(\lambda)$ and $u\in{{}_{\BF_p}\!\widehat{\U}^\imath}$. We show
\begin{equation}\label{id:id}
    (id\otimes \vartheta) ({\ifr} \otimes id \circ \widehat{\Delta}(u)   \cdot( v\otimes w))={\ifr}(u)\cdot (v\otimes w^p).
\end{equation}


It follows from Proposition~\ref{prop:cosp} that
\[
\big({ \ifr\otimes id}\big)\circ\widehat{\Delta}(u)=\big(id\otimes{ \ofr}\big)\circ \widehat{\Delta}\circ{ \ifr}(u).
\]
Recall \eqref{eq:prop:Ga} that 
\[
 \vartheta \big((  \ofr (x)\cdot v) \big)= x \cdot\vartheta (v), \text{ for } v\in \baV(\lambda)^{ \ofr}, x\in {{}_{\BF_p}\!\dot{\U}}. 
\]
Hence we have
\begin{align*}
   &(id\otimes \vartheta) ({\ifr} \otimes id \circ \widehat{\Delta}(u) \cdot( v\otimes w))\\
    &=\big(id\otimes \vartheta\big)\big((id\otimes {\ofr})\circ \widehat{\Delta}\circ {\ifr(u)}\cdot (v\otimes w)\big)\\
    &=\big(\widehat{\Delta}\circ {\ifr(u)}\big)\cdot \big(id\otimes \vartheta\big)\big(v\otimes w \big)\\
    &={\ifr(u)}\cdot (v\otimes w ^p).
\end{align*}
This proves the identity \eqref{id:id}.

The highest weight vector $v_{\mu+\lambda}^+$ maps to the same vector $v_{p\mu}^+\otimes (v_\lambda^{+})^p$ via both compositions by the following computation
\begin{equation*}
    \begin{tikzcd}
    v_{\mu+\lambda}^+ \arrow[r,mapsto] \arrow[d,mapsto] & v_\mu^+\otimes v_\lambda^+ \arrow[r,mapsto] & v_{p\mu}^+\otimes v_\lambda^+ \arrow[d,mapsto] \\ v_{p\mu+p\lambda}^+ \arrow[r,mapsto] & v_{p\mu}^+\otimes (v_\lambda^{+})^{\otimes p} \arrow[r,mapsto] & v_{p\mu}^+\otimes (v_\lambda^{+})^p.
    \end{tikzcd}
\end{equation*}

Since all morphisms in the diagram \eqref{dia:v} are ${}_{\BF_p}\!\dot{\U}^\imath$-equivariant, and $\baV(\mu+\lambda)$ is generated by $v_{\mu+\lambda}^+$ as a ${}_{\BF_p}\!\widehat{\U}^\imath$-module, the proof follows now.
\end{proof}

\subsubsection{}\label{sec:sptfl}
Recall \S\ref{sec:kFr} the map  $(\;\cdot\;)^{1/p}:k\rightarrow k$ sends $t$ to $t^{1/p}$. Define
\[
\Psi_k =(\;\cdot\;)^{1/p}\otimes \Psi:  k \otimes_{\BF_p} R_{\BF_p}= R_k\longrightarrow R_k.
\]
For any $\lambda\in X^{++}$, it is clear that the subalgebra $k$-subalgebra $R_{\mathcal{L}_\lambda}$ is preserved by $\Psi_k$. Thanks to Proposition \ref{prop:Psi} and Lemma~\ref{le:algFrfl}, we make the following conclusion.

(a) {\em The map $\Psi_k$ induces a Frobenius splitting  of the flag variety $\CB_k$. }

The splitting is independent of the choice of $\lambda\in X^{++}$ by a similar argument as \cite{KL2}*{Lemma~6.3}.
 
\subsubsection{}

 Recall Proposition~\ref{prop:Korbits} the construction of  codimension-one $K_k$-orbits on $\CB_k$. 
For any $\mu\in X^{++}$, and  $i\in I$   with $\tau i\neq i$, we consider the restriction  map
$
r_{\mu,i}:H^0(\mu)=H^0(\CB_k,\mathcal{L}_\mu)\longrightarrow H^0(\overline{\mO_i},\mathcal{L}_\mu)$. Let $I_{\mu,i}$ be the kernel.

\begin{lemma}\label{lem:Iui}
Under the isomorphism $V_k(\mu)^*\cong H^0(\mu)$, the subspace $I_{\mu,i}$ consists of linear forms vanishing on $_k\dot{\U}^\imath\cdot n_iv_\mu^+$, which equals the $_k\dot{\U}^\imath$-submodule of $V_k(\mu)^*$ generated by the extremal vector of weight $s_i\mu$, where $s_i \in W$ is the simple reflection associated to $i$.
\end{lemma}

\begin{proof}
For any $f\in V_k(\mu)^*$, we identify $f$ with a section in $H^0(\mu)$. Then $r_{\mu,i}(f)=0$ $\Leftrightarrow$  $f\mid_{\mathcal{O}_i} = 0$ $\Leftrightarrow$  $f(kn_i\cdot v_\mu^+)=0$, for all $k\in K$ $\Leftrightarrow$ $f$ vanishes on the $K$-submodule generated by the extremal vector $n_iv_\mu^+$, which is exactly $_k\dot{\U}^\imath\cdot n_i v_\mu^+$, thanks to Proposition \ref{prop:KvsUi}. Hence the lemma is proved.
\end{proof}

\begin{theorem}\label{thm:spl1}
The map $\Psi_k: R_k \rightarrow R_k$ provides a Frobenius splitting of $\CB_k$, which compatibly splits subvarieties $\overline{\mO_i}$, for all $i\in \I$  with $\tau i\neq i$. 
\end{theorem}

\begin{proof}
Recall \S\ref{sec:Phi}. The ${}_{\BF_p}\!\dot{\U}^\imath$-module homomorphism
\[
\psi_\mu: \baV(\mu)\longrightarrow \baV(p\mu)^{\ifr}
\]
sends ${}_{\BF_p}\!\dot{\U}^\imath$-submodules to ${}_{\BF_p}\!\dot{\U}^\imath$-submodules. We compute
\[
\psi_\mu(f^{(\mu_i)}_i\one_\mu\cdot v_\mu^+)=\psi_\mu(b_{i,\bar{\mu}}^{(\mu_i)}\cdot v_\mu^+)=b_{i,\overline{p\mu}}^{(p\mu_i)}\cdot v_{p\mu}^+=f_i^{(p\mu_i)}\cdot v_{p\mu}^+, \quad \text{where }\mu_i=\langle \coroot_i,\mu\rangle.
\]
 Hence $\psi_\mu$ sends the weight space $\baV(\mu)_{s_i\mu}$ to the weight space $\baV(p\mu)_{s_i(p\mu)}$.

Thus by definition, we have \[
\Psi_k\mid_{H^0(p\mu)}:H^0(p\mu)\longrightarrow H^0(\mu), \quad \text{sending }I_{p\mu,i} \mapsto I_{\mu,i}.
\]
The theorem now follows from \S\ref{sec:sptfl} and  Lemma~\ref{le:algFrfl}.
\end{proof}


 
\subsection{The splitting $\Psi^J_k$} \label{sec:PhiJ}
Recall  \S \ref{sec:propJ} and the notations in \S\ref{sec:oFrbach}. We fix a subset $J \subset \I$ with the property (*). We  denote by $b_{J,\bar{\lambda}}^{(p-1)}$ the image of $\RB_{J,\bar{\lambda}}^{(p-1)}$ in ${}_{\BF_p}\!\dot{\U}^\imath$ under the canonical isomorphism $\BF_p\otimes_{\A'}{_{\A'}\dot{\U}^\imath}\cong{{}_{\BF_p}\!\dot{\U}^\imath}$.
\subsubsection{}




For any $\mu\in X^+$, recall the map $\psi_\mu^J:{_{\A'}\mathfrak{L}}(\mu)\rightarrow ({_{\A'}L}(p\mu))^{\ifr}$ in Proposition \ref{prop:faJ}. Applying base change, we obtain the ${}_{\BF_p}\!\dot{\U}^\imath$-module homomorphism:
\[
\psi_\mu^J:\baV(\mu)\longrightarrow \baV(p\mu)^{\ifr}, \quad v_\mu^+ \mapsto b_{J,\overline{p\mu}}^{(p-1)}\cdot v_{p\mu}^+.
\]

We denote its $\BF_p$-linear dual by $(\psi_\mu^J)^*$. Define the $\BF_p$-linear map $\Psi^J: R\longrightarrow R$ by 
\[
\Psi^J(f) = 
\begin{cases}
(\psi_\mu^J)^*(f), &\text{if }f\in \baV(p\mu)^*;\\
0, &\text{if }f\in \baV(\mu)^*  \text{ with }\mu\not\in pX.
\end{cases}
\]

\begin{prop}\label{prop:PsiJ}
The map $\Psi^J$ is a Frobenius splitting for the algebra $R_{\BF_p}$. 
\end{prop}

\begin{proof}

The proof is similar to the proof of the Proposition \ref{prop:Psi}. It suffices to check the commutativity of the diagram
\begin{equation}\label{dia:vp}
    \begin{tikzcd}
    \baV(\mu+\lambda) \arrow[r] \arrow[d,"\psi_{\mu+\lambda}^J"'] & \baV(\mu)\otimes \baV(\lambda) \arrow[r,"\psi_{\mu}^J\otimes id"] & \baV(p\mu)^{\ifr}\otimes \baV(\lambda) \arrow[d,"id\otimes \vartheta"] \\
    \baV(p\mu+p\lambda)^{\ifr} \arrow[r] & (\baV(p\mu)\otimes \baV(\lambda)^{\otimes p})^{\ifr} \arrow[r] & (\baV(p\mu)\otimes S^p\baV(\lambda))^{\ifr}.
    \end{tikzcd}
\end{equation}
One can show that all the maps in the above diagram is ${}_{\BF_p}\!\widehat{\U}^\imath$-equivariant. Hence it will suffice to keep track of the maximal vector:
\begin{equation*}
    \begin{tikzcd}
    v^+_{\mu+\lambda} \arrow[r,mapsto] \arrow[d,mapsto] & v^+_\mu\otimes v^+_\lambda \arrow[r,mapsto] & \left(b_{J,\overline{p\mu}}^{(p-1)}v_{p\mu}^+\right)\otimes v_\lambda^+ \arrow[d,mapsto]\\ b_{J,\overline{p\mu+p\lambda}}^{(p-1)}v_{p\mu+p\lambda}^+ \arrow[r,mapsto] & b_{J,\overline{p\mu+p\lambda}}^{(p-1)}\cdot \left(v_{p\mu}^+\otimes (v_\lambda^+)^{\otimes p}\right) \arrow[r,mapsto,"(\heartsuit)"] & \left(b_{J,\overline{p\mu}}^{(p-1)}v_{p\mu}^+\right)\otimes (v_\lambda^+)^p.
    \end{tikzcd}
\end{equation*}

The $(\heartsuit)$ follows from the fact that 
    $f_i^{(s)}\cdot (v_\lambda^+)^{\otimes p}\in I_\lambda$,
for any $0<s<p$  and $i\in\I$ (cf. the proof of Proposition~\ref{prop:Ga}). Here $I_\lambda\subset V(\lambda)^{\otimes p}$ stands for the kernel of the quotient map $V(\lambda)^{\otimes p}\rightarrow S^p\baV(\lambda)$.

Hence the diagram \eqref{dia:vp} commutes. This completes the proof for the proposition.
\end{proof}
 
\subsubsection{}

For $\lambda\in X^+$,  $i\in \I$ with $\tau i=i$, and $\epsilon=\pm 1$, 
we define
\[
v_{\lambda,i,\epsilon} =\sum_{a\geqslant 0}\epsilon^{\langle \coroot_i,\lambda\rangle+a}f_i^{(a)}v_\lambda^+ \in \baV(\lambda).
\]

By \cite[(2.16) \& (2.17)   \& (3.9)\& (3.8)]{BerW18}, we have 
\[
    b_{i,\bar{\lambda}}^{(\langle \coroot_i,\lambda\rangle )}v_\lambda^+=\sum_{t\geqslant 0}f_i^{(\langle \coroot_i,\lambda\rangle -2t)}v_\lambda^+,\qquad  
    b_{i,\bar{\lambda}}^{(\langle \coroot_i,\lambda\rangle -1)}v_\lambda^+=\sum_{t\geqslant 0}f_i^{(\langle \coroot_i,\lambda\rangle-2t-1)}v_\lambda^+.
\]
Hence we conclude 
\begin{equation}\label{eq:bv}
v_{\lambda,i,1}=(b_{i,\bar{\lambda}}^{(\langle \coroot_i,\lambda\rangle )}+b_{i,\bar{\lambda}}^{(\langle \coroot_i,\lambda\rangle -1)})v_\lambda^+, \quad v_{\lambda,i,-1}=(b_{i,\bar{\lambda}}^{(\langle \coroot_i,\lambda\rangle )}-b_{i,\bar{\lambda}}^{(\langle \coroot_i,\lambda\rangle -1)})v_\lambda^+.
\end{equation}



\begin{lemma}\label{le:bb}
 For $i\in \I$ with $\tau i=i$, $\zeta\in X_\imath$, and $n\in\mathbb{N}$, we have
 \[
     'b_{i,\zeta}^{(np)}\cdot {'b_{i,\zeta}^{(p-1)}}={'b_{i,\zeta}^{(np+p-1)}}  \quad \text{in } {{}_{\BF_p}\!\dot{\U}^\imath}.
 \]
\end{lemma}

\begin{proof}
By Lemma~\ref{lem:balanced} in $\mathrm{U}^\imath$, we have
\[
    B_i^{(np)}B_i^{(p-1)}=\sum_{t\geqslant 0}\qbinom{np+p-1}{np}_i\prod_{m=1}^t\frac{[np-2m+2]_i[p-2m+1]_i}{[np+p-2m]_i[2m]_i}B_i^{(np+p-1-2t)}.
\]

Note that
\[
    \qbinom{np+p-1}{np}_i\prod_{m=1}^t\frac{[np-2m+2]_i[p-2m+1]_i}{[np+p-2m]_i[2m]_i} = 0, \quad \text{unless } 0\leqslant t\leqslant (p-1)/2.
\]

Assume $0\leqslant t\leqslant (p-1)/2$, then the above element belongs to $\mathcal{A}_{f_p}$, i.e., the denominator is not divisible by $f_p$.
Then since $\phi([np]_i)=0$, we have
\[
    \phi\left(\qbinom{np+p-1}{np}_i\prod_{m=1}^t\frac{[np-2m+2]_i[p-2m+1]_i}{[np+p-2m]_i[2m]_i}\right)=0, \quad \text{if } t \neq 0.
\]
Finally, when $t=0$, we obtain $\phi\left(\qbinom{np+p-1}{np}_i\right)=1$. This proved the lemma.
\end{proof}


\begin{lem}\label{lem:pj}
For any $j\in J$, the element $\psi_\mu^J(v_{\mu,j,\epsilon})$ belongs to the ${}_{\BF_p}\dot{\U}^\imath$-submodule of $\baV(p\mu)$ generated by the vector $v_{\mu,j,\epsilon}$.
\end{lem}

\begin{proof}
We write $\mu_j = \langle \coroot_j, \mu \rangle $. Thanks to \eqref{eq:bv}, it suffices to prove that 
\[
\psi_\mu^J(b_{j,\bar{\mu}}^{(\mu_j)}v_\mu^+)\in {{}_{\BF_p}\!\dot{\U}^\imath}\cdot b_{j,\overline{p\mu}}^{(p\mu_j)}v_{p\mu}^+ \quad\text{and}\quad\psi_\mu^J(b_{j,\bar{\mu}}^{(\mu_j-1)}v_\mu^+)\in {{}_{\BF_p}\!\dot{\U}^\imath}\cdot b_{j,\overline{p\mu}}^{(p\mu_j-1)}v_{p\mu}^+
.\]

By definition, 
\begin{equation}\label{id:jm}
\psi_\mu^J(b_{j,\bar{\mu}}^{(\mu_j)}v_\mu^+)=\ifr (b_{j,\bar{\mu}}^{(\mu_j)})b_{J,\overline{p\mu}}^{(p-1)}v_{p\mu}^+=b_{J\backslash \{j\},\overline{p\mu}}^{(p-1)}\cdot \ifr(b_{j,\bar{\mu}}^{(\mu_j)})'b_{j,\overline{p\mu}}^{(p-1)}v_{p\mu}^+.
\end{equation}

By the proof of Corollary \ref{cor:iFr} and Lemma \ref{le:bb}, we have 
\begin{equation*}
    \begin{split}
        \ifr(b_{j,\overline{\mu}}^{(\mu_j)})\cdot {'b_{i,\overline{p\mu}}^{(p-1)}}&=\sum_{t\geqslant 0}\lr{1}{2t}{'b_{j,\overline{p\mu}}^{(p\mu_j-2tp)}}\cdot {'b_{i,\overline{p\mu}}^{(p-1)}}\\&=\sum_{t\geqslant 0}\lr{1}{2t} {'b}_{i,\overline{p\mu}}^{(p\mu_j-2tp+p-1)}\\
        &=\sum_{t\geqslant 0}\sum_{s\geqslant 0}\lr{1}{2t} \lr{-1}{2s} {b}_{i,\overline{p\mu}}^{(p\mu_j-2tp+p-1-2s)}\\
        &=\sum_{k\geqslant 0}\sum_{t\geqslant 0}\lr{1}{2t}\lr{-1}{2k-1+p-2tp}b_{i,\overline{p\mu}}^{(p\mu_j-2k)}.
    \end{split}
\end{equation*}

For any $k\geqslant 0$, let $k_0\in [0,p)$, such that $k_0\equiv k+(p-1)/2\; \text{mod }p$. By the Lemma \ref{le:qBinomAtUnity}, we have
\begin{align*}
\lr{-1}{2k-1+p-2tp} & = \lr{p-1}{2k_0}\lr{-p}{2k-2k_0-1+p-2tp}\\
&=\lr{p-1}{2k_0}\lr{-1}{(2k-2k_0-1)/p+1-2t}.
\end{align*}

Hence, for any $k \geqslant 0$, we have
\begin{align*}
&\sum_{t\geqslant 0}\lr{1}{2t}\lr{-1}{2k-1+p-2tp}\\
&=\lr{p-1}{2k_0}\sum_{t\geqslant 0}\lr{1}{2t}\lr{-1}{(2k-2k_0-1)/p+1-2t}\\&=\lr{p-1}{2k_0}\delta_{0,k+(p-1)/2-k_0}. \quad (\text{By Lemma~\ref{lem:DoubleInA2}})
\end{align*}

When $k+(p-1)/2-k_0= 0$, we have
\[
\lr{p-1}{2k_0}=\binom{(p-1)/2}{k_0}=\binom{(p-1)/2}{(p-1)/2+k} = \delta_{0,k}.
\]
Hence we conclude that
\[
\sum_{t\geqslant 0}\lr{1}{2t}\lr{-1}{2k-1+p-2tp}=\delta_{0,k},
\]
and hence 
$\ifr(b_{j,\overline{\mu}}^{(\mu_j)})\cdot {'b_{i,\overline{p\mu}}^{(p-1)}}=b_{j,\overline{p\mu}}^{(p\mu_j)}$.
Now \eqref{id:jm}  implies 
\[
\psi_\mu^J(b_{j,\bar{\mu}}^{(\mu_j)}v_\mu^+)=b_{J\backslash \{j\},\overline{p\mu}}^{(p-1)}\cdot b_{j,\overline{p\mu}}^{(p\mu_j)}v_{p\mu}^+\in {{}_{\BF_p}\!\dot{\U}}^\imath\cdot b_{j,\overline{p\mu}}^{(p\mu_j)}v_{p\mu}^+.
\]

For another relation, still by (the easier version of) Corollary \ref{cor:iFr} and Lemma \ref{le:bb}, we have 
\begin{align*}
\psi_\mu^J(b_{j,\overline{\mu}}^{(\mu_j-1)}v_{p\mu}^+)&=\ifr(b_{j,\overline{\mu}}^{(\mu_j-1)})b_{J,\overline{p\mu}}^{(p-1)}v_{p\mu}^+=b_{J\backslash \{j\},\overline{p\mu}}^{(p-1)}\cdot {'b_{j,\overline{p\mu}}^{(p\mu_j-p)}}{'b_{j,\overline{p\mu}}^{(p-1)}}v_{p\mu}^+\\
&=b_{J\backslash \{j\},\overline{p\mu}}^{(p-1)}\cdot {'b_{j,\overline{p\mu}}^{(p\mu_j-1)}}v_{p\mu}^+ \in {{}_{\BF_p}\!\dot{\U}^\imath}\cdot b_{j,\overline{p\mu}}^{(p\mu_j)} v_{p\mu}^+. 
\end{align*}

Hence we complete the proof for the lemma.
\end{proof}

\subsubsection{}\label{sec:spltJk}
Recall \S\ref{sec:kFr} the map $(\;\cdot\;)^{1/p}: k \rightarrow k$ mapping $t$ to $t^{1/p}$. Define 
\[
\Psi_k^J =(\;\cdot\;)^{1/p}\otimes \Psi^J:R_k\longrightarrow R_k.
\]
For any $\lambda\in X^{++}$, it is clear that the subalgebra $k$-subalgebra $R_{\mathcal{L}_\lambda}$ is preserved by $\Psi_k^J $. Thanks to Proposition \ref{prop:PsiJ} and Lemma~\ref{le:algFrfl}, we make the following conclusion. 

(a) {\em The map $\Psi_k^J $ induces a Frobenius splitting  of the flag variety $\CB_k$.}

The splitting is independent of the choice of $\lambda\in X^{++}$ by a similar argument as \cite{KL2}*{Lemma~6.3}.


 
\subsubsection{}

Recall that $\mO_j^\pm=K_ky_j(\pm 1)B_k/B_k$ are codimensional one $K_k$-orbits on the flag variety $\CB_k$ by Proposition~\ref{prop:Korbits}, for any $ j \in J$.


For any $\mu\in X^+$ and $j \in J$, let $I_{\mu,j}^+$ be the kernel of the restriction map
\[
r_{\mu,i}:H^0(\mu)=H^0(\CB_k,\mathcal{L}_\mu)\longrightarrow H^0(\overline{\mO_j^+},\mathcal{L}_\mu).
\]
Similarly define the subspace $I_{\mu,j}^-$ of $H^0(\mu)$. Similar to Lemma~\ref{lem:Iui}, we see that 
the subspace $I_{\mu,j}^+$ (resp. $I_{\mu,j}^-$) consists of linear forms vanishing on $_k\dot{\U}^\imath\cdot y_j(1)v_\mu^+$ (resp. ${}_k\dot{\U}^\imath\cdot y_j(-1)v_\mu^+$), under the isomorphism $V_k(\mu)^*\cong H^0(\mu)$.

\begin{theorem}\label{thm:spl2}
The map $\Psi^J_k: R_k \rightarrow R_k$ induces a Frobenius splitting of $\CB_k$, which compatibly splits subvarieties $\overline{\mO_j^\pm}$, for all $j\in J$.
\end{theorem}

\begin{proof}Let $j \in J$. 
Note that $y_j(1)v_\mu^+=\sum_{a\geqslant 0}f_j^{(a)}\one_\mu\cdot v_\mu^+$ and similarly $y_j(-1)v_\mu^-=\sum_{a\geqslant 0}(-1)^af_j^{(a)}\one_\mu\cdot v_\mu^+$. Hence by Lemma~\ref{lem:pj}, the map 
\[
\Psi_k^J\mid_{H^0(p\mu)}: H^0(p\mu)\longrightarrow H^0(\mu), \quad \text{ mapping } I_{p\mu,j}^\pm \text{ to }I_{\mu,j}^\pm. 
\]

Take $\lambda\in X^{++}$. We deduce that the map $\Psi_k^J$ induces a Frobenius splitting for the algebra $R_{\mathcal{L}_\lambda}$ which preserves the ideals $I_{\mathcal{L}_\lambda,\overline{\mO_i^\pm}}$, for any $i\in J$. The theorem follows from \S\ref{sec:spltJk} and Lemma \ref{le:algFrfl}.
\end{proof}


\subsection{Results on partial flags}

Let $P = P_S$ be a standard parabolic subgroup of $G_k$ (recall we fixed a pinning of $G_k$), corresponding to $S\subset\I$. Let $\pi=\pi_{P}:G_k/B_k\rightarrow G_k/P$ be the projection. Let $\mu\in X^+$ be such that $\langle \alpha_j^\vee,\mu\rangle=0$ for any $j\in S$. Then $\mu$ extends to a character of $P$. Let
\[
\mathcal{L}^P_\mu =G_k\times^{P}k_\mu\longrightarrow G_k/P
\]
be the associated line bundle of $G_k/P$. Then $\pi_P^*(\mathcal{L}^P_\mu) \cong \mathcal{L}_\mu$, as line bundles on $G_k/B_k$. Hence we have
\[
H^0(G_k/P,\mathcal{L}^P_\mu)\cong H^0(G_k/B_k,\mathcal{L}_\mu)\cong V_k(\mu)^*.
\]
The $k$-algebra $R_{\mathcal{L}^P_\mu} =\bigoplus_{n\geq 0}H^0(G_k/P,\mathcal{L}^P_{n\mu})$ is isomorphic to a subalgebra of $R_k$. It is clear that the splitting maps defined in \ref{sec:sptfl} and \ref{sec:spltJk} restrict to $R_{\mathcal{L}^P_\mu}$.  

\begin{corollary}\label{cor:parab}
The map $\Psi_k$ induces a Frobenius splitting of $G_k/P$, which compatibly splits $\pi(\overline{\mathcal{O}_i})$, for all $i\in\I$  with $\tau i\neq i$.

For any $J\subset\I$ with the property (*), the map $\Psi_k^J$ induces a Frobenius splitting of $G_k/P$, which compatibly splits $\pi\big(\overline{\mathcal{O}_j^\pm}\big)$  for all $j\in J$.
\end{corollary}

\subsection{Geometric consequences} \label{sec:geo}
Theorem~\ref{thm:spl1}, Theorem~\ref{thm:spl2}, and Corollary~\ref{cor:parab} imply standard consequences on $K_k$-orbit closures following \S\ref{sec:algFr}. 


We futher obtain the following stronger vanishing result on semi-ample line bundles. 


\begin{theorem}
Let $i\in\I$ be such that $\langle\coroot_j,\alpha_i\rangle\geqslant -1$, for any $j\in\I$, and be such that $\langle\coroot_i,\mu\rangle >0$ and $\langle\coroot_{\tau i},\mu\rangle >0$, for any $\mu\in X^+$. We have

(a) the restriction $H^0(G_k/B_k,\mathcal{L}_\mu)\rightarrow H^0(\overline{\mathcal{O}_i^\pm},\mathcal{L}_\mu)$ is surjective;

(b) $H^n(\overline{\mathcal{O}_i^\pm},\mathcal{L}_\mu)=0$, for $n>0$.

Here we write $\mathcal{O}_i^\pm=\mathcal{O}_i$, if $\tau i\neq i$.
\end{theorem}

\begin{proof}
Set $S=\{j\in\I:\langle \alpha_j^\vee,\mu\rangle=0\}$. Let $P =P_S$ be the parabolic subgroup associated to the subset $S$, and $\pi:G_k/B_k\rightarrow G_k/P$ be the projection. It can be deduced from the case study in \S \ref{sec:Korbits} that $\overline{\mathcal{O}_i^\pm}=\pi^{-1}\circ\pi(\overline{\mathcal{O}_i^\pm})$. Hence the restriction $\pi_{i}^\pm:\overline{\mathcal{O}_i^\pm}\rightarrow\pi(\overline{\mathcal{O}_i^\pm})$ is a locally trivial fibration with fibre isomorphic to $P/B_k$. By the projection formula, we have isomorphisms (for any $n$)
\[
\pi^*\!:\!H^n(G_k/P,\mathcal{L}^{P}_\mu)\xrightarrow{\sim}H^n(G_k/B_k,\mathcal{L}_\mu), \,\, \pi_{i}^{\pm,*}\!:\!H^n(\pi\big(\overline{\mathcal{O}_i^\pm}\big),\mathcal{L}_\mu^{P})\xrightarrow{\sim}H^n(\overline{\mathcal{O}_i^\pm},\mathcal{L}_\mu).
\]

We also have the commutative diagram
\[
\begin{tikzcd}
H^0(G_k/P,\mathcal{L}_\mu^P) \arrow[r]  \isoarrow{d} & H^0(\pi_P(\overline{\mathcal{O}_i^\pm}),\mathcal{L}_\mu^P) \isoarrow{d}\\ H^0(G_k/B_k,\mathcal{L}_\mu) \arrow[r] & H^0(\overline{\mathcal{O}_i^\pm},\mathcal{L}_\mu).
\end{tikzcd}
\]

It follows from Corollary \ref{cor:parab} that $G_k/P$ is Frobenius split compatibly with $\pi_P(\overline{\mathcal{O}_i^\pm})$. Note that the assumption on $i$ guarantees $\{i\}$ is a subset satisfying property (*), if $\tau i=i$. Since $\mathcal{L}_\mu^P$ is a very ample line bundle on $G/P$, by Corollary \S \ref{sec:algFr} (c), we deduce that the restriction $H^0(G_k/P,\mathcal{L}_\mu^P)\rightarrow H^0(\pi_P(\overline{\mathcal{O}_i^\pm}),\mathcal{L}_\mu^P)$ is surjective, and $H^n(\pi_P(\overline{\mathcal{O}_i^\pm}),\mathcal{L}_\mu^P)=0$ for $n>0$. Hence the theorem follows.
\end{proof}

\begin{remark}
If $\tau i\neq i$ (or, $\tau i = i$ with $\mathcal{O}_i^+\neq\mathcal{O}_i^-$), the orbit closures $\overline{\mathcal{O}_i^\pm}$ are multiplicity-free. In these cases, a stronger statement is obtained by Brion \cite{Br01b}*{Corollary 8}, in which he showed the results for any dominant weights $\mu$. 

Our result is new in the case $\tau i=i$ with $\mathcal{O}_i^+=\mathcal{O}_i^-$, which corresponds to $\overline{\mathcal{O}_i^\pm}$ being multiplicity-two. The assumption on $\mu$ cannot be dropped by \cite{Br01b}*{Proposition 10}.
\end{remark}

Let $\CO$ be a $K_k$-orbit on the flag variety $\CB_k=G_k/B_k$. Recall from Proposition \ref{prop:Korbits} (3) that the 
parameterization of such orbits is independent of the base filed. We denote by $\CB_{\BZ[2^{-1}]}$ the flag scheme over $\BZ[2^{-1}]$ (cf. \cite{Jan03}*{Part II, \S1}).
	
	\begin{prop}\label{prop:KorbitsZ}
	There exists a closed subscheme $\CZ=\CZ (\mathcal{O})$ of $\CB_{\BZ[2^{-1}]}$ (over $\BZ[2^{-1}]$), such that, 
	
	(1) $\CZ\rightarrow Sp\,\BZ[2^{-1}]$ is flat;
	
	(2) there is a nonempty open subset $U$ of $Sp\,\BZ[2^{-1}]$ such that for any algebraically closed field $k'$ and a morphism $Sp\,k'\rightarrow U$, the base change $\CZ_{k'} = \CZ \times_{Sp\, \BZ[2^{-1}]} Sp\, k'$ gives the closure of the corresponding $K_{k'}$-orbit on $\mathcal{B}_{k'}$. In particular, $\CZ_{k'}$ is reduced.
	\end{prop}
	
	\begin{proof}
	Thanks to Proposition~\ref{prop:Korbits}, a representative of $\CO$ can be chosen in the  $\BZ[2^{-1}]$-points of $\CB_{\BZ[2^{-1}]}$. Let $\G^\imath_{\BZ[2^{-1}]} \rightarrow \CB_{\BZ[2^{-1}]}$ be the orbit map. Let $\CZ = \CZ (\mathcal{O})$ be the scheme-theoretical image (cf. \cite{Ha77}*{Exercise II 3.11(d)}). The rest follows from similar arguments as in the proof of \cite{MR85}*{Lemma 3}.
	\end{proof}

    \begin{remark}
	It is an open question if one can take $U=Sp\,\BZ[2^{-1}]$ for part (2). Similar question for Schubert varieties is addressed in \cite{Se83}*{Theorem 2 (ii)}.
	\end{remark}

Thanks to the above lemma and semiconuity, all the results stated in this subsection remain true over algebraically closed fields with characteristic 0 (cf. \cite{BK05}*{\S 1.6}).

\subsection{Normality: type $\textrm{AIII}$}\label{sec:AIII}
Let $k$ be an algebraically closed field of characteristic $>2$. We consider the type AIII symmetric pairs in this subsection. We illustrate how to deduce normality from our splittings, which requires detailed information of the Bruhat order of $K_k$-orbits on $\CB_k$. 

\subsubsection{}

Let $G_k=GL_{n,k}$ and $a=\begin{pmatrix}
I_{\lceil n/2\rceil} & 0 \\ 0 & -I_{\lfloor n/2 \rfloor}
\end{pmatrix} \in G_k$.
Let $\theta_k: G_k \rightarrow G_k$ be the conjugation by $a$. Then $K_k=G_k^{\theta_k}=GL_{\lceil n/2\rceil,k}\times GL_{\lfloor n/2 \rfloor,k}$. The symmetric pair $(G_k,\theta_k)$ is quasi-split. The underlying Satake diagram is of type $\textrm{AIII}_{n-1}$. 

Let $B^{st}_k$ be the subgroup of $G_k$ consisting of upper-triangular matrices, and $T^{st}_k$ be the subgroup of diagonal matrices. Then $B^{st}_k$ is a $\theta$-stable Borel subgroup. Set  
\[
g=\begin{pmatrix}
1 & & & & & -1\\ & 1 & & & -1 & \\ & & \ddots &   \reflectbox{$\ddots$}  & & \\ & & \reflectbox{$\ddots$} & \ddots & & \\ & 1 & & & 1 & \\ 1 & & & & & 1
\end{pmatrix} \in G_k.
\]

The middle entry of $g$ can be either $1$ or $-1$ when $n$ is odd.
Let $B_k=gB_k^{st}g^{-1}$, and $T_k=gT^{st}_kg^{-1}$. Then $B_k$ is a $\theta$-anisotropic Borel subgroup and $T_k = B_k \cap \theta_k(B_k)$.

Let $W$ be the (absolute) Weyl group, and $\I=\{1,2,\dots,n-1\}$ be the index set for the simple roots. The graph automorphism $\tau$ induced by $\theta$ is given by $\tau(i)=n-i$, for any $i\in\I$. Let $\mathcal{B}_k$ denote the flag variety. 

 Let $\CV$ denote the set of $K_k$-orbits of $\mathcal{B}_k$, and let $\mathcal{O}_0\in \CV$ be the open orbit.  The Weyl group $W$ acts on the set $\CV$ by \cite{RS90}*{\S 2} by identifying $W \cong N(T_k) / T_k$. Set $J=\{1,2,\dots,\lfloor n/2\rfloor-1\}$, and $J'=\tau(J)=\{n-1,n-2,\cdots,\lceil n/2\rceil+1\}$ be two subsets of $\I$. Let $W_J$ and $W_{J'}$ be the parabolic subgroup of $W$ associated to $J$ and $J'$, respectively. For any $j\in J\cup J'$, we have $\tau j\neq j$. Set $\CV'=W_{J'}\cdot\mathcal{O}_0\cup W_J\cdot\mathcal{O}_0$. Note that $W_{J'}\cdot\mathcal{O}_0=W_J\cdot\mathcal{O}_0$, when $n$ is even.

\begin{proposition}
The splitting defined in \S \ref{sec:sptfl} compatibly splits all the orbits in $\CV'$ simultaneously.
\end{proposition}

\begin{proof}
Let $\mathcal{O}\in \CV'$. Then $\mathcal{O}=K_knB_k/B_k$, for some $n\in N(T_k)$, representing $w\in W_J\cup W_{J'}$. Let $w=s_{i_1}s_{i_2}\cdots s_{i_r}$ be a reduced expression for $w$, with $i_t$ ($1\leqslant t\leqslant r$) either all belong to $J$ , or all belong to $J'$. 

Let $\mu\in X^+$. Note that in the $\U$-module $L(\mu)$, we have
\[
B_{i_1}^{(n_1)}B_{i_2}^{(n_2)}\cdots B_{i_r}^{(n_r)}v_\mu^+=F_{i_1}^{(n_1)}F_{i_2}^{(n_2)}\cdots F_{i_r}^{(n_r)}v_\mu^+
\]
which is a nonzero vector of extremal weight $w\cdot\mu$, where $n_t=\langle \coroot_{i_t},\mu-n_r\alpha_{i_r}-\cdots-n_{t+1}\alpha_{i_{t+1}}\rangle$, for $1\leqslant t\leqslant r$. Then the proposition follows by a similar proof to Theorem \ref{thm:spl1}. 
\end{proof}

    \subsubsection{} \label{sec:AIIIa}
    For any $i\in\I$, let $P_i$ be the minimal parabolic subgroup containing $B_k$ associated to the index $i$. Let $\pi_i:G_k/B_k\rightarrow G_k/P_i$ be the projection.  Then for any $\mathcal{O}\in \CV$, the decomposition of $\pi_i^{-1} \circ \pi_i (\CO) = \tilde{\pi}_i (\CO)$ into $K_k$-orbits  depends on the case analysis in \S\ref{sec:Korbits}. The following claim follows from \cite{Wy16}*{Theorem~1.2}. The proof can be obtained merely by translating our languages into clans, which we shall skip. 

(a) {\em Let $\mathcal{O}\in \CV$ and $i \in I$ be such that $s_i \cdot \CO \neq \CO$. 
Then for any $\mathcal{O}'\in \CV$ with $\overline{\mathcal{O}'}\supseteq \mathcal{O}$ and $\overline{\mathcal{O}'}\supseteq s_i\cdot\mathcal{O}$,  we have $\overline{\mathcal{O}'}\supseteq\tilde{\pi}_i  (\CO)$.}

\begin{lem}\label{lem:AIII}
    Let $\CO' \in \CV$ and $\CO \in \CV$ the unique open dense $K_k$-orbit in $\tilde{\pi}_i(\CO')$ for some $ i \in \I$. Let $\widetilde{\mathcal{O}'}$ be the preimage of $\overline{\mathcal{O}'}$ via the projection $G_k \rightarrow G_k / B_k$.  Then the following surjective morphism has connected fibers
    \[
    f: \widetilde{\mathcal{O}'}\times^{B_k} P_{i}/B_k\longrightarrow\overline{\mathcal{O}}.
    \]
\end{lem}

    \begin{proof}
         Recall the closed embedding via the convolution product
       \[
       \widetilde{\mathcal{O}'}\times^B P_{i}/B_k \rightarrow \overline{\mathcal{O}'}\times \overline{\mathcal{O}}, \quad (g_1, g_2 B_k/B_k) \mapsto (g_1B_k/B_k, g_1g_2 B_k/B_k).
       \]
        Let $xB_k / B_k \in \overline{\mathcal{O}}$ for some $x \in G_k$. Then we have 
        \[
        f^{-1} (xB_k / B_k) \cong \tilde{\pi}_i (xB_k / B_k) \cap \overline{\mathcal{O}'}.
        \]
        
       Since the fiber is stable under the action of $K \cap x P_i x^{-1}$, it remains to consider $H$-orbits on  $\tilde{\pi}_i  (xB_k / B_k) \cong P_i / B_k \cong \BP^1_k$, where $H$ denotes the image of $K \cap x P_i x^{-1}$ in $Aut(\BP^1_k) \cong PGL_{2,k}$. By construction, $H$ has an open dense orbit. We have four cases to consider by \cite{RS90}*{\S4}:
       
       Case I: $H^\circ$ is unipotent. In this case, we have two orbits, one open dense orbit and one closed fixed point. The fibre is always connected. 
       
       Case II: $H = PGL_{2,k}$. There is only one orbit. The fibre is connected. 
       
       Case III: $H$ is a torus. There are three orbits, one open dense orbit and two fixed points. However, the fibre can not be the union of the two fixed points, by \S\ref{sec:AIIIa} (a). The fibre is always connected in this case. 
       
       Case IV:  $H$ is the normalizer of a torus. Then $H^\circ$ is a torus. The open dense $H^\circ$-orbit is $H$-stable and $H$ permutes the two fixed points of $H^\circ$.  However, this case can not happen in our setting by \cite{Br01b}*{Corollary~2 \& Page 18}
       
       We now finish the proof. 
    \end{proof}
\begin{remark}
(1) We expect the claim \S\ref{sec:AIIIa} (a)  holds for arbitrary symmetric pairs. 

(2) It follows from the proof of Lemma~\ref{lem:AIII} that the connectedness of the  fibres can be read from Brion's graphs in \cite{Br01b}. This is closely related to multiplicity-free subvarieties.

Brion's graph has only simple edges for the symmetric pair $(GL_{n,k}, GL_{\lceil n/2\rceil,k}\times GL_{\lfloor n/2 \rfloor,k})$, so that Lemma~\ref{lem:AIII} holds for all orbits in this case.
\end{remark}

\subsubsection{}
We now prove the main result of this section. 
\begin{prop}
For any $\mathcal{O}\in \CV$, its closure $\overline{\mathcal{O}}$ is normal if it is Frobenius split. In particular,  $\overline{\mathcal{O}}$ is normal for any $\mathcal{O}\in \CV'$. 
\end{prop}

\begin{proof}

 Let $\mathcal{O}\in \CV$. By  \cite{RS90}*{Theorem 4.6}, there exists a closed orbit $\mathcal{O}_1\in \CV$ and a sequence $i_1, \dots, i_m \in \I$, such that $\tilde{\pi}_{i_m} \circ \cdots \circ \tilde{\pi}_{i_1}(\mathcal{O}_1)=\mathcal{O}$ and $\textrm{dim}(\mathcal{O})=\textrm{dim}(\mathcal{O}_1)+m$. Let $\widetilde{\mathcal{O}_1}\subset G$ be the preimage of $\mathcal{O}_1$ via the natural projection $G_k \rightarrow G_k/B_k$. Then there is a proper surjective morphism:
\[
f:\widetilde{\mathcal{O}_1}\times^{B_k}P_{i_1}\times^{B_k}\cdots\times^{B_k}P_{i_m}/B_k\longrightarrow \overline{\mathcal{O}}.
\]

We show the map $f$ has connected fibres by induction on $m$. 
The base case is by Lemma~\ref{lem:AIII}. By the induction hypothesis, the map 
\[
f_1:\widetilde{\mathcal{O}_1}\times^{B_k}P_{i_1}\times^{B_k}\cdots\times^{B_k}P_{i_{m-1}}\times^{B_k} P_{i_m}/B_k\longrightarrow \widetilde{\mathcal{O}'}\times ^{B_k} P_{i_m}/B_k
\]
has connected fibers. The projection $f_2: \widetilde{\mathcal{O}'}\times ^{B_k} P_{i_m}/B_k \rightarrow \overline{\mathcal{O}}$ also has connected fibres. Since $f_1$ is proper hence closed, we conclude that $f = f_2 \circ f_1$ has connected fibers as well. 

The rest of the proposition follows by \cite{MR87}*{Lemma 1}.
\end{proof}

Thanks to Proposition \ref{prop:KorbitsZ}, any orbit closure over characteristic 0 can be viewed as the geometric generic fibre of the corresponding (flat) scheme over $\BZ[2^{-1}]$. Since geometric normality is a open condition over $Sp\,\BZ[2^{-1}]$ (\cite{EGA}*{Theorem 12.2.4 (iv)}), closures of orbits in $\mathcal{V}'$ are normal over characteristic 0.

\begin{remark}
Similar maps to $f$ were considered by Barbasch-Evens in \cite{BE94}*{6.3.7}. Barbasch-Evens' construction starts from ``distinguished orbits'', instead of closed ones, which guarantees the map is birational. In general, the map defined from closed orbits can be generically finite-to-one. 
\end{remark}

\begin{remark}
In the langauage of \cite{Wy16}, the orbits in $\CV'$ correspond to clans which have at most one sign, and the first $\lfloor n/2\rfloor$ natural numbers are distinct.
\end{remark}

\end{document}
\\